\documentclass[12pt]{amsart}
\usepackage{amscd,amssymb,amsthm,verbatim}
\newcommand{\const}{\operatorname{const.}}

\newcommand{\diam}{\operatorname{diam}}
\newcommand{\Diff}{\operatorname{Diff}}

\newcommand{\dvol}{\operatorname{dvol}}
\newcommand{\End}{\operatorname{End}}

\newcommand{\Id}{\operatorname{Id}}

\newcommand{\Isom}{\operatorname{Isom}}

\newcommand{\N}{{\mathbb N}}

\newcommand{\R}{{\mathbb R}}

\newcommand{\Ric}{\operatorname{Ric}}

\newcommand{\Rm}{\operatorname{Rm}}

\newcommand{\SL}{\operatorname{SL}}
\newcommand{\SO}{\operatorname{SO}}

\newcommand{\supp}{\operatorname{supp}}

\newcommand{\Tr}{\operatorname{Tr}}

\newcommand{\vol}{\operatorname{vol}}
\newcommand{\Z}{{\mathbb Z}}

\numberwithin{equation}{section}
\theoremstyle{plain}
\newtheorem{definition}[equation]{Definition}
\newtheorem{assumption}[equation]{Assumption}
\newtheorem{lemma}[equation]{Lemma}
\newtheorem{theorem}[equation]{Theorem}
\newtheorem{proposition}[equation]{Proposition}
\newtheorem{corollary}[equation]{Corollary}

\theoremstyle{remark}

\newtheorem{remark}[equation]{Remark}
\newtheorem{example}[equation]{Example}
\usepackage{graphicx}
\input{amssym.def}
\input{amssym}
\input{epsf}
\usepackage{a4wide}

\begin{document}

\title[Collapsing in the Einstein flow]
{Collapsing in the Einstein flow}

\author{John Lott}
\address{Department of Mathematics\\
University of California, Berkeley\\
Berkeley, CA  94720-3840\\
USA} \email{lott@berkeley.edu}

\thanks{Research partially supported by NSF grant
DMS-1510192}
\date{April 17, 2018}

\begin{abstract}
We consider expanding vacuum spacetimes with a CMC foliation by
compact spacelike hypersurfaces.
Under scale invariant {\it a priori} geometric bounds (type-III),
we show that there are 
arbitrarily large future time intervals that are modelled by a flat 
spacetime or a Kasner spacetime. We give related results for a class
of expanding vacuum spacetimes that do not satisfy the {\it a priori} bounds
(type-II).
\end{abstract}

\maketitle

\tableofcontents

\section{Introduction}

This paper is about the future behavior of vacuum Einstein
solutions.  We make the following assumptions :
\begin{enumerate}
\item We have a globally hyperbolic vacuum 
spacetime $M$ with a single boundary component which 
is an initial spacelike hypersurface.
\item There is a foliation of $M$ by compact
$n$-dimensional  constant mean
curvature (CMC) spacelike hypersurfaces $X$.
\item The mean curvatures $H$ of the hypersurfaces are future-increasing 
and range over an interval
$[H_0, 0)$, where $H_0 < 0$. 
\end{enumerate}

To say a word about the assumptions, 
there are examples of spatially compact globally hyperbolic vacuum
spacetimes without a CMC hypersurface 
\cite{Chrusciel-Isenberg-Pollack (2005)}. Nevertheless,  
having a CMC foliation is generally considered to be a fair assumption
and it allows one to define a canonical time
function, the Hubble time $t = - \frac{n}{H}$. The expanding nature of the
spacetime is the statement that $H < 0$.

The Lorentzian metric on $M$ can be written as $g = - L^2 dt^2 + h(t)$, where
$h(t)$ is a Riemannian metric on the compact manifold $X$.
It is well known that 
the vanishing of the Ricci curvature of $g$ can be written as a
flow ${\mathcal E}$, parametrized by time $t$,
on triples $(h,K,L)$ that satisfy certain constraint equations.
Here $K$ is a covariant $2$-tensor field on $X$ that becomes the
second fundamental form of the time slices. We call ${\mathcal E}$ an
Einstein flow.

Fischer and Moncrief found that the normalized spatial volume
$(-H)^n \vol(X, h(t))$ is monotonically
nonincreasing, and constant exactly when $g$ describes a Lorentzian cone over
a Riemannian Einstein manifold with Einstein constant $-(n-1)$, i.e.
$L=1$ and $h(t) = t^2 h_{Ein}$ 
\cite{Fischer-Moncrief (2002)}.  (A closely related monotonic quantity
was found by Anderson \cite{Anderson (2001)}.) They
suggested that the monotonicity of their normalized volume should
imply that for a large part of $X$, in the sense of relative volume,
its future development 
is modelled on a Lorentzian cone of the type mentioned above.

\subsection{Results}

In this subsection 
we state the main results of this paper in a somewhat loose form,
with references to the precise statements in the body of the paper.

\subsubsection{Integral result}

We must first introduce the rescaling of an
expanding CMC
Einstein flow ${\mathcal E}$.  Given $s \ge 1$, put $h_s(u) = s^{-2} h(su)$,
$K_s(u) = s^{-1} K(su)$ and $L_s(u) = L(su)$. Then
${\mathcal E}_s = (h_s, K_s, L_s)$ is also an expanding CMC Einstein flow. Given 
$\Lambda > 1$, the time interval $[\Lambda^{-1}, \Lambda]$ for 
${\mathcal E}_s$ corresponds to the time interval
$[s\Lambda^{-1}, s\Lambda]$ for ${\mathcal E}$. Thus we can
analyze the future behavior of ${\mathcal E}$ by understanding
the limit as $s \rightarrow \infty$ of ${\mathcal E}_s$, on a fixed time
interval. It is not hard to see that ${\mathcal E}$ is
scale invariant if and only if it describes a Lorentzian cone of the
type mentioned above.

There is a pointwise monotonicity statement : $(-H)^n \dvol(X,h(t))$
is monotonically nonincreasing.  Put $\dvol_\infty = 
\lim_{t \rightarrow \infty} (-H)^n \dvol(X,h(t))$. 

As a consequence of the monotonicity of normalized volume, one obtains an
integral result about future evolution.

\begin{theorem} \label{thm1}
(Propositions \ref{1.37} and \ref{1.42})
After rescaling, the future evolution becomes increasingly
scale invariant, in an integral sense with respect to 
$\dvol_\infty$.
\end{theorem}

Theorem \ref{thm1} can be 
considered to say that the Fischer-Moncrief suggestion
is true in an integral sense.  If $\dvol_\infty = 0$ then
Theorem \ref{thm1} is true but vacuous.

To proceed, we divide the expanding CMC Einstein flows into
two types.  Using the time vector field, one can
make sense of the norm $|\Rm|_T$ of the Lorentzian curvature tensor
(\ref{1.47}).
Borrowing terminology from Ricci flow, we say that 
an expanding CMC Einstein flow is type-III if
$|\Rm|_T = O(t^{-2})$, and type-IIb otherwise.

As model spaces, we list the simply connected spatially homogeneous
solutions with a future-directed expanding homothetic Killing vector field
(${\mathcal L}_V g = 2 g$) and a spatially compact quotient, in the case $n=3$
\cite[p. 187]{Ellis-Wainwright (1997)}.  They are all type-III.
\begin{enumerate}
\item The Milne spacetime.  This is the interior of a forward light cone
in the Minkowski space $\R^{1,3}$.
\item The Bianchi-III flat spacetime.  This is the product of $\R$ with
the interior of a forward light cone in the Minkowski space $\R^{1,2}$.
\item The Taub-flat spacetime.  This is the product of $\R^2$ with the
interior of a forward light cone in the Minkowski space $\R^{1,1}$.
\item The Kasner spacetimes on $(0, \infty) \times \R^3$, with metric
$g = - du^2 + u^{2p_1} dx^2 + u^{2p_2} dy^2 + u^{2p_3} dz^2$.
Here $p_1 + p_2 + p_3 = p_1^2 + p_2^2 + p_3^2 = 1$.  
\end{enumerate} 
The Taub-flat spacetime is also the Kasner spacetime with $(p_1,p_2,p_3)=
(1,0,0)$, but we list it separately. Only the Milne spacetime is
scale invariant in our earlier sense.

\subsubsection{Type-III Einstein flows}

In this subsubsection we assume that the Einstein flow ${\mathcal E}$
is type-III. Then
we can improve Theorem \ref{thm1} to a pointwise statement.

\begin{theorem} \label{thm2}
(Proposition \ref{2.5} and Remark \ref{2.6.5})
Given $x \in X$, if $\dvol_\infty(x) \neq 0$ then
after rescaling, the future evolution near $x$ becomes increasingly
like that of a Lorentzian cone over
a Riemannian Einstein space with Einstein constant $-(n-1)$.
\end{theorem}

The Riemannian Einstein space in the preceding theorem is of a
generalized type, as discussed below.

We say that the Einstein flow is noncollapsing if $\dvol_\infty  \neq 0$.
Anderson initiated the study of noncollapsing type-III Einstein flows
using rescaling, monotonicity and compactness results
\cite{Anderson (2001)}.  We
recapitulate these results in Subsection \ref{subsect1.3}. Most of this paper
is concerned with the collapsing case, i.e. when $\dvol_\infty$ vanishes.
The main point of the paper is to make use of results on 
Einstein flows with continuous spatial symmetries.
Einstein flows with symmetries have long been studied in general
relativity as toy models.  As in \cite{Lott (2010)}, our viewpoint is
rather that information about Einstein flows with symmetries can give
information about all Einstein flows that satisfy an {\it a priori}
curvature bound.

The appearance of continuous symmetries in collapsing Riemannian manifolds,
under uniform sectional curvature bounds, is known from work of Margulis,
Gromov, Cheeger, Fukaya and many others.  In this paper we promote this
to type-III Einstein flows, in analogy to earlier work on type-III
Ricci flows \cite{Lott (2010)}. To describe the idea, consider first
a manifold $X$ with a sequence of Riemannian metrics that collapse with
uniformly bounded curvature.
To analyze the geometry near a point
$x \in X$, one approach is to pass to finite covers of $X$, 
if possible, that have a noncollapsed pointed limit.  This unwrapping approach
was used for the Einstein flow by Anderson in \cite{Anderson (2001)}.
Another approach is to pull back metrics to a ball in $T_xX$, using the
exponential map, and pass to a noncollapsed pointed limit.  
Both of these methods work well for local regularity
issues in the Einstein flow.  However, to obtain nonlocal results, for example
to apply monotonicity formulas, it is necessary to have a global
approach.  For example, in the tangent space approach, it is necessary to 
glue together the various noncollapsed limits on the balls in the tangent 
spaces $T_x X$, with their
local symmetries, as one varies $x$.
A convenient language to do this is that of \'etale groupoids, as used
for the Ricci flow in \cite{Lott (2007)} and \cite{Lott (2010)}.
A collapsing sequence of pointed $n$-dimensional
Riemannian manifolds, with uniformly bounded
curvature, has a subsequential
limit that is a pointed $n$-dimensional Riemannian groupoid.
The Riemannian groupoid is an object with local symmetries; its orbit
space is the Gromov-Hausdorff limit of the collapsing Riemannian manifolds,
and it also retains information about the limit of their universal covers.

\begin{theorem} \label{thm3} (Corollary \ref{1.55})
Given a type-III Einstein flow ${\mathcal E}$ on a pointed 
$n$-dimensional manifold,
if $\{t_i\}_{i=1}^\infty$ is a sequence tending to infinity then
after passing to a subsequence, the rescalings ${\mathcal E}_{t_i}$ converge to
a type-III Einstein flow ${\mathcal E}^\infty$ on a pointed 
$n$-dimensional \'etale groupoid.
\end{theorem}

The convergence in Theorem \ref{thm3} is in the weak $W^{2,p}$-topology
for any $p < \infty$, and in the $C^{1,\alpha}$-topology for any
$\alpha \in (0,1)$.

In the rest of this subsubsection, we assume that $n = 3$ and $X$ is 
aspherical, i.e. has contractible universal cover.
Then the limit
Einstein flow ${\mathcal E}^\infty$ is of the type that occurs in
dimensional reduction.  It lives on an orbifold $X^\infty$,
which is the orbit space of ${\mathcal E}^\infty$. When $X^\infty$ is not
a point, the fields on $X^\infty$ consist of a quintuple
$(h^\infty,K^\infty,L^\infty,G^\infty,A^\infty)$ 
where $h^\infty$ is a Riemannian metric, $K^\infty$ is a covariant
$2$-tensor field, $L^\infty$ is a function, $G^\infty$ 
is locally an $N \times N$
positive definite matrix and $A^\infty$ is locally an $\R^N$-valued $1$-form.
Here $N = 3 -\dim(X^\infty)$. 

Thus we are reduced to understanding the future behavior of 
${\mathcal E}^\infty$. To do so, we again use monotonic quantities.
We need to assume that there is some $D < \infty$ so that the original
flow ${\mathcal E}$ has $\diam(X, h(t)) \le D t$. This ensures that
$X^\infty$ is compact.

We now list the results about type-III Einstein flows
in order of increasing dimension of 
$X^\infty$. All of the results have consequences for the pointed future
behavior of the lift of the Einstein flow ${\mathcal E}$ to 
the universal cover $\widetilde{X}$, that do not invoke
groupoids (Corollaries \ref{2.13}, \ref{2.29} and \ref{2.48}).
We say that an Einstein flow on an \'etale groupoid is of Kasner type
if it is locally isometric to a Kasner solution, and similarly
for the other model solutions. We first consider the case when the
orbit space $X^\infty$ is a point.

\begin{theorem} \label{thm4} (Corollary \ref{2.12})
Suppose that the original Einstein flow ${\mathcal E}$ is such that
$\liminf_{t \rightarrow \infty} t^{-1} \diam(X, h(t)) = 0$. Then
there is a sequence $\{t_i\}_{i=1}^\infty$ going to infinity
such that the rescaled solutions ${\mathcal E}_{t_i}$ approach
an Einstein flow of Kasner type.
\end{theorem}

We now assume that we are not in the situation covered by
Theorem \ref{thm4}, and
consider the case when the orbit space $X^\infty$ is one dimensional.
To analyze the future behavior of the limit flow
${\mathcal E}^\infty$, we use monotonic quantities from Appendix
\ref{sectA}. To do so, we need to make an assumption about the existence of an
equiareal foliation.

\begin{theorem} \label{thm5} (Proposition \ref{2.25})
Suppose that any limit Einstein flow has an orbit space of
positive dimension, and there is a limit Einstein 
flow ${\mathcal E}^\infty$ whose
orbit space is one dimensional.
Suppose that there is a time function $\widehat{u}$ for the
limit flow ${\mathcal E}^\infty$ that is comparable to the time
function $u$ for ${\mathcal E}^\infty$, with
the property that $\det(G)$ is constant on level sets of $\widehat{u}$
(Assumption \ref{2.23}).
Then there is a sequence $\{t_i\}_{i=1}^\infty$ going to infinity
such that the rescaled solutions ${\mathcal E}_{t_i}$ approach
an Einstein flow of Taub-flat type.
\end{theorem}

Next, we assume that we are not in the situations covered by
Theorems \ref{thm4} and \ref{thm5}, and
consider the case when the orbit space $X^\infty$ is 
two dimensional.
To analyze the future behavior of the limit flow
${\mathcal E}^\infty$, we again use monotonic quantities from Appendix
\ref{sectA}. We now need to make an assumption about the existence of
a CMC foliation on a conformally related three-dimensional Lorentzian
metric.

\begin{theorem} \label{thm6} (Proposition \ref{2.44})
Suppose that any limit Einstein flow has an orbit space of
dimension at least two, and there is a limit 
Einstein flow ${\mathcal E}^\infty$ whose
orbit space is two dimensional.
Suppose that there is a time function $\widehat{u}$ for the
limit flow ${\mathcal E}^\infty$ that is comparable to the time
function $u$ for ${\mathcal E}^\infty$, so that the
level sets of $\widehat{u}$
have constant mean curvature for the conformally modified
Lorentzian metric $\widehat{g}$ of (\ref{2.34})
(Assumption \ref{2.42}).
Then there is a sequence $\{t_i\}_{i=1}^\infty$ going to infinity
such that the rescaled solutions ${\mathcal E}_{t_i}$ approach
an Einstein flow of Bianchi-III flat type.
\end{theorem}

Finally, if a limit flow has an orbit space of dimension three
then the rescalings of ${\mathcal E}$ approach a spatially compact quotient of
the Milne spacetime (Proposition \ref{1.57}).

We made some additional assumptions in Theorems \ref{thm5} and \ref{thm6}.
There is some flexibility in the precise assumptions to make.
Under weaker assumptions, one can
prove integral convergence results (Propositions \ref{2.20}, \ref{2.38} and
\ref{2.40}). We need some assumptions to apply the monotonicity results
of Appendix \ref{sectA}, which are an ingredient in our description of the
future behavior of ${\mathcal E}^\infty$. Any other way to describe
the future behavior would also work.

\subsubsection{Type-II Einstein flows}

The type-III condition is generally not stable under perturbation
\cite{Ringstrom (2006),Ringstrom (2015)}. Hence it is relevant to obtain
information about expanding CMC Einstein flows that are not type-III.
Following Ricci flow terminology, we call them type-IIb Einstein flows.
Given such an Einstein flow ${\mathcal E}$ and
a time $\widehat{t}$, let $x_{\widehat{t}}$ be a point on the 
time-$\widehat{t}$ slice where
$|\Rm|_T$ is maximized. One can rescale the Einstein flow by
$|\Rm|_T(x_{\widehat{t}},\widehat{t})$ and shift the time parameter so
that the new flow has $|\Rm|_T$ maximized by one on the time-$0$ slice.
With an appropriate choice of parameters
$\{\widehat{t}_i\}_{i=1}^\infty$ tending to infinity, these
pointed rescaled flows converge to 
an Einstein flow ${\mathcal E}^\infty$.
It exists for all times $u \in \R$, possibly on an \'etale groupoid.

\begin{theorem} \label{thm7} (Corollary \ref{3.6})
When $n=3$, if the type-IIb Einstein flow ${\mathcal E}$ 
has its second fundamental form $K$ controlled by the mean curvature $H$, 
then ${\mathcal E}^\infty$ is a static flat Einstein flow.
\end{theorem}

Theorem \ref{thm7} applies to the locally homogeneous
examples in \cite{Ringstrom (2006)}.
The theorem may sound paradoxical, because the
rescaled flows have $|\Rm|_T$ equal to one at their basepoints,
whereas the limit flow is flat.  The point is that the metrics converge
in the {\em weak} $W^{2,p}$-topology.  This is not enough to give
pointwise convergence of the curvature norm, even in the locally
homogeneous case.  The interpretation is that
the type-IIb Einstein solution has
increasing fluctuations of the curvature tensor, at least near points of
maximal curvature, that average it
out to zero; c.f. Corollary \ref{3.7}.  
We do however have convergence to the flat metric
in the $C^{1,\alpha}$-topology for any $\alpha \in (0,1)$.

\subsection{Comparison with Ricci flow}

One can compare 
expanding CMC Einstein flows, on compact three dimensional manifolds,
to immortal Ricci flows on compact three dimensional manifolds. 
(A Ricci flow is immortal if it exists for $t \in [0, \infty)$.) There are
some common features.
\begin{enumerate}
\item There is a natural rescaling, and hence notions
of type-III and type-IIb solutions.
\item There is a notion of a self-similar solution.
For Ricci flow, this is a Ricci soliton.
For Einstein flow, this is a Lorentzian metric with a timelike
homothetic vector field. 
\item There is a classification of homogeneous self-similar solutions
for the contractible Thurston geometries.
The geometries $\R^3$, $H^3$, $H^2 \times \R$, $Nil$ and $Sol$ admit
self-similar Ricci flow solutions.
The geometries $\R^3$, $H^3$ and $H^2 \times \R$ admit 
self-similar Einstein flow solutions.
\item The normalized volume form is nonincreasing.
\item Type-III Ricci flows with a scale invariant {\it a priori} 
diameter bound become increasingly homogeneous \cite{Lott (2010)}.
The same is true for type-III Einstein flows with a 
scale invariant {\it a priori} diameter bound, at least to the
extent proven in this paper.
\end{enumerate}

On the other hand, there are important differences.

\begin{enumerate}
\item As a weakly parabolic flow, the Ricci flow is smoothing (in the
right coordinates), as seen by Shi's local derivative estimates. In
particular, this allows one to take smooth limits.  On the other hand,
when taking limits of Einstein flows, one cannot expect the limits
to be much better than $W^{2,p}$-regular.
\item An immortal three dimensional Ricci flow 
is always type-III \cite{Bamler (2014)}.
Expanding CMC Einstein flows need not be type-III.
\item Given a Thurston type, if there is a homogeneous expanding
Ricci soliton with that geometry then it is unique. The analogous statement
is not true for Einstein flows, as the Kasner solutions all have
Thurston type $\R^3$.
\item Considering immortal homogeneous Ricci flows, 
there is a single transmutation: under the Ricci flow, a homogeneous
$\widetilde{\SL(2, \R)}$ geometry has a rescaling 
limit with $H^2 \times \R$ geometry
\cite{Lott (2007)}.
On the other hand, there are three transmutations for type-III homogeneous 
Einstein flows: 
a homogeneous $\widetilde{\SL(2, \R)}$ geometry has a rescaling limit with 
$H^2 \times \R$ geometry, and a homogeneous $Nil$ or $Sol$ geometry has
a rescaling limit with $\R^3$ geometry.
\end{enumerate}

On a technical level, in \cite{Lott (2010)} we showed that any
type-III Ricci flow, with a scale invariant {\it a priori} diameter bound,
becomes increasingly homogeneous as time increases.  In the present paper
we only show that there are large future time intervals on which the Einstein
flow becomes increasingly homogeneous.  The reason for the stronger
conclusion in \cite{Lott (2010)} is that we had unconditional results
for the long-time behavior of the limit Ricci flows, and hence could apply
contradiction arguments to get uniform statements about
the long-time behavior of the
original Ricci flow.  In the present paper, Assumptions
\ref{2.23} and \ref{2.42} are needed in order
to characterize the future behavior of the
limit Einstein flows. Because of this, we cannot apply contradiction
arguments to get uniform statements about the future behavior of the
original Einstein flow.

\subsection{Structure of the paper}

Section \ref{sect1} is about noncollapsed expanding CMC 
Einstein flows, first without
any {\it a priori} curvature assumptions and then with a type-III
curvature assumption. Section \ref{sect2} concerns collapsing
type-III Einstein flows.  Section \ref{sect3} is about type-IIb Einstein
flows.  More detailed descriptions are at the beginnings of the sections.

Appendix \ref{sectA} has monotonicity formulas for expanding CMC
$n$-dimensional Einstein flows with a local $\R^N$-symmetry.  
When $n=3$, 
the monotonic quantities largely reduce to those considered in 
\cite{Berger-Chrusciel-Isenberg-Moncrief (1997)}, 
\cite{Choquet-Bruhat (2004)} and \cite{Choquet-Bruhat-Moncrief (2001)}.
We work in the more general setting partly because, in our opinion, the
derivations become clearer and simpler there. 

I thank Mike Anderson and Jim Isenberg for helpful discussions.
I also thank Mike for comments on an earlier version of this paper.

\subsection{Conventions}

Convergence in $W^{k,p}$ will mean convergence for all $p < \infty$.
Convergence in $C^{k,\alpha}$ will mean convergence for all
$\alpha \in (0,1)$. We will use the Einstein summation convention freely.

\section{Noncollapsed Einstein flows} \label{sect1}

In this section we give results about Einstein flows with a 
scale invariant lower volume bound.
Subsection \ref{subsect1.1} gives the definitions of Einstein flow,
CMC Einstein flow and expanding CMC Einstein flow. We then recall the
monotonicity of normalized volume from \cite{Fischer-Moncrief (2002)}.

In Subsection \ref{subsect1.2} we consider expanding CMC Einstein flows
with compact spacelike hypersurfaces, but no {\it a priori} curvature bounds.
We show that in an integral sense, relative to the limiting normalized
volume form, for large time
the rescaled flow is asymptotically scale invariant.  

Subsection \ref{subsect1.3} is about long-time results for
noncollapsed type-III expanding CMC Einstein flows, due largely to
Anderson \cite{Anderson (2001)}. We give 
relevant notions of convergence
of a sequence of Einstein flows. We define the type-III condition and show
that with a lower volume bound and an upper diameter bound, one gets
convergence (after rescaling) to the space of Lorentzian cones over 
Riemannian Einstein manifolds with Einstein constant $-(n-1)$. 
The rest of the subsection is devoted to what one can say without
the upper diameter bound.

More detailed descriptions are at the beginnings of the subsections.

\subsection{Volume monotonicity} \label{subsect1.1}

\begin{definition} \label{1.1}
Let $I$ be an interval in $\R$. 
An Einstein flow
${\mathcal E}$ on an $n$-dimensional manifold $X$ is given by a
family of nonnegative functions $\{L(t)\}_{t \in I}$ on $X$,
a family of Riemannian metrics $\{h(t)\}_{t \in I}$ on $X$, and a family
of symmetric covariant $2$-tensor fields $\{K(t)\}_{t \in I}$ on $X$,
so that if
$H = h^{ij} K_{ij}$ and $K^0 = K - \frac{H}{n} h$ then 
the constraint equations
\begin{equation} \label{1.2}
R - |K^0|^2 + \left( 1 - \frac{1}{n} \right) H^2 =  0
\end{equation}
and
\begin{equation} \label{1.3}
\nabla_i K^i_{\: \: j} - \nabla_j H =  0,
\end{equation}
are satisfied, along with the evolution equations
\begin{equation} \label{1.4}
\frac{\partial h_{ij}}{\partial t} = - 2 L K_{ij}
\end{equation}
and
\begin{equation} \label{1.5}
\frac{\partial K_{ij}}{\partial t} =  L H K_{ij} - 2 L
h^{kl} K_{ik} K_{lj} - L_{;ij} + L R_{ij}.
\end{equation}
\end{definition}

For now, we will assume that $X$ is compact and connected, and
that all of the data
is smooth.
At the moment, $L$ is unconstrained; it will be determined by
the elliptic equation (\ref{1.13}) below.
We will generally want $L(t)$ to be positive. 

An Einstein flow gives rise to a Ricci-flat Lorentzian metric
\begin{equation} \label{1.6}
g = - L^2 dt^2 + h(t)
\end{equation}
on $I \times X$, for which the second fundamental form of the
time-$t$ slice is $K(t)$. Conversely, given a 
Lorentzian metric $g$ on a manifold with a proper time function $t$,
we can write it in the
form (\ref{1.6}) by using the flow of $\frac{\nabla t}{|\nabla t|^2}$ to
identify nearby leaves.
Letting $K(t)$ be the second fundamental form of 
the time-$t$ slice, the metric $g$ is Ricci-flat if and only if
$(L,h,K)$ is an Einstein flow.

\begin{definition} \label{1.7}
A CMC Einstein flow is an Einstein flow 
for which $H$ only depends on $t$.
It is expanding if $I = [t_0, \infty)$ 
(or $I = (t_0, \infty)$), $H$ is
monotonically increasing in $t$ and takes all values in
$[H_0, 0)$ for some $H_0 < 0$.
\end{definition}

We digress to briefly discuss scale invariant expanding CMC
Einstein flows.
We say that this is the case if $I = (0, \infty)$ and
\begin{equation} \label{1.8}
L = 1, \: \: \: h(ct) = c^2 h(t)
\end{equation}
for all $c > 0$.
Then from (\ref{1.4}),
\begin{equation} \label{1.9}
K_{ij} = - \: t h(1)_{ij} = - \: \frac{1}{t} h_{ij}.
\end{equation}

\begin{lemma} \label{1.10}
Equation (\ref{1.8}) is equivalent to 
\begin{equation} \label{1.11}
L=1, \: \: \:
H = - \: \frac{n}{t}, \: \: \: K^0 = 0.
\end{equation}
In this case,
equations (\ref{1.2})-(\ref{1.5}) are satisfied if and only if
${\mathcal E}$ is a Lorentzian cone over a Riemanniann Einstein manifold
with Einstein constant $-(n-1)$, i.e.
\begin{equation} \label{1.12}
g = - dt^2 + t^2 h_{Ein},
\end{equation}
where $h_{Ein}$ is a Einstein metric on $X$ with Einstein constant 
$- (n-1)$.
\end{lemma} 
\begin{proof}
The equivalence of (\ref{1.8}) and (\ref{1.11}) is straightforward.
If (\ref{1.12}) holds then it is easy to see that 
equations (\ref{1.2})-(\ref{1.5}) are
satisfied.  Conversely, if (\ref{1.2})-(\ref{1.5}) are satisfied then 
(\ref{1.11}) implies that  $R_{ij} = \: - (n-1) h(1)_{ij}
= \: - \frac{n-1}{t^2} h_{ij}$. 
\end{proof}

There is a more general notion of self-similarity for a vacuum Einstein
solution, namely having a future-directed homothetic Killing vector field $V$.
This means, in the expanding case, that ${\mathcal L}_V g = 2g$.
If there is a compact spacelike hypersurface $X$ of constant mean curvature 
then $g$ must be a Lorentzian cone over a Riemannian Einstein manifold
with Einstein constant $-(n-1)$; see 
\cite{Eardley-Isenberg-Marsden-Moncrief (1986)} for the case $n=3$. 
As mentioned in the introduction, 
if $X$ is noncompact then there are other possibilities.

Returning to general expanding CMC Einstein flows,
equation (\ref{1.5}) gives
\begin{align} \label{1.13}
\frac{\partial H}{\partial t} = & - \triangle_h L + LH^2
+ LR \\
= & - \triangle_h L + L |K^0|^2 + \frac{1}{n} LH^2. \notag
\end{align}
The maximum principle gives
\begin{equation} \label{1.14}
\frac{1}{\sup_X |K(t)|^2} \frac{\partial H}{\partial t} \le L(t) 
\le \frac{n}{H^2} \frac{\partial H}{\partial t}.
\end{equation}

We note in passing that if $n>1$ then (\ref{1.2}) gives a formula
for the normalized volume, as
\begin{equation*}
(-H)^n \vol(X, h(t)) =
\frac{n}{n-1} (-H)^{n-2} \int_X \left( - R^h + |K^0|^2 \right)
\dvol(X, h(t)). 
\end{equation*}

\begin{proposition} \label{1.15} \cite{Fischer-Moncrief (2002)}
Let ${\mathcal E}$ be an expanding CMC Einstein flow.
The quantity $(-H)^n \vol(X,h(t))$ is monotonically nonincreasing
in $t$. It is constant in $t$ if and only if, taking
$t = - \: \frac{n}{H}$, the Einstein flow ${\mathcal E}$ is
a Lorentzian cone over a Riemannian Einstein manifold with
Einstein constant $-(n-1)$.
\end{proposition}
\begin{proof}
As in \cite{Fischer-Moncrief (2002)},
using (\ref{1.4}) we have the pointwise identity
\begin{equation} \label{1.16}
\frac{\partial}{\partial t} \left( (-H)^n \dvol(X, h) \right) =
(-H)^{n+1} \left( L -  \frac{n}{H^2} \frac{\partial H}{\partial t}
\right) \: \dvol(X,h).
\end{equation}
From (\ref{1.14}),
it follows that $(-H)^n \dvol(X,h(t))$ is pointwise 
monotonically nonincreasing
in $t$, and hence $(-H)^n \vol(X,h(t))$ is monotonically nonincreasing in $t$. 
Alternatively, applying (\ref{1.13}) to (\ref{1.16}) gives
\begin{equation} \label{1.17}
\frac{d}{dt} \left( (-H)^n \vol(X, h) \right) =
- \: n (-H)^{n-1} \int_X |K^0|^2 L \: \dvol(X,h).
\end{equation}
 If it is constant in $t$ then $K^0 = 0$. Taking
$t = - \: \frac{n}{H}$, equation
(\ref{1.13}) gives $L = 1$.
As $K_{ij} = \frac{1}{n} H h_{ij} = - \frac{h_{ij}}{t}$, equation
(\ref{1.4}) gives $h_{ij}(t) = t^2 h_{ij}(1)$.
Equation (\ref{1.5}) gives
$R_{ij} = - \frac{n-1}{t^2} h_{ij}$. The proposition follows. 
\end{proof}

\begin{remark} \label{1.18}
Proposition \ref{1.15} remains valid
if $L$ and $h$ are locally $W^{2,p}$-regular in spacetime, and
$K$ is locally $W^{1,p}$-regular in spacetime.
It is also valid for an expanding CMC Einstein flow with complete finite volume
time slices, provided that $L$, $K$ and the curvature of $h$ are
bounded on compact time intervals.
\end{remark}

\subsection{Expanding CMC Einstein flows without 
{\it a priori} bounds} \label{subsect1.2}

In this subsection we show that in an integral sense, for large time
the rescaled Einstein flow is asymptotically scale invariant.

To motivate the result of the subsection, let us mention some 
properties
of a scale invariant solution in the sense of Lemma \ref{1.10}:
\begin{enumerate} \label{1.18.5}
\item $t^{-n} \dvol(X, h(t))$ is constant in $t$,
\item $t^{- 2} h(t)$ is constant in $t$,
\item $L-1 = 0$,
\item $K^0 = 0$ and
\item $R + \frac{n(n-1)}{t^2} = 0$.
\end{enumerate}

An expanding CMC Einstein flow has a limiting normalized volume measure
$\dvol_\infty$; see
equation (\ref{1.22}) below. 
The results of this subsection will be true but vacuous
if $\dvol_\infty = 0$. Hence the results are only meaningful in the
noncollapsing case.

We will introduce the rescaling of a expanding 
CMC Einstein flow by a parameter $s > 1$, to obtain a new expanding
CMC Einstein flow. Using the monotone quantity
from Subsection \ref{subsect1.1}, we show that on any fixed time interval
$[\Lambda^{-1}, \Lambda]$, the properties in (\ref{1.18.5}) are asymptotically
true for large $s$. More precisely, properties (1) 
and (2) hold asymptotically with respect to the spatial measure
$\dvol_\infty$, while properties (3), (4), (5) hold asymptotically
with respect to the spacetime measure $du \: \dvol_\infty$. 

We essentially show $C^0$-closeness of the rescaled flows to
a scale invariant flow (relative to $\dvol_\infty$) by showing
that properties (1), (2) and (3) hold asymptotically. To consider
a stronger statement, 
Lemma \ref{1.10} says that a scale invariant Einstein flow in the 
sense of (\ref{1.8}) has time slices with Ricci curvature
$- \frac{n-1}{t^2} h(t)$. It is conceivable
that some weak form of this statement holds asymptotically in an
integral sense.  We do show that the corresponding statement about
scalar curvature, i.e. property (5), holds asymptotically.

To begin, taking $t = - \frac{n}{H}$,
from (\ref{1.16}) the measures $\{t^{-n} \dvol(X, h(t)) \}_{t \ge t_0}$
are pointwise nonincreasing in $t$. They are all absolutely
continuous with respect to some arbitrary smooth Riemannian measure on $X$,
and their $L^1$-densities are pointwise nonincreasing. Put
\begin{equation} \label{1.22}
\dvol_\infty = \lim_{t \rightarrow \infty} \frac{\dvol(X, h(t))}{t^n},
\end{equation}
a nonnegative absolutely continuous measure on $X$.

We give a sufficient condition for $\dvol_\infty$ to be nonzero.
From (\ref{1.2}),
\begin{equation} \label{1.23}
t^2 R + n(n-1) = |tK^0|^2. 
\end{equation}
Hence $R \ge - \: \frac{n(n-1)}{t^2}$ and letting $g_X$ range over
all Riemannian metrics on $X$, we have
\begin{align} \label{1.24}
t^{-n} \: \vol(X, h(t)) \ge &
\inf_{g_X} \left\{ t^{-n} \: \vol(X, g_X) \: : \:
R(g_X) \ge - \: \frac{n(n-1)}{t^2} \right\} \\
= &  \inf_{g_X} \{ \vol(X, g_X) \: : \:
R(g_X) \ge - \: n(n-1) \}. \notag
\end{align}
It follows that
\begin{equation} \label{1.25}
\int_X \dvol_\infty \ge \inf \{\vol(X, g_X) \: : \:
R(g_X) \ge - \: n(n-1) \}.
\end{equation}
If $X$ has a nonpositive $\sigma$-invariant then we obtain
\begin{equation} \label{1.26}
\int_X \dvol_\infty \ge \left( - \: \frac{\sigma(X)}{n(n-1)} \right)^{\frac{n}{2}},
\end{equation}
as was recognized in \cite{Fischer-Moncrief (2002)}.
In particular, if $\dim(X) = 3$ and $X$ contains a hyperbolic
piece in its Thurston decomposition then $\sigma(X) < 0$ and hence
$\int_X \dvol_\infty > 0$.

The results that follow in this subsection
will be true but vacuous if $\dvol_\infty$
vanishes.

\begin{lemma} \label{1.31}
We have
\begin{equation} \label{1.32}
\lim_{t \rightarrow \infty} \frac{t^{-n} \dvol(X, h(t))}{\dvol_\infty} = 1
\end{equation}
in $L^1(\supp(\dvol_\infty); \dvol_\infty)$.
\end{lemma}
\begin{proof}
As $t^{-n} \dvol(X, h(t))$ and $\dvol_\infty$ are absolutely continuous on
$X$, the ratio $\frac{t^{-n} \dvol(X, h(t))}{\dvol_\infty}$ is measurable on
$\supp(\dvol_\infty)$. As $\frac{t^{-n} \dvol(X, h(t))}{\dvol_\infty}$ is
monotonically decreasing to $1$ as $t \rightarrow \infty$, the
monotone convergence theorem gives
\begin{equation} \label{1.33}
\lim_{t \rightarrow \infty} \int_X \left| 
\frac{t^{-n} \dvol(X, h(t))}{\dvol_\infty}
- 1 \right| \dvol_\infty = 
\lim_{t \rightarrow \infty} \int_X \left( 
\frac{t^{-n} \dvol(X, h(t))}{\dvol_\infty}
- 1 \right) \dvol_\infty = 0.
\end{equation}
This proves the lemma.
\end{proof}

We now prove some integral inequalities.
From (\ref{1.17}), we have
\begin{align} \label{1.19}
& n \int_{t_0}^\infty (-H)^{n-1} \int_X  |K^0|^2 L \: \dvol(X, h(t)) \: dt = \\
& (-H(t_0))^n \vol(X, h(t_0)) - \lim_{t \rightarrow \infty}
(-H(t))^n \vol(X, h(t)) < \infty. \notag
\end{align}
As
\begin{equation} \label{1.20}
t = - \: \frac{n}{H},
\end{equation}
we obtain
\begin{equation} \label{1.21}
\int_{t_0}^\infty \int_X |tK^0|^2 L \: \frac{\dvol(X, h(t))}{t^n} \: \frac{dt}{t}
< \infty.
\end{equation}
Hence
\begin{equation} \label{1.27}
\int_{t_0}^\infty \int_X |tK^0|^2 L \: \dvol_\infty \: \frac{dt}{t} \le
\int_{t_0}^\infty \int_X |tK^0|^2 L \: \frac{\dvol(X, h(t))}{t^n} \: \frac{dt}{t}
< \infty.
\end{equation}
Using (\ref{1.13}) and (\ref{1.20}), we have
\begin{equation} \label{1.28}
n (1-L) = - t^2 \triangle_h L + |tK^0|^2 L.
\end{equation}
The maximum principle gives $1 - L \ge 0$.
Then
\begin{align} \label{1.29}
\int_{t_0}^\infty \int_X |L-1| \: \dvol_\infty \frac{dt}{t} = &
\int_{t_0}^\infty \int_X (1-L) \: \dvol_\infty \frac{dt}{t} \\
& \le
\int_{t_0}^\infty \int_X (1-L) \: \frac{\dvol(X, h(t))}{t^n} \frac{dt}{t} 
\notag \\ 
= & \frac{1}{n} \int_{t_0}^\infty \int_X |tK^0|^2 \: L
\frac{\dvol(X, h(t))}{t^n} \frac{dt}{t} < \infty. \notag
\end{align}
From (\ref{1.23}) and (\ref{1.27}),
\begin{align} \label{1.30}
\int_{t_0}^\infty \int_X t^2 \left| R + \frac{n(n-1)}{t^2} \right| L \: 
\dvol_\infty \frac{dt}{t} = &
\int_{t_0}^\infty \int_X (t^2 R + n(n-1))L \: \dvol_\infty \frac{dt}{t} \\
\le & \int_{t_0}^\infty \int_X (t^2 R + n(n-1))L \: 
\frac{\dvol(X, h(t))}{t^n} \frac{dt}{t} \notag \\ 
= &  \int_{t_0}^\infty \int_X |tK^0|^2 L \:
\frac{\dvol(X, h(t))}{t^n} \frac{dt}{t} < \infty. \notag
\end{align}

For $s > 1$, the Lorentzian metric $s^{-2} g$ is isometric to
\begin{equation} \label{1.35}
g_s = - L^2(su) du^2 + s^{-2} h(su).
\end{equation}
Hence we put
\begin{align} \label{1.36}
& L_s(u) =  L(su), 
& h_s(u) = s^{-2} h(su), \: \: \: 
& K_{s,ij}(u) = s^{-1} K_{ij}(su),\\ 
& H_s(u) = s H(su),
& K^0_{s,ij}(u) = s^{-1} K^0_{ij}(su), \: \: \: 
& |K^0|^2_{s}(u) = s^2 K_{ij}(su), \notag \\
& R_{s,ij}(u) = R_{ij}(su), 
& R_s(u) = s^2 R(su).
& \notag
\end{align}
The variable $u$ will refer to the time parameter of a rescaled Einstein
flow, or a limit of such.

\begin{proposition} \label{1.37}
Given $\Lambda > 1$, we have
\begin{equation} \label{1.38}
\lim_{s \rightarrow \infty} (L_s-1) =
\lim_{s \rightarrow \infty} |K^0|^2_s L_s =   
\lim_{s \rightarrow \infty} \left| R_s + \frac{n(n-1)}{u^2} \right| L_s = 0
\end{equation}
in $L^1 \left( [\Lambda^{-1}, \Lambda] \times X,
du \dvol_\infty \right)$.
\end{proposition}
\begin{proof}
We prove that $\lim_{s \rightarrow \infty} |K^0|^2_s L_s = 0$. The proofs for
the other statements are similar, using
(\ref{1.29}) and (\ref{1.30}).

Suppose that it is not true that 
$\lim_{s \rightarrow \infty} |K^0|^2_s L_s = 0$
in $L^1 \left( X \times [\Lambda^{-1}, \Lambda],
\dvol_\infty du \right)$. Then there is some $\epsilon > 0$ and a sequence
$\{s_i\}_{i=1}^\infty$ with $\lim_{i \rightarrow \infty} s_i = \infty$
and 
\begin{equation} \label{1.39}
\int_{\Lambda^{-1}}^\Lambda \int_X |K^0|^2_{s_i} L_{s_i} \: \dvol_{\infty} du
\ge \epsilon.
\end{equation}
After passing to a subsequence, we can assume that
the intervals $[s_i\Lambda^{-1}, s_i\Lambda]$ are disjoint.
Now
\begin{align} \label{1.40}
\int_{s_i\Lambda^{-1}}^{s_i\Lambda} \int_X 
|tK^0(t)|^2 L(t) \: \dvol_\infty \: \frac{dt}{t} = &
\int_{\Lambda^{-1}}^{\Lambda} \int_X 
s_i^2 u^2 |K^0(s_i u)|^2 L(s_i u) \: \dvol_\infty \: \frac{du}{u} \\
\ge & \Lambda^{-1} 
\int_{\Lambda^{-1}}^{\Lambda} \int_X 
s_i^2 |K^0(s_i u)|^2 L(s_i u) \: \dvol_\infty \: du \notag \\
\ge & \Lambda^{-1} \epsilon. \notag
\end{align}
This contradicts (\ref{1.27}).
\end{proof}

From Lemma \ref{1.31}, the volume forms of the rescaled metrics
$h_s(u)$ approach $\dvol_\infty$ in an appropriate sense, as $s \rightarrow
\infty$.
We now look at what one can say about the rest of $h_s(u)$.
In the scale invariant setting of Lemma \ref{1.10}, for any
$s >1$, the rescaled
metrics $u^{-2} h_s(u)$ are constant in $u$.  In particular, for any 
$\Lambda > 1$, we have $h_s(1) = \Lambda^{-2} h_s(\Lambda)$.  Without
assuming scale invariance, we would like to compare the nonvolume parts of
$h_s(1)$ and  $h_s(\Lambda)$ as $s \rightarrow \infty$. To do so, 
we look at their pointwise change as an element of a symmetric space.
 
Given $\Lambda > 1$, $s >> 1$ and $x \in X$, there is some
$H_{s,\Lambda}(x) \in \End(T_x)$ such that
\begin{equation} \label{1.41}
h_s(x, \Lambda) = H_{s,\Lambda}(x)^* h_s(x, 1) H_{s,\Lambda}(x).
\end{equation}
It is defined up to left multiplication by $\Isom(T_xX, h_s(x,1))$.
Let ${H}^1_{s,\Lambda}(x)$ be the rescaling of 
$H_{s,\Lambda}(x)$ to have determinant one.  After choosing an orthonormal
basis of $(T_xX, h_s(x,1))$, the endomorphism
${H}^1_{s,\Lambda}(x)$ defines an element of the symmetric
space $\SO(n) \backslash \SL(n)$ of $n \times n$ symmetric matrices
with determinant one. 
Let $I_n \in \SO(n) \backslash \SL(n)$ be the basepoint represented
by the identity matrix. Let $d_{symm}$ be the distance on
$\SO(n) \backslash \SL(n)$, coming from the Riemannian metric
given by $\langle H, H \rangle = \Tr(H^2)$ for a traceless symmetric matrix
$H \in T_{I_n} (\SO(n) \backslash \SL(n))$.

\begin{proposition} \label{1.42}
We have
\begin{equation} \label{1.43}
\lim_{s \rightarrow \infty} d_{symm}(H^1_{s,\Lambda}, I_n) = 0
\end{equation}
in $L^2(X, \dvol_\infty)$.
\end{proposition}
\begin{proof}
Let $M^0$ denote the traceless part of an $n \times n$ matrix $M$, i.e.
$M^0 = M - \frac{1}{n} (\Tr M) I_n$.
From (\ref{1.4}), the length of the curve
$\{ H^1_{s,u}(x) \}_{u=1}^\Lambda$ is
\begin{align} \label{1.44}
& \int_1^\Lambda \sqrt{
\Tr \left( \left( \left( h_s(x,u)^{- \: \frac12} 
\frac{\partial h_s(x,u)}{\partial u} h_s(x,u)^{- \: \frac12} \right)^0 
\right)^2 \right)
}
\: du = \\
& 2 \int_1^\Lambda |K_s^0(x,u)| L_s(x,u) du. \notag
\end{align}
Using the Cauchy-Schwarz inequality and the fact that $L_s \le 1$,
\begin{align} \label{1.45}
d_{symm}^2(H^1_{s,\Lambda}(x), I_n) \le &
4
\int_1^\Lambda |K_s^0|^2(x,u) L_s(x,u) du
\int_1^\Lambda L_s(x,u) du \\
\le & 4 (\Lambda-1) \int_1^\Lambda |K_s^0|^2(x,u) L_s(x,u) du. \notag
\end{align}
The proposition now follows from Proposition \ref{1.37}.
\end{proof}

\begin{remark} \label{1.46}
We cannot conclude from (\ref{1.43}) that there is a 
$\dvol_\infty$-almost everywhere limit as $t \rightarrow
\infty$ of $(\frac{\dvol_\infty}{\dvol_{h(t)}})^{\frac{1}{n}} h(t)$.
The reason is the factor of $(\Lambda - 1)$ in (\ref{1.45}), which
prevents us from taking $\Lambda \rightarrow \infty$.
\end{remark}

\subsection{Noncollapsed type-III Einstein flows} \label{subsect1.3}

This subsection is devoted to noncollapsed expanding CMC Einstein
flows with an {\it a priori} scale invariant curvature bound.
The results of this subsection are largely due to Anderson
\cite{Anderson (2001)}. As we will need some of the results in a
more general setting, we give a self-contained presentation,
modulo some technical results that we quote.  

Subsubsection \ref{subsubsect1.3.1} begins with the notion
of convergence for a sequence of 
CMC Einstein flows. We then give a compactness result for
CMC Einstein flows that uniformly satisfy certain geometric bounds.
We define type-III Einstein flows and obtain a compactness result for
the rescalings of a noncollapsed type-III Einstein flow.

In Subsubsection \ref{subsubsect1.3.2} we assume that the noncollapsed
type-III Einstien flow has
a scale invariant {\it a priori}  diameter bound. 
We show that the rescalings approach the collection of
Lorentzian cones over compact Riemannian Einstein manifolds with
Einstein constant $-(n-1)$.  This is a straightforward generalization
of the $n=3$ results in \cite[Section 3]{Anderson (2001)}.
(We use the Fischer-Moncrief normalized volume functional,
whereas Anderson used a different but closely related monotonic
quantity.)

Subsubsection \ref{subsubsect1.3.3} analyzes noncollapsed type-III
Einstein flows without the {\it a priori} diameter bound.  The result
is that for large time, there is a decomposition of $X$ into a
``thick part'' where the rescaled flow looks like a 
Lorentzian cone over a compact Riemannian Einstein manifold with
Einstein constant $-(n-1)$, and a ``thin'' part that has a $F$-structure
in the sense of Cheeger-Gromov \cite{Cheeger-Gromov (1986)}.
When $n$ is two or three, one can also say that after rescaling,
points in the thin part are volume collapsed.  The $n=3$ result was stated
in \cite[Section 3]{Anderson (2001)}; we add some detail to the
arguments.

Subsubsection \ref{subsubsect1.3.3} is
not needed for the rest of the paper. Stronger conclusions in the
$n=3$ case, under stronger
assumptions (boundedness of Bel-Robinson energies), are in
\cite{Reiris (2010)}.

To begin, we say how we measure the pointwise size of the curvature tensor.
Let ${\mathcal E}$ be an Einstein flow. Let $g$ be the corresponding
Lorentzian metric.
Put $e_0 = T = \frac{1}{L} \frac{\partial}{\partial t}$, 
a unit timelike vector
that is normal to the level sets of $t$.
Let $\{e_i\}_{i=1}^n$ be an orthonormal basis for $e_0^\perp$. Put
\begin{equation} \label{1.47}
|\Rm|_T = \sqrt{\sum_{\alpha, \beta, \gamma, \delta = 0}^n 
R_{\alpha \beta \gamma \delta}^2}.
\end{equation}

\subsubsection{Limits of CMC Einstein flows} \label{subsubsect1.3.1}

Let ${\mathcal E}^\infty = \left( L^\infty, h^{\infty}, K^{\infty}
\right) $ be a CMC Einstein flow on a pointed $n$-manifold
$\left( X^\infty, x^\infty \right)$, 
with complete time slices, defined on a time interval
$I^\infty$. For the moment, $t$ need not be the Hubble time.

\begin{definition} \label{1.48}
The flow ${\mathcal E}^\infty$ is $W^{2,p}$-regular if 
$X^\infty$ is a $W^{3,p}$-manifold, 
$L^\infty$ and $h^\infty$ are locally $W^{2,p}$-regular in space and time, and 
$K^\infty$ is locally $W^{1,p}$-regular in space and time.
\end{definition}
Note that the equations of Definition \ref{1.1} make sense in this
generality.

Let ${\mathcal E}^{(k)} = \{h^{(k)}, K^{(k)}, L^{(k)} \}_{k=1}^\infty$ be 
smooth CMC
Einstein flows on pointed $n$-manifolds 
$\{ \left( X^{(k)}, x^{(k)} \right) \}_{k=1}^\infty$,
defined on time intervals $I^{(k)}$.

\begin{definition} \label{1.49}
We say that $\lim_{k \rightarrow \infty} {\mathcal E}^{(k)} =
{\mathcal E}^\infty$ in the pointed weak $W^{2,p}$-topology
if 
\begin{itemize}
\item Any compact interval $S \subset I^\infty$ is contained in
$I^{(k)}$ for large $k$, and
\item For any compact interval $S \subset I^\infty$ and
any compact $n$-dimensional manifold-with-boundary $W^\infty \subset X^\infty$
containing $x^\infty$, for large $k$
there are pointed time-independent $W^{3,p}$-regular diffeomorphisms
$\phi_{S,W,k} : W^\infty \rightarrow W^{(k)}$ (with
$W^{(k)} \subset X^{(k)}$) so that 
\begin{itemize}
\item $\lim_{k \rightarrow \infty}
\phi_{S,W,k}^* L^{(k)} = L^\infty$ weakly in $W^{2,p}$ on $S \times W^\infty$,
\item $\lim_{k \rightarrow \infty}
\phi_{S,W,k}^* h^{(k)} = h^\infty$ weakly in $W^{2,p}$ on $S \times W^\infty$
and
\item $\lim_{k \rightarrow \infty}
\phi_{S,W,k}^* K^{(k)} = K^\infty$ weakly in $W^{1,p}$ on $S \times W^\infty$.
\end{itemize}
\end{itemize}
\end{definition}

We define pointed (norm) $C^{1,\alpha}$-convergence similarly.

\begin{definition} \label{1.50}
Let ${\mathcal S}$ be a collection of pointed CMC
Einstein flows defined on a time interval $I^\infty$.
 We say that a sequence 
$\{ {\mathcal E}^{(k)} \}_{k=1}^\infty$ approaches ${\mathcal S}$ 
as $k \rightarrow \infty$, in the pointed weak $W^{2,p}$-topology, if
for any subsequence of $\{ {\mathcal E}^{(k)} \}_{k=1}^\infty$, there
is a further subsequence
that converges to an element of ${\mathcal S}$
in the pointed weak $W^{2,p}$-topology.
\end{definition}

\begin{definition} \label{1.51}
Let ${\mathcal S}$ be a collection of pointed CMC 
Einstein flows defined on a time
interval $I^\infty$. We say that a $1$-parameter family
$\{ {\mathcal E}^{(s)} \}_{s \in [s_0, \infty)}$ of pointed CMC Einstein flows
 approaches 
${\mathcal S}$, in the pointed weak $W^{2,p}$-topology, if
for any sequence $\{s_k\}_{k=1}^\infty$ in $[s_0, \infty)$ with
$\lim_{k \rightarrow \infty} s_k = \infty$, there is a subsequence
of the flows $\{ {\mathcal E}^{(s_k)} \}_{k=1}^\infty$
that converges to an element of ${\mathcal S}$
in the pointed weak $W^{2,p}$-topology.
\end{definition}

We define ``approaches ${\mathcal S}$'' in the 
pointed (norm) $C^{1,\alpha}$-topology similarly.
The motivation for these definitions comes from how one can
define convergence to a compact subset of a metric space, just
using the notion of sequential convergence. In our applications,
the relevant set ${\mathcal S}$ of Einstein flows can be taken to be
sequentially compact.

The next result is essentially contained in 
\cite[Proof of Theorem 3.1]{Anderson (2001)}.

\begin{proposition} \label{1.52} 
Let $\left\{ {\mathcal E}^{(k)} \right\}_{k=1}^\infty$ 
be a sequence of CMC
Einstein flows on pointed $n$-dimensional manifolds $(X^{(k)}, x^{(k)})$.
Suppose that each ${\mathcal E}^{(k)}$ is
defined on a time-interval $I^{(k)}$, on which the mean curvature
$H^{(k)}$ is negative and increasing. Suppose that each ${\mathcal E}^{(k)}$
has complete time slices.
Suppose that $I^\infty \subset \R$ is an
interval so that for any compact interval $S \subset I^\infty$,
\begin{itemize}
\item For large
$k$ we have $S \subset I^{(k)}$, and
\item For large $k$, there are uniform upper bounds on
 $\left| H^{(k)} \right|$, 
$\left| \frac{d}{dt} H^{(k)} \right|$, 
$\left|\frac{d^2}{dt^2} H^{(k)} \right|$, 
$\left| \frac{d^3}{dt^3} H^{(k)} \right|$,
$- \: \frac{d}{dt} \frac{1}{H^{(k)}}$ and
$\left| \Rm^{(k)} \right|_T$ 
on $S$.
\end{itemize}

Fix $t_0 \in I^\infty$.
Suppose that there is some $v_0 > 0$ so that for all large $k$, the 
time-$t_0$ unit ball 
satisfies
$\vol \left( B_{h^{(k)}(t_0)}(x^{(k)}, 1)  \right) \ge v_0$.
Then after passing to a subsequence,
there is a limit
$\lim_{k \rightarrow \infty} {\mathcal E}^{(k)} =
{\mathcal E}^\infty$
in the pointed weak $W^{2,p}$-topology and the pointed 
$C^{1,\alpha}$-topology. The limit flow
${\mathcal E}^\infty$ is defined on a pointed $n$-manifold
$(X^\infty, x^\infty)$, and on the time interval $I^\infty$.
Its time slices are complete.

If for each compact interval $S \subset I^\infty$, there is some
$C_S < \infty$ such that 
$\left| K^{(k)} \right|^2 \le C_S
\frac{dH^{(k)}}{dt}$ for all large $k$, on the time interval $S$, 
then the limiting lapse function $L^\infty$ is
positive.
\end{proposition}
\begin{proof}
On any compact interval $S \subset I^\infty$, the bounds on
$H^{(k)}$ and $\left| \Rm^{(k)} \right|_T$ give bounds on
$\left| K^{(k)} \right|$ for large $k$
\cite[Proposition 2.2]{Anderson (2001)}, 
and hence on the curvature of $h^{(k)}$.
From (\ref{1.14}), there is a uniform upper bound on the lapse
functions $L^{(k)}$.
Using (\ref{1.13}) and taking $t$-derivatives of it,
there are $W^{2,p}$-bounds on the $L^{(k)}$'s;
see \cite[p. 551]{Anderson (2001)} and 
\cite[Section 3]{Chen-LeFloch (2009)}. One also has first derivative
bounds on $K$. In all, one obtains $W^{2,p}$-bounds on
$\left\{ {\mathcal E}^{(k)} \right\}_{k=1}^\infty$ over the
time interval $S$; c.f. \cite[Theorem 3.1]{Chen-LeFloch (2009)}

Using the lower volume bound,
after passing to a subsequence of the pointed Riemannian manifolds 
$\{ (X^{(k)}, x^{(k)}, h^{(k)}(t_0)) \}_{k=1}^\infty$
there is a pointed $W^{3,p}$-regular limit manifold 
$(X^\infty, x^\infty)$ with a 
complete $W^{2,p}$-regular limit Riemannian metric $h^\infty(t_0)$. 
Let $W^\infty \subset X^\infty$ be a 
compact $n$-dimensional manifold-with-boundary containing $x^\infty$ and
let $\phi_{W,k} : W^\infty \rightarrow W^{(k)}$ be the comparison
diffeomorphisms inherent in forming $X^\infty$.
Put 
\begin{equation} \label{1.53}
\phi_{S, W, k} = (\Id_S \times \phi_{W,k}) : 
(S \times W^\infty) \rightarrow \left( S \times X^{(k)} \right).
\end{equation}
We have
uniform (in $k$) pointed
$W^{2,p}$-bounds on $\left\{ \phi_{S, W, k}^* {\mathcal E}^{(k)}
\right\}_{k=1}^\infty$, in the sense of Definition \ref{1.48}.
The construction of
${\mathcal E}^\infty$ now follows from
a standard diagonal argument; c.f. \cite[Section 2]{Hamilton (1995)}.

The uniform bounds on $\left| K^{(k)} \right|$ give uniform 
multiplicative bounds
on the distance distortion when going from time $t_0$ to another
time $t \in I^\infty$, from which the completeness of
$(X^\infty, h^\infty(t))$ follows.

If $\left| K^{(k)} \right|^2 \le C_S
\frac{dH^{(k)}}{dt}$ 
then (\ref{1.14}) implies that $L^{(k)} \ge C_S^{-1}$
on the time interval $S$. Hence $L^\infty \ge C_S^{-1}$ on $S$.
\end{proof}

We take $t = - \frac{n}{H}$.

\begin{definition} \label{1.54}
A type-III Einstein flow is an expanding CMC Einstein flow
for which there is some $C < \infty$ so that 
$|\Rm|_T \le C t^{-2}$. 
\end{definition}

Recall the rescaling from (\ref{1.36}).
We write the rescaled Einstein flow as ${\mathcal E}_s$. It is also type-III,
with the same constant $C$.

\begin{corollary} \label{1.55}
Let ${\mathcal E}$ be a type-III
Einstein flow on an $n$-dimensional manifold $X$.
Suppose that it is
defined on a time-interval $[t_0, \infty)$ with 
$t_0 > 0$, and has complete time slices.
Let $\{t_i\}_{i=1}^\infty$ be a sequence in $[t_0, \infty)$ with
$\lim_{i \rightarrow \infty} t_i = \infty$ and let 
$\{x_i\}_{i=1}^\infty$ be a sequence in $X$ with the property that
$\vol \left( B_{h(t_i)}(x_i, t_i)  \right) \ge v_0 t_i^n$ for large
$i$, and some $v_0 > 0$.
Then after passing to a subsequence, which we relabel as
$\{t_i\}_{i=1}^\infty$ and $\{x_i\}_{i=1}^\infty$, there is a limit
$\lim_{i \rightarrow \infty} {\mathcal E}_{t_i} = 
{\mathcal E}^\infty$
in the pointed weak $W^{2,p}$-topology and the pointed 
$C^{1,\alpha}$-topology. The limit flow
${\mathcal E}^\infty$ is
defined on the time interval $(0, \infty)$. Its time slices
$\{(X^\infty, h^\infty(u))\}_{u > 0}$ are complete. Its
lapse function $L^\infty$ is uniformly bounded below by a
positive constant.
\end{corollary}
\begin{proof}
Put $I^{(i)} = [t_0/t_i, \infty)$, $I^\infty = (0, \infty)$ and 
${\mathcal E}^{(i)} = {\mathcal E}^{(i)}_{t_i}$. The existence
of ${\mathcal E}^\infty$ follows from Proposition \ref{1.52}.
From the proof of Proposition \ref{1.52}, on any compact interval
$S \subset I^\infty$ there is a bound
$\left| K^{(i)} \right|^2 \le \const t^{-2}$ that is uniform
in $i$. As $\frac{dH^{(i)}}{dt} = \frac{n}{t^2}$, Proposition \ref{1.52}
implies that $L^\infty > 0$. From its proof, $L^\infty$
is uniformly bounded below by a positive constant.
\end{proof}

\subsubsection{Noncollapsed Type-III Einstein flows with a diameter bound}
\label{subsubsect1.3.2}

In the rest of this section, we will only consider type-III
Einstein flows.
In this subsubsection, we make the following assumption.

\begin{assumption} \label{1.56}
There is some $D < \infty$ so that for all $t$, we have
$\diam(X, h(t)) \le D t$.
\end{assumption}

Let ${\mathcal S}$ be the collection of 
Einstein flows that generate Lorentzian cones
over compact $n$-dimensional
Riemannian Einstein manifolds with Einstein constant $-(n-1)$.
They are defined on the time interval $(0, \infty)$.

\begin{proposition} \cite{Anderson (2001)} \label{1.57}
Suppose that a type-III Einstein flow ${\mathcal E}$ satisfies
Assumption \ref{1.56}, with
$\lim_{t \rightarrow \infty} t^{-n} \vol(X, h(t)) > 0$.
Then as $s \rightarrow \infty$, the rescaled flows
${\mathcal E}_s$ approach ${\mathcal S}$
in the weak $W^{2,p}$-topology and
$C^{1,\alpha}$-topology.
\end{proposition}
\begin{proof}
Let
$\{ s_i \}_{i=1}^\infty$ satisfy $\lim_{i \rightarrow \infty} s_i = \infty$.
Pick arbitrary basepoints $x_i \in X$.
From the upper diameter bound and the positive lower volume bound
on $(X, h_{s_i}(1))$, the Bishop-Gromov inequality gives
a $v_0 > 0$ so that for each $i$, we have
$\vol \left( B_{h_{s_i}(1)}(x_i,1) \right) \ge v_0$.  
Corollary \ref{1.55} now gives a 
subsequential limit Einstein flow ${\mathcal E}^\infty$,
which {\it a priori}
is $W^{2,p}$-regular in the sense of Definition \ref{1.48}.
Because of the diameter bounds, $X^\infty$ is compact.
The monotonicity of $t^{-n} \vol(X, h(t))$ implies
that $u^{-n} \vol(X^\infty, h^\infty(u))$ is constant in $u$.
By Proposition \ref{1.15} and Remark \ref{1.18},
${\mathcal E}^\infty \in {\mathcal S}$.
This proves the proposition.
\end{proof}

The lower volume bound in Proposition \ref{1.57} is guaranteed
when the topology of $X$ is such that
it cannot collapse with bounded curvature and
bounded diameter.  For example, it suffices that $X$ have a
nonzero characteristic number or a nonvanishing simplicial
volume. The conclusion of Proposition \ref{1.57} implies that
$X$ carries an Einstein metric with Einstein constant $-(n-1)$.

In three dimensions, if $X$ admits a hyperbolic metric then it
has positive simplicial volume and 
cannot collapse with bounded curvature and
bounded diameter. An Einstein three-manifold with
Einstein constant $-2$ is hyperbolic.

\begin{corollary} \label{1.58}
Suppose that a three dimensional
type-III Einstein flow ${\mathcal E}$ satisfies
Assumption \ref{1.56}, with
$\lim_{t \rightarrow \infty} t^{-3} \vol(X, h(t)) > 0$.
Let $\widetilde{\mathcal E}$ denote the pullback Einstein flow on
the universal cover $\widetilde{X}$. 
For $s > 0$, choose $\widetilde{x}_s \in \widetilde{X}$.
Then as $s \rightarrow \infty$, the pointed rescaled flows
$\left( \widetilde {\mathcal E}_s, \widetilde{x}_s \right)$ 
approach the flat Milne solution with basepoint
$(1,\widetilde{x}_\infty) \in (0,\infty) \times H^3$,
in the pointed weak $W^{2,p}$-topology and the pointed
$C^{1,\alpha}$-topology.
\end{corollary}

\begin{remark} \label{1.59}
The notion of convergence in Proposition \ref{1.57} is up to
$s$-dependent diffeomorphisms.  For this reason, Proposition \ref{1.57}
does not imply that $\lim_{s \rightarrow \infty} h_s(\cdot)$ exists as
a metric. As Proposition \ref{1.57} does give regions that
are arbitrarily close to Lorentzian cones in ${\mathcal S}$, a stability result
for Lorentzian cones would imply that that 
$\lim_{s \rightarrow \infty} h_s(\cdot)$ exists as
a metric. When $n=3$, the stability result of
\cite{Andersson-Moncrief (2004)} needs, in particular, $H^3$-closeness of
$h_s(1)$ to the hyperbolic metric on $X$. From Proposition \ref{1.57} we only
get weak $W^{2,p}$-closeness or $C^{1,\alpha}$-closeness.  
If we strengthen the type-III assumption
to include $|\nabla \Rm|_T \le C t^{-3}$ and 
$|\nabla \nabla \Rm|_T \le C t^{-4}$ then we will get
$C^{3,\alpha}$-closeness and the stability result will apply.
\end{remark}

\subsubsection{Noncollapsed Type-III Einstein flows without a diameter bound}
\label{subsubsect1.3.3}

In this subsubsection we remove the diameter assumption in
Subsubsection \ref{subsubsect1.3.2}.

Let ${\mathcal S}$ be the collection of 
Einstein flows that generate Lorentzian cones
over finite volume complete connected pointed $n$-dimensional
Riemannian Einstein manifolds with Einstein constant $-(n-1)$.
They are defined on the time interval $(0, \infty)$.

\begin{proposition} \label{1.68}
  Let ${\mathcal E}$ be a type-III Einstein flow with time slices
  diffeomorphic to a compact connected $n$-dimensional manifold $X$.
  Given $v > 0$, there is some $N_v \in \N$ so that for all
  $t \ge t_0$, there is a set $\{x_{t,j}\}_{j=1}^{N^\prime_t}$ in $X$,
  with $N^\prime_t \le N_v$, such that
\begin{itemize}
\item Each $x \in X - \bigcup_{j=1}^{N^\prime_t} B_{h(t)}(x_{t,j}, 2t)$
  has $t^{-n} \vol(B_{h(t)}(x,t)) < v$.
  Here $B_{h(t)}(x, t)$ denotes the ball of radius $t$ around $x$ with
respect to the metric $h(t)$.
\item Let $\{t_i\}_{i=1}^\infty$ be a sequence tending to infinity
  such that $\{x_{t_i,j}\}_{j=1}^{N^\prime_{t_i}}$ is nonempty for each $i$.
Let  $x_{i}$ be a choice of an element of
  $\{x_{t_i,j}\}_{j=1}^{N^\prime_{t_i}}$ for each $i$.
Then as $i \rightarrow \infty$, the pointed rescaled flows
  $({\mathcal E}_{t_i}, x_i)$ approach ${\mathcal S}$ in the
pointed weak $W^{2,p}$-topology and the pointed $C^{1,\alpha}$-topology.
\end{itemize}
\end{proposition}
\begin{proof}
Define the $v$-thick part of
$(X,h(t))$ by
\begin{equation} \label{1.60}
X_{v-thick,h(t)} = \{x \in X \: : \: t^{-n}
\vol \left( B_{h(t)}(x, t) \right) 
\ge v \}.
\end{equation}
If $X_{v-thick,h(t)} \neq \emptyset$,
choose a maximal collection of points $\{x_{t,j}\}$
in $X_{v-thick,h(t)}$ so
that the balls $B_{h(t)} (x_{t,j},  t )$ are disjoint. 
From volume monotonicity, there is some $V_0 < \infty$ so that for all 
$t \ge t_0$,
we have $t^{-n} \vol(X, h(t)) \le V_0$.  Hence the number of points
in the collection is bounded above by $N = \frac{V_0}{v}$. 
The first conclusion of Proposition \ref{1.68} follows.

Given the sequence $\{t_i\}_{i=1}^\infty$ tending to infinity,
after passing to a subsequence,
from Corollary \ref{1.55}
the rescaled pointed flows
$\{({\mathcal E}_{t_i},x_i) \}$ converge to a pointed
$W^{2,p}$-regular Einstein flow
$({\mathcal E}^\infty, x^\infty)$
with complete time slices $(X^\infty, h^\infty(u))$
of constant mean curvature $- \: \frac{n}{u}$,
defined for $u \in (0, \infty)$.

To show that $({\mathcal E}^\infty, x^\infty)$ lies in ${\mathcal S}$,
we claim first that its lapse function $L^\infty$ is identically one.
Suppose not.  From (\ref{1.28}), the lapse $L$ for ${\mathcal E}$
is bounded above by one. Hence $L^\infty$ is also bounded above by one.
Suppose that $ L^\infty \neq 1$. Then there are a compact
set $K^\infty \subset X^\infty$, a time interval $[u_1, u_2] \subset
(0, \infty)$ and a number $\epsilon > 0$ so that
\begin{equation}
  \int_{u_1}^{u_2} \int_{K^\infty} (1 - L^\infty)
  \frac{\dvol(X^\infty, h^\infty(u))}{u^n} \frac{du}{u} > \epsilon.
  \end{equation}
It follows that for large $i$, there are compact subsets $K_i \subset X$
so that
\begin{equation} \label{toomany}
  \int_{t_i u_1}^{t_i u_2} \int_{K_i} (1 - L)
  \frac{\dvol(X, h(t))}{t^n} \frac{dt}{t} > \frac{\epsilon}{2}.
\end{equation}
From the set of intervals $\{ [t_i u_1, t_i u_2]\}_{i=1}^\infty$, we can
extract a subset consisting of an infinite number of disjoint intervals.
Then (\ref{toomany}) gives a contradiction to the fact from (\ref{1.29}) that
\begin{equation}
  \int_{t_0}^{\infty} \int_{X} (1 - L)
  \frac{\dvol(X, h(t))}{t^n} \frac{dt}{t} < \infty.
\end{equation}

Hence $L^\infty = 1$. A similar argument, using (\ref{1.21}), shows that
$|K^{0,\infty}|^2 L^\infty = 0$. Then from Lemma \ref{1.10}, the limit flow
${\mathcal E}^\infty$ lies in ${\mathcal S}$.
This proves the proposition.
\end{proof}

\begin{remark}
From \cite[Theorem 0.1]{Cheeger-Gromov (1990)} and
\cite[Theorem 4.7]{Cheeger-Gromov-Taylor (1982)},
there is some $v_0 = v_0(n,C) > 0$ so that for all large $t$,
the complement of the $v_0$-thick set $X_{v_0-thick,h(t)}$
is part of an open subset of $X$ with an $F$-structure.
(Here $C$ is the constant from Definition \ref{1.54}.)
In particular, if $X$ does not carry an $F$-structure then
$X_{v_0-thick,h(t)}$ is nonempty for all large $t$.
For example, it suffices that $X$ have
a nonzero Euler characteristic or a nonvanishing simplicial volume, 
e.g. if $\dim(X) = 3$
that $X$ has a hyperbolic piece in its Thurston decomposition.
\end{remark}

\subsubsection{Dimensions two and three}
\label{subsubsect1.3.4}

If $n$ is two or three then a finite volume complete Riemannian manifold
with $\Ric = - (n-1) g$ is hyperbolic, i.e.
has constant sectional curvature $-1$.
There is a positive lower bound
on the volumes of such manifolds.

For $n \in \{2,3\}$, let ${\mathcal S}$ now be the collection of 
Einstein flows that generate flat Lorentzian cones
over finite volume complete connected pointed $n$-dimensional
hyperbolic manifolds.

\begin{proposition} \label{1.68prime}
  Let ${\mathcal E}$ be a type-III Einstein flow with time slices
  diffeomorphic to a compact connected $n$-dimensional manifold $X$,
  where $n \in \{2,3\}$. Then
there are a number $N \in \N$ and
a function $\sigma : [t_0, \infty) \rightarrow (0, \infty)$ with
  $\lim_{t \rightarrow \infty} \sigma(t) = 0$ so that for all
  $t \ge t_0$, there is a set $\{x_{t,j}\}_{j=1}^{N_t^\prime}$ in $X$, with
  $N_t^\prime \le N$, such that
  \begin{itemize}
  \item Each $x \in X - \bigcup_{j=1}^{N^\prime_{t}}
    B_{h(t)} \left( x_{t,j}, \frac{t}{\sigma(t)} \right)$
    has $t^{-n} \vol(B_{h(t)}(x,t)) < \sigma(t)$.
  \item
    Let $\{t_i\}_{i=1}^\infty$ be a sequence tending to infinity
  such that $\{x_{t_i,j}\}_{j=1}^{N^\prime_{t_i}}$ is nonempty for each $i$.
  Let  $x_{i}$ be a choice of an element of
  $\{x_{t_i,j}\}_{j=1}^{N^\prime_{t_i}}$ for each $i$.
Then as $i \rightarrow \infty$, the pointed rescaled flows
  $({\mathcal E}_{t_i}, x_i)$ approach ${\mathcal S}$ in the
pointed weak $W^{2,p}$-topology and the pointed $C^{1,\alpha}$-topology.
\end{itemize}
\end{proposition}
\begin{proof}
  Using the Margulis lemma and pointed compactness,
  there is some $v_0 > 0$ so that for all sufficiently small $v > 0$,
  there is some $D(v) < \infty$ with the following property.
  If $(Z,h)$ is a finite volume complete
  connected $n$-dimensional
  hyperbolic manifold, $n \in \{2, 3\}$, then
  $Z_{v-thick,h}$ is contained in the $D(v)$-neighborhood of
  $Z_{2v_0-thick,h}$.

  Consequently, we can carry out the proof of Proposition \ref{1.68}
  while letting $v$ go to zero, but keeping basepoints
  $\{x_{t,j}\}_{j=1}^{N^\prime_t}$ within
  $X_{v_0-thick,h(t)}$. The proposition follows.
\end{proof}

\begin{remark}
  There is a possible redundancy in the choice of basepoints
  $\{x_{t,j}\}_{j=1}^{N^\prime_t}$ in Proposition
  \ref{1.68prime}. In the second conclusion of the proposition,
 if $x_i = x_{t_i,j_i}$ and
  $x_i^\prime = x_{t_i,j^\prime_i}$ are choices of basepoints
  with $d_{h(t_i)} (x_i, x^\prime_i) = O(t_i)$ then
  they will give rise to the same element of ${\mathcal S}$, up to
  a change of basepoint. After eliminating this redundancy,
  we can say that for large $t$, there is a decomposition of
  $\left( X, \frac{h(t)}{t^2} \right)$ into an almost-hyperbolic part
  and a locally collapsing part.
  
  We do not claim that as $t \rightarrow \infty$, the volume of the
  almost-hyperbolic part
  approaches
$\lim_{t \rightarrow \infty}
t^{-n} \vol(X, h(t))$.
That is, it is conceivable that there
is a substantial part of the volume in the locally collapsing part of
$(X, h(t))$. 
\end{remark}

\section{Einstein flows on \'etale groupoids} \label{sect2}

This section contains the results about collapsed type-III Einstein flows.
The convergence results are phrased in terms of Einstein flows on
\'etale groupoids.
We refer to \cite[Section 3]{Lott (2010)} for an overview, aimed at
geometers, of the use of groupoids in collapsing theory.  More details
appear in \cite[Section 5]{Lott (2007)}.

In Subsection \ref{subsect2.1} we define Einstein flows on \'etale groupoids.
We extend the results of
Subsubsection \ref{subsubsect1.3.1} by removing the lower volume
bound assumption.  As an immediate application, we strengthen the
convergence result of Subsection \ref{1.2} when the Einstein flow is
type-III.
Namely, for any point $x \in X$ with $\dvol_\infty(x) \neq 0$,
the rescaled Einstein flows around $x$ converge in the pointed sense
to  Lorentzian cones over (possibly collapsed) Riemannian Einstein metrics
with Einstein
constant $-(n-1)$.

From Subsection \ref{subsect2.1}, after performing rescalings on
a type-III Einstein flow ${\mathcal E}$, we can extract subsequential
limit Einstein flows that live on \'etale groupoids.
In the rest of the section, we restrict to the case $n=3$. We also
assume a scale invariant {\it a priori} diameter bound on 
${\mathcal E}$.
The goal is to show that there are arbitrarily large future time
intervals on which ${\mathcal E}$ is modelled, in a scale invariant way,
by one of a few
homothety-invariant homogeneous Einstein flows, depending on the
Thurston type of $X$.
 
The dimension of the orbit space of the \'etale groupoid is the same as the
dimension of the Gromov-Hausdorff limit of the rescaled time slices.
The case when the dimension is three was covered in
Subsubsection \ref{subsubsect1.3.2}. The cases when the orbit space has
dimension zero, one or two are covering in Subsections
\ref{subsect2.2}, \ref{subsect2.3} and \ref{subsect2.4}, respectively.
More detailed descriptions are at the beginnings of the subsections.  

To summarize the relation between type-III flows (with a scale invariant
diameter bound) and topology, we recall the notion of the Thurston
type of a compact $3$-manifold \cite{Scott (1983)}.
This is a topological notion, i.e. we
do not only consider locally homogeneous metrics.
\begin{enumerate}
\item 
If ${\mathcal E}$ has a rescaling limit flow 
with a zero dimensional orbit
space then $X$ has
Thurston type $\R^3$ or $Nil$.
\item If $X$ has Thurston type $Sol$ then
any rescaling limit flow of 
${\mathcal E}$ has a one dimensional orbit
space. Conversely, 
if a rescaling limit flow of 
${\mathcal E}$ has a one dimensional orbit
space then $X$ has Thurston type $Sol$, $\R^3$ or $Nil$.
\item If $X$ has Thurston type $H^2 \times \R$ or $\widetilde{\SL(2, \R)}$
then any rescaling limit flow of
${\mathcal E}$ has a two dimensional orbit
space. Conversely, 
if a rescaling limit flow of
${\mathcal E}$ has a two dimensional orbit
space then $X$ has Thurston type $H^2 \times \R$, $\widetilde{\SL(2, \R)}$,
$\R^3$ or $Nil$.  One
can speculate that in fact, $X$ must have Thurston type $H^2 \times \R$ or 
$\widetilde{\SL(2, \R)}$; this is true when Proposition \ref{2.44} applies.
\item If $X$ has Thurston type $H^3$
then any rescaling limit flow of
${\mathcal E}$ has a three dimensional orbit
space. Conversely, 
if a rescaling limit flow of
${\mathcal E}$ has a three dimensional orbit
space then $X$ has Thurston type $H^3$.
\end{enumerate}

If $X$ has Thurston type $\R^3$ then the orbit space
of a rescaling limit flow could be zero dimensional 
(as happens for a quotient of a generic Kasner solution) or one dimensional
(as happens for a quotient of the Taub-flat spacetime).
The same is true for Thurston type $Nil$.

\subsection{Collapsing limits of expanding CMC Einstein flows} 
\label{subsect2.1}

In what follows, ${\mathcal X}$ will denote a closed effective Hausdorff 
\'etale groupoid \cite[Section 5]{Lott (2007)}. 
We will loosely refer to it just as an \'etale groupoid.

\begin{definition} \label{2.1}
Let $I$ be an interval in $\R$. 
An Einstein flow
${\mathcal E}$ on an $n$-dimensional \'etale groupoid ${\mathcal X}$ is 
given by a
family of nonnegative functions $\{L(t)\}_{t \in I}$ on ${\mathcal X}$,
a family of Riemannian metrics $\{h(t)\}_{t \in I}$ on ${\mathcal X}$, 
and a family
of symmetric covariant $2$-tensor fields $\{K(t)\}_{t \in I}$ on ${\mathcal X}$
so that equations (\ref{1.2})-(\ref{1.5}) are satisfied.
\end{definition}

We can talk about ${\mathcal E}$ being $W^{2,p}$-regular as in
Definition \ref{1.48}. The $W^{2,p}$-norms can be defined using
integration over the orbit space, as in \cite[Section 2.6]{Hilaire (2015)}.

An expanding CMC Einstein flow on an \'etale groupoid ${\mathcal X}$ is defined
as in Definition \ref{1.7}.

The definition of convergence of Riemannian groupoids is given in
\cite[Definition 5.8]{Lott (2007)}. 
Let $\{ {\mathcal E}^{(k)} \}_{k=1}^\infty$
be CMC Einstein flows on pointed
\'etale groupoids. If ${\mathcal E}^\infty$ is a CMC
Einstein flow on a pointed \'etale groupoid, whose time slices have
metrically complete orbit spaces, then we define 
pointed weak $W^{2,p}$-convergence of $\{ {\mathcal E}^{(k)} \}_{k=1}^\infty$
to ${\mathcal E}^\infty$ by the corresponding spacetime extension,
as in Definition \ref{1.49}. 
Let ${\mathcal S}$ be a set of CMC
Einstein flows on pointed \'etale groupoids, whose time slices have
metrically complete orbit spaces. If $\{ {\mathcal E}^{(k)} \}_{k=1}^\infty$
are pointed CMC Einstein flows on \'etale groupoids then as in Definition
\ref{1.50}, we can talk about
$\{ {\mathcal E}^{(k)} \}_{k=1}^\infty$ approaching ${\mathcal S}$ as
$k\rightarrow \infty$. 
If $\{ {\mathcal E}^{(s)} \}_{s \in [s_0, \infty)}$ is a $1$-parameter
family of 
pointed CMC Einstein flows on \'etale groupoids then as in Definition
\ref{1.51}, we can talk about
$\{ {\mathcal E}^{(s)} \}_{s \in [s_0, \infty)}$ approaching ${\mathcal S}$ as
$s\rightarrow \infty$.

\begin{proposition} \label{2.2} 
Let $\left\{ {\mathcal E}^{(k)} \right\}_{k=1}^\infty$ 
be sequence of CMC
Einstein flows on pointed $n$-dimensional manifolds $(X^{(k)}, x^{(k)})$.
Suppose that ${\mathcal E}^{(k)}$ is
defined on a time-interval $I^{(k)}$, on which the mean curvature
$H^{(k)}$ is negative and increasing. Suppose that ${\mathcal E}^{(k)}$
has complete time slices.
Suppose that $I^\infty \subset \R$ is an
interval so that for any compact interval $S \subset I^\infty$,
\begin{itemize}
\item For large
$k$ we have $S \subset I^{(k)}$, and
\item For large $k$, there are uniform upper bounds on
 $\left| H^{(k)} \right|$, 
$\left| \frac{d}{dt} H^{(k)} \right|$, 
$\left|\frac{d^2}{dt^2} H^{(k)} \right|$, 
$\left| \frac{d^3}{dt^3} H^{(k)} \right|$
$- \: \frac{d}{dt} \frac{1}{H^{(k)}}$ and
$\left| \Rm^{(k)} \right|_T$ 
on $S$.
\end{itemize}

Then after passing to a subsequence,
there is a limit
$\lim_{k \rightarrow \infty} {\mathcal E}^{(k)} =
{\mathcal E}^\infty$
in the pointed weak $W^{2,p}$-topology and the pointed 
$C^{1,\alpha}$-topology. The limit flow
${\mathcal E}^\infty$ is defined on a pointed $n$-dimensional
\'etale groupoid
$({\mathcal X}^\infty, {\mathcal O}^\infty)$, 
and on the time interval $I^\infty$.
The time slices have metrically complete orbit spaces.

If for each compact interval $S \subset I^\infty$, there is some
$C_S < \infty$ such that 
$\left| K^{(k)} \right|^2 \le C_S
\frac{dH^{(k)}}{dt}$ for all large $k$, on the time interval $S$, 
then the limiting lapse function $L^\infty$ is
positive.
\end{proposition}
\begin{proof}
The proof is similar to that of Proposition \ref{1.52}.
On any compact interval $S \subset I^\infty$, the uniform bounds on
$H^{(k)}$ and $\left| \Rm^{(k)} \right|_T$ give uniform bounds on
$\left| K^{(k)} \right|$ for large $k$
\cite[Proposition 2.2]{Anderson (2001)}, 
and hence on the curvature of $h^{(k)}$.
Choose $t_0 \in I^\infty$.
As in \cite[Proposition 5.9]{Lott (2007)},
after passing to a subsequence the pointed Riemannian manifolds
$\{ (X^{(k)}, x^{(k)}, h^{(k)}) \}_{k=1}^\infty$ converge in the
pointed  weak $W^{2,p}$-topology to a pointed Riemannian groupoid
$({\mathcal X}^\infty, {\mathcal O}^\infty, h^\infty(t_0))$ whose orbit space
is metrically complete.  (The smooth convergence in
\cite[Proposition 5.9]{Lott (2007)} gets replaced by
pointed weak $W^{2,p}$-convergence.)
Given this, the construction of a limit on the 
time interval $I^\infty$ is similar to
that in the proof of Proposition \ref{1.52}.
\end{proof}

\begin{remark} \label{2.3}
There is an analog of Proposition \ref{2.2} when ${\mathcal E}^{(k)}$ is a CMC
Einstein flow on an $n$-dimensional \'etale groupoid.
\end{remark}

The unit space ${\mathcal X}^\infty_{(0)}$
 of the \'etale groupoid ${\mathcal X}^\infty$ carries
a locally constant sheaf ${\frak n}$ of finite-dimensional
Lie algebras, which act as germs of Killing vector fields on 
$({\mathcal X}^\infty_{(0)}, h^\infty(t))$. 
For any $t$, the Riemannian groupoid 
$({\mathcal X}^\infty_{(0)}, h^\infty(t))$ 
is a limit of Riemannian manifolds with
bounded curvature; hence the Lie algebras are
nilpotent.

We take $t = - \frac{n}{H}$.

\begin{corollary} \label{2.4}
Let ${\mathcal E}$ be a type-III
Einstein flow on a pointed $n$-dimensional manifold $(X,x)$.
Suppose that it is
defined on a time-interval $[t_0, \infty)$ with 
$t_0 > 0$, and has complete time slices.
Then for any sequence $\{t_i\}_{i=1}^\infty$ in $[t_0, \infty)$ with
$\lim_{i \rightarrow \infty} t_i = \infty$,
after passing to a subsequence, which we relabel as
$\{t_i\}_{i=1}^\infty$, there is a limit
$\lim_{i \rightarrow \infty} {\mathcal E}_{t_i} = 
{\mathcal E}^\infty$
in the pointed weak $W^{2,p}$-topology and the pointed 
$C^{1,\alpha}$-topology. The limit flow
${\mathcal E}^\infty$ exists on a pointed \'etale groupoid
$({\mathcal X}^\infty, {\mathcal O}^\infty)$ and is
defined on the time interval $(0, \infty)$. The orbit spaces
of its time slices are metrically complete. Its
lapse function $L^\infty$ is uniformly bounded below by a
positive constant.
\end{corollary}
\begin{proof}
Given Proposition \ref{2.2}, the proof is similar to that of
Corollary \ref{1.55}.
\end{proof}

Recall the definition of $\dvol_\infty$ from (\ref{1.22}).

\begin{proposition} \label{2.5}
Let ${\mathcal S}$ denote the collection of Einstein flows that are
Lorentzian cones over Riemannian Einstein metrics on \'etale groupoids, with 
Einstein constant $- (n-1)$. 
Under the hypotheses of Corollary \ref{2.4}, suppose that the basepoint
$x$ is such that $\dvol_\infty(x) \neq 0$. Then as $t \rightarrow \infty$,
the rescaled flows $\{ {\mathcal E}_t \}_{t=1}^\infty$ approach 
${\mathcal S}$.
\end{proposition}
\begin{proof}
With reference to Corollary \ref{2.4}, we must show that ${\mathcal E}^\infty$
describes a Lorentzian cone over an Einstein metric on ${\mathcal X}^\infty$,
with Einstein constant $-(n-1)$. 
From (\ref{1.16}), we have
\begin{equation} \label{2.5.5}
\frac{\partial}{\partial t} \ln \frac{
\dvol_\infty
}{
t^{-n} \dvol_t
} = \frac{n}{t} (1-L).
\end{equation}
Hence
\begin{equation} \label{2.6}
\int_{t_0}^\infty (1-L(x,t)) \: \frac{dt}{t} < \infty.
\end{equation}
(Recall that $L \le 1$.)
Then $L_\infty(x,u) = 1$ for all $u \in (0, \infty)$.
Equation (\ref{1.28}) (with $t$ replaced by $u$), 
along with elliptic regularity, the fact that
$g_\infty$ is locally $C^{1, \alpha}$-regular and the fact that
$K^0_\infty$ is locally $C^\alpha$-regular, implies that $L_\infty(x, u)$
is locally $C^{2,\alpha}$-regular in $x$.  
We can apply the strong maximum principle to (\ref{1.28}) on the unit space of
${\mathcal X}^\infty$ to obtain that $L_\infty = 1$ and $K_\infty^0 = 0$.
The proposition follows from Lemma \ref{1.10}.
\end{proof}

\begin{corollary} \label{2.6.25}
Under the hypotheses of Proposition \ref{2.5}, for any $C < \infty$ we
have the following asymptotics as $t \rightarrow \infty$:
\begin{enumerate}
\item The supremum of $|L-1|$ on the time-$t$ ball
$B_{h(t)}(x, Ct)$ is $o ( t^0 )$. \\
\item The supremum of $|K^0|$ on the time-$t$ ball
$B_{h(t)}(x, Ct)$ is $o ( t^{-1} )$.
\end{enumerate}
\end{corollary}
\begin{proof}
The corollary follows from the $C^\alpha$ convergence of $L$ and $K^0$,
after rescaling.
\end{proof}

\begin{remark} \label{2.6.5}
We cannot conclude that the Einstein flow is noncollapsing around $x$,
in the sense of volumes of metric balls.
Although the volume form is noncollapsing in a neighborhood of $x$ in $X$,
there is not enough control on the change of distances to deduce
noncollapsing of a time-$t$ metric ball around $x$ of radius 
comparable to $t$, as $t \rightarrow \infty$.
{\it A priori}, the rescaled spatial geometry around $x$ could resemble that
of a point going out the end of a hyperbolic cusp, for example,
while the diameter of a fixed neighborhood $U \subset X$ of $x$
increases faster than $O(t)$ so as to keep $\vol(U) \ge \const t^n$.
\end{remark}

In what follows, we will assume that $n=3$, and also make the
following assumption.

\begin{assumption} \label{2.7}
There is some $D < \infty$ so that for all $t$, we have
$\diam(X, h(t)) \le Dt$.
\end{assumption}

The reason for Assumption \ref{2.7} is that we will want to apply
monotonicity arguments to limit spaces, and will need to know that
they have compact spatial hypersurfaces.  This is ensured by 
Assumption \ref{2.7}

Because of the bounded diameter assumption, in what follows we
will not have to choose basepoints. Assumption \ref{2.7}, along with
the type-III assumption, implies that the Thurston decomposition
of $X$ consists of a single topological type
\cite[Proposition 3.5]{Lott (2010)}.
We will also assume that $X$ is aspherical, i.e. has a contractible
cover.  Then the relevant Thurston types are
$\R^3$, $H^3$, $H^2 \times \R$, $\widetilde{\SL(2, \R)}$, $Nil$ and $Sol$. 
From \cite[Lemma 6.1]{Lott (2010)}, the \'etale groupoid 
${\mathcal X}^\infty$ of
Corollary \ref{2.4} is locally free. 

\subsection{Zero dimensional orbit space} \label{subsect2.2}

In this subsection we look at the case when there is a rescaling
limit that is an Einstein flow on an \'etale groupoid with a
zero-dimensional orbit space. This is the case when 
$\liminf_{t \rightarrow \infty} t^{-1} \diam(X, h(t)) = 0$.
The unit space of
the groupoid is locally homogeneous with respect to the local
action of the sheaf ${\frak n}$ of nilpotent Lie algebras. 
The only possibilities for the stalk of ${\frak n}$ are $\R^3$ or $nil$. 

In the $\R^3$ case, the limiting Einstein flow must be a Kasner
solution.  In the $nil$ case, the limiting Einstein flow must be
the Taub-nil solution.  Hence we can say that if 
$\lim_{t \rightarrow \infty} t^{-1} \diam(X, h(t)) = 0$ then
at large times, after rescaling the geometry is modeled by
one of these two solutions.

The rescalings of a Taub-nil solution approach a Kasner solution.
Using this fact, if $\liminf_{t \rightarrow \infty} t^{-1} \diam(X, h(t)) = 0$
then we show that there is a sequence of times going to infinity so
that after rescaling, the geometry is modeled by a Kasner solution.

We first recall some facts about homogeneous Einstein solutions from
\cite{Ellis-Wainwright (1997)}.
The only $\R^3$-invariant expanding CMC Einstein flows on $\R^3$ 
(up to time translation) are
the Kasner solutions
\begin{equation} \label{2.8}
g = - du^2 + u^{2p_1} dx^2 + u^{2p_2} dy^2 + u^{2p_3} dz^2, 
\end{equation}
where $p_1 + p_2 + p_3 = p_1^2 + p_2^2 + p_3^2 = 1$. 
Equation (\ref{2.8}) is not in CMC form. Putting it in CMC form
and using the time parameter $- \frac{3}{H}$ removes the
time translation freedom.

The only left-invariant expanding CMC Einstein flows on $Nil$ 
(up to time translation) are
the Taub-nil solutions
\begin{equation} \label{2.9}
g = - A^2 du^2 + u^{2p_1} A^{-2} (dx+4p_1bzdy)^2 + u^{2p_2} A^2 dy^2 +
u^{2p_3} A^2 dz^2,
\end{equation}
where $A^2 = 1 + b^2 u^{4p_1}$ and
$p_1 + p_2 + p_3 = p_1^2 + p_2^2 + p_3^2 = 1$. 
Equation (\ref{2.9}) is not in CMC form. Putting it in CMC form
and using the time parameter $- \frac{3}{H}$ removes the
time translation freedom.

With reference to Corollary \ref{2.4}, 
suppose that the groupoid ${\mathcal X}^\infty$ has
$\dim({\frak n}) = 3$. Then the orbit space is a point.

Let ${\mathcal E}$ be an Einstein flow on a $3$-dimensional
\'etale groupoid ${\mathcal X}$ with $\dim({\frak n}) = 3$. Then the stalk
of ${\frak n}$ is $\R^3$
or $nil$.  
If the stalk is $\R^3$ then there is a cross-product description
${\mathcal X} = \R^3 \rtimes \Gamma$, where $\Gamma$ is a group (with the discrete
topology) that contains $\R^3$ as a finite-index
subgroup of translations. We say that
${\mathcal E}$ is of Kasner type.

If the stalk of ${\frak n}$ is $nil$ then 
${\mathcal X} = Nil \rtimes \Gamma$, where $\Gamma$ is a group (with the discrete
topology) that contains $Nil$
as a finite-index subgroup, acting by left multiplication.
We say that ${\mathcal E}$ is of Taub-nil type.
One can check as $s \rightarrow \infty$, the rescaled Einstein flow
${\mathcal E}_s^\infty$ approaches an Einstein flow of Kasner type (\ref{2.8})
with the same indices $(p_1, p_2, p_3)$.

Let $Kas$ denote the Einstein flows of Kasner type on $3$-dimensional \'etale
groupoids whose orbit space is a point.
Let $Taub-nil$ denote the Einstein flows 
of Taub-nil type on $3$-dimensional \'etale
groupoids whose orbit space is a point.

\begin{proposition} \label{2.10}
Let ${\mathcal E}$ be a type-III Einstein flow on a compact manifold
$X$. Suppose that there
is a sequence $\{t_i\}_{i=1}^\infty$ with $\lim_{i \rightarrow \infty}
t_i = \infty$ so that $\lim_{i \rightarrow \infty} t_i^{-1}
\diam(X, h(t_i)) = 0$. Then $X$ has Thurston type $\R^3$ or $Nil$.
As $i \rightarrow \infty$,
\begin{enumerate}
\item 
If $X$ has Thurston type $\R^3$ then
the rescaled Einstein flow ${\mathcal E}_{t_i}$ approaches $Kas$
in the weak $W^{2,p}$-topology and the $C^{1,\alpha}$-topology.
\item
If $X$ has Thurston type $Nil$ then
the rescaled Einstein flow ${\mathcal E}_{t_i}$ approaches $Kas \cup Taub-nil$
in the weak $W^{2,p}$-topology and the $C^{1,\alpha}$-topology.
\end{enumerate}
\end{proposition} 
\begin{proof}
Since $X$ admits a sequence of Riemannian metrics $\{h(t_i)\}_{i=1}^\infty$
with $\lim_{i \rightarrow \infty} |\Rm|_{h(t_i)} \diam^2(X, h(t_i)) = 0$, 
it is an almost flat manifold and
hence of Thurston type $\R^3$ or $Nil$.  Suppose that 
$X$ has Thurston type $\R^3$. Consider any subsequence
of the $t_i$'s, which we relabel as $\{t_i\}_{i=1}^\infty$.  From
Corollary \ref{2.4}, a further subsequence converges to an Einstein flow
${\mathcal E}^\infty$ on an \'etale groupoid
${\mathcal X}^\infty$, with a zero dimensional orbit space from the
diameter bound. Since $X$ is a finite quotient of $T^3$,
the local symmetry algebra ${\frak n}^\infty$
of ${\mathcal X}^\infty$ must be $\R^3$. 
Hence ${\mathcal E}^\infty \in Kas$.

If $X$ has Thurston type $Nil$ then we can again construct 
${\mathcal E}^\infty$. Now the stalk of
${\frak n}^\infty$ is $\R^3$ or $nil$. Hence
${\mathcal E}^\infty \in Kas \cup Taub-nil$.
\end{proof}

\begin{corollary} \label{2.11}
Under the hypotheses of Proposition \ref{2.10},
let $\widetilde{\mathcal E}$ denote the pullback Einstein flow on
the universal cover $\widetilde{X}$. 
Choose $\widetilde{x}_i \in \widetilde{X}$.
Under conclusion (1) of Proposition \ref{2.10}, 
$\lim_{i \rightarrow \infty} \widetilde{\mathcal E}_{t_i}$ approaches the
set of pointed 
Kasner solutions on $\R^3$, in the pointed weak $W^{2,p}$-topology
and the pointed $C^{1,\alpha}$-topology. 
Under conclusion (2) of Proposition \ref{2.10}, 
$\{ \widetilde{\mathcal E}_{t_i}, 
\widetilde{x}_i \}_{i=1}^\infty$ approaches the
set of pointed Kasner and Taub-nil 
solutions on $\R^3$, in the pointed weak $W^{2,p}$-topology
and the pointed $C^{1,\alpha}$-topology. 
\end{corollary}
\begin{proof}
Given Proposition \ref{2.10}, the corollary follows as in
\cite[Section 6.2]{Lott (2010)}.
\end{proof}

\begin{corollary} \label{2.12}
Under the hypotheses of Proposition \ref{2.10}, there is a sequence
$\{ t^\prime_j\}_{i=1}^\infty$ with $\lim_{j \rightarrow \infty}
t^\prime_j = \infty$ so that the rescalings 
$\left\{ {\mathcal E}_{t^\prime_j} \right\}_{j=1}^\infty$ approach
an Einstein flow of Kasner type on a three dimensional \'etale groupoid,
in the weak $W^{2,p}$-topology and the $C^{1,\alpha}$-topology.
\end{corollary}
\begin{proof}
Proposition \ref{2.10} implies that after passing to a subsequence of
$\{t_i\}_{i=1}^\infty$, which we relabel as $\{t_i\}_{i=1}^\infty$,
there is a limit $\lim_{i \rightarrow \infty} {\mathcal E}_{t_i} =
{\mathcal E}^\infty$, with the stalk of ${\frak n}^\infty$ equal to 
$\R^3$ or $nil$. If the stalk is $\R^3$ then
${\mathcal E}^\infty \in Kas$ and we can take
$t^\prime_i = t_i$. Suppose that the stalk is $nil$. Then
${\mathcal E}^\infty \in Taub-nil$.  As
$\lim_{s \rightarrow \infty} {\mathcal E}_s^\infty =
{\mathcal E}^{\infty, \infty}$ for some
${\mathcal E}^{\infty, \infty} \in Kas$, 
we can find a sequence $\{s_j\}_{i=1}^\infty$ with
$\lim_{j \rightarrow \infty} s_j = \infty$ so that
$\lim_{j \rightarrow \infty} {\mathcal E}_{s_j}^\infty =
{\mathcal E}^{\infty, \infty}$. From the definition of
convergence of flows, we can find a subsequence
$\left\{ t_{i_j} \right\}_{j=1}^\infty$ of 
$\left\{ t_{i} \right\}_{i=1}^\infty$ so that
$\lim_{j \rightarrow \infty} {\mathcal E}_{s_j t_{i_j}} = 
{\mathcal E}^{\infty, \infty}$. Putting $t^\prime_j = s_j t_{i_j}$,
the corollary follows. 
\end{proof}

Let $\widetilde{X}$ denote the universal cover of $X$. We give it the
pullback Einstein flow.

\begin{corollary} \label{2.13}
Under the hypotheses of Corollary \ref{2.12},
choose $\widetilde{x}^\prime_j \in \widetilde{X}$.
Then $\{
\widetilde{\mathcal E}_{t^\prime_j}, 
\widetilde{x}^\prime_j \}_{j=1}^\infty$ 
approaches the
set of pointed 
Kasner solutions on $\R^3$, in the pointed weak $W^{2,p}$-topology
and the pointed $C^{1,\alpha}$-topology.
\end{corollary}
\begin{proof}
Given Corollary \ref{2.12}, the corollary follows as in
\cite[Section 6.2]{Lott (2010)}.
\end{proof}

\begin{remark} \label{2.14}
Under the hypotheses of Proposition \ref{2.10}, it does not immediately
follow that as $i \rightarrow \infty$, the rescaled Einstein flows
$\{ {\mathcal E}_{t_i} \}_{i=1}^\infty$ approach an Einstein flow 
of Kasner type.
For example, it is conceivable that there is an infinite number of
increasingly sparse subsequences $\{ t_{j,m} \}_{j=1}^\infty$ so that
for each $m$, the limit
$\lim_{j \rightarrow \infty} {\mathcal E}_{t_{j,m}}$ exists and is
always the same Taub-nil solution. We do not know if there is an example
where the rescaled flows
$\{ {\mathcal E}_{t_i} \}_{i=1}^\infty$ approach a Taub-nil
solution.
\end{remark}

\subsection{One dimensional orbit space} \label{subsect2.3}

In this subsection we deal with the case when a limiting Einstein
flow ${\mathcal E}^\infty$
is on an \'etale groupoid whose orbit space is one dimensional,
i.e. is a circle or an interval. The goal is to show that after
performing a further rescaling, there is a new limit 
${\mathcal E}^{\infty,\infty}$
which is a Taub-flat flow;
hence an appropriate 
rescaling limit of the original Einstein flow ${\mathcal E}$
is a Taub-flat flow.

If ${\mathcal E}$ has 
$\liminf_{t \rightarrow \infty} t^{-1} \diam(X, h(t)) = 0$ then we
can consider the flow to be treated by
Corollary \ref{2.12}. Hence we assume that $\diam(X, ht)) \ge c t$
for all $t \in [t_0, \infty)$ and some $c > 0$. Then the limiting flow
${\mathcal E}^\infty$ satisfies
$\diam \left( {\mathcal X}^\infty, h^\infty(u) \right) \ge c u$ 
for all $u > 0$.
This means that any rescaling limit ${\mathcal E}^{\infty,\infty}$ of 
${\mathcal E}^\infty$ also has
a one dimensional orbit space.

Using the type-III assumption on the original flow ${\mathcal E}$, we
argue that a rescaling limit ${\mathcal E}^{\infty,\infty}$ of
${\mathcal E}^{\infty}$ exists.  Then the issue is to show that it
is a Taub-flat flow.  In order to do this, we need a monotonic quantity.
The Einstein flow ${\mathcal E}^{\infty}$ has $\dim({\frak n}) = 2$, i.e.
has local $\R^2$-symmetries.  The metric in the $\R^2$-directions is
locally given by a $2 \times 2$ matrix $G$, whose determinant is
a well-defined function on the two dimensional Lorentzian manifold
$(0, \infty) \times X^\infty$. We can assume that
${\mathcal E}^\infty$ is not already a Taub-flat flow. Then it is
known that $\nabla \det G$ is a nonvanishing timelike vector on
$(0, \infty) \times X^\infty$. The level sets of $\det G$ are 
spacelike submanifolds; we assume that they are compact.  Then we can
use the monotonic quantities defined in Subsection \ref{subsectA.3}.

We obtain an integral convergence result along the lines of
Proposition \ref{1.37}. To go further, we make the additional
assumption that there is a time function $\widehat{u}$ for the foliation
of $(0, \infty) \times X^\infty$ by level sets of $\det(G)$, which is
comparable to $u$.
Using this time function and the monotonic quantities
from Subsection \ref{subsectA.3}, we deduce that 
${\mathcal E}^{\infty,\infty}$ is a Taub-flat flow. Hence there are
arbitrarily large future time intervals on which the
original flow ${\mathcal E}$ is modelled, in a scale invariant way,
by a Taub-flat flow.

To begin,
the Taub-flat vacuum solution is the isometric product of $\R^2$ with
the Lorentzian cone over $H^1 \cong \R^1$. Suppose that ${\mathcal X}$ is a
three dimensional cross-product groupoid $(H^1 \times \R^2) \rtimes \Gamma$,
where $\Gamma$ is a group (with the discrete topology) that contains
the translations $L\Z \times \R^2$ as a finite-index subgroup, for
some $L > 0$. 
We say that an Einstein flow on
${\mathcal X}$ is of Taub-flat type if the corresponding Lorentzian
groupoid is equivalent to the cross product of $\Gamma$ with
the Taub-flat solution.

Let ${\mathcal E}$ be an Einstein flow as in the hypotheses of Corollary 
\ref{2.4}, satisfying Assumption \ref{2.7}.  
If $X$ has Thurston type $Sol$ and $\{t_i\}_{i=1}^\infty$ is a sequence
with $\lim_{i \rightarrow \infty} t_i = \infty$ then a 
Gromov-Hausdorff limit of $\{ (X, t_i^{-2} h(t_i)) \}_{i=1}^\infty$
cannot be three dimensional by Proposition \ref{1.57}.
It cannot be zero dimensional (or else $X$ would have Thurston type
$\R^3$ or $Nil$) and it cannot be two dimensional (or else $X$ would be
a Seifert $3$-manifold).  Thus if 
$X$ has Thurston type $Sol$ then a 
Gromov-Hausdorff limit of $\{ (X, t_i^{-2} h(t_i)) \}_{i=1}^\infty$
must be one dimensional.
More generally, if $\{ (X, t_i^{-2} h(t_i)) \}_{i=1}^\infty$
has a one dimensional Gromov-Hausdorff limit
then $X$ must have Thurston type $Sol$, $\R^3$ or $Nil$.  

With reference to Corollary \ref{2.4}, 
suppose that the groupoid ${\mathcal X}^\infty$ has
$\dim({\frak n}) = 2$. Then the orbit space is one dimensional.
If $X$ has Thurston type $Sol$ then there is some $c > 0$ so that
$\diam(X, h(t)) \ge c t$ for all $t \ge t_0$, as $X$ is not almost
flat.  If $X$ has Thurston type $\R^3$ or $Nil$, and
$\liminf_{t \rightarrow \infty} t^{-1} \diam(X, h(t)) = 0$, then
we can consider the Einstein flow to be covered by Corollary \ref{2.12}.
Hence we make the following assumption.

\begin{assumption} \label{2.15}
For some $c > 0$, we have $\diam(X, h(t)) \ge c t$ for
all $t \in [t_0, \infty)$.
\end{assumption}

With reference to Corollary \ref{2.4},
letting $X^\infty$ denote the orbit space of ${\mathcal X}^\infty$, we
loosely write $\diam({\mathcal X}^\infty, h^\infty(u))$ for the
diameter of $X^\infty$ with the induced metric.
Then 
\begin{equation} \label{2.16}
\diam({\mathcal X}^\infty, h^\infty(u)) \ge cu
\end{equation}
for all $u \in (0, \infty)$.

The orbit space $X^\infty$ is a one dimensional orbifold
\cite[Pf. of Proposition 3.5]{Lott (2010)}.  Hence it is
$S^1$ or $S^1/\Z_2$. In the latter case, we can pullback under the orbifold
covering map $S^1 \rightarrow S^1/\Z_2$ to reduce the statements to the
$S^1$ case.  Hence we assume that the orbit space $X^\infty$ is 
diffeomorphic to $S^1$.

The coordinate function $u$ on $(0, \infty) \times 
{\mathcal X}^\infty$ pulls back
from a function on $(0, \infty) \times S^1$, which we again denote by $u$.
Similarly, the lapse function $L$ pulls back from a function on
$(0, \infty) \times S^1$, which we again denote by $L$. 
As in Subsection \ref{A.1},
the Lorentzian 
metric corresponding to the Einstein flow
${\mathcal E}^\infty$ can locally be written as
\begin{equation} \label{2.17}
-L^2 du^2 + h d\theta^2 + \sum_{I,J=1}^2 G_{IJ} (db^I+A^I) (db^J + A^J).
\end{equation}
We note that $\det(G)$ is a well defined function on $(0, \infty) \times S^1$,
since the flat twisting bundle $e$ on $(0, \infty) \times S^1$ has holonomy
in $\SL(2, \Z)$ \cite[Proof of Lemma 6.1]{Lott (2010)}.
Put $g = -L^2 du^2 + h d\theta^2$.

If ${\mathcal E}^\infty$ is not flat then
the function $\det(G)$ has a timelike gradient on
$(0, \infty) \times S^1$; see 
\cite[Proof of Proposition 5.1]{LeFloch-Smulevici (2015)} and
references therein. If ${\mathcal E}^\infty$ is flat then the results
of this subsection will be valid, so we assume that
${\mathcal E}^\infty$ is not flat. Then $\det(G)$ has level sets that
foliate $(0, \infty) \times S^1$.
We assume that $\nabla \det(G)$ is future-directed; 
this holds, for example, in the Ellis-MacCallum $Sol$-solution 
\cite[Section 9.2.3]{Ellis-Wainwright (1997)}. Then we can choose
a time function $\widehat{u}$ on $(0, \infty) \times S^1$ that
is an increasing function of $\det(G)$.

\begin{assumption} \label{2.18}
There is an open set $U \subset (0, \infty) \times S^1$ containing
$[u_0, \infty) \times S^1$ for some $u_0 < \infty$, and a proper function 
$\widehat{u} \in W^{3,p}_{loc}(U)$ so that
\begin{enumerate}
\item $\nabla \widehat{u}$ is timelike, and
\item On $U$, $\det(G)$ is a function of $\widehat{u}$.
\end{enumerate}
\end{assumption}

Assumption \ref{2.18} implies the level sets of $\widehat{u}$ are compact
manifolds.
We can assume that they are diffeomorphic to $S^1$.
Then for suitable $\widehat{u}_0 < \infty$, the space
$\widehat{u}^{-1}([\widehat{u}_0,
\infty))$ is $W^{3,p}_{loc}$-diffeomorphic to 
$[\widehat{u}_0, \infty) \times S^1$. We write the
Lorentzian metric $g$ on $[\widehat{u}_0, \infty) \times S^1$, in terms of
$\widehat{u}$ and $\theta$, as
$\widehat{g} = - \widehat{L}^2 d\widehat{u}^2 + \widehat{h} d\theta^2$.

If the curvature $F$ of the $\R^N$-valued connection $A$ vanishes then
from (\ref{A.20}),
\begin{equation} \label{2.19}
\int_{\widehat{u}_0}^\infty \frac{1}{\sqrt{\det G}} \widehat{L}^{-1} 
\Tr \left( \left( G^{-1} \partial_{\widehat{u}} G \right)^2 \right)  
\dvol(S^1, \widehat{h}(\widehat{u})) \: d\widehat{u} < \infty.
\end{equation}

Given $s > 0$, define $\widehat{L}_s$ and $\widehat{h}_s$ as in 
(\ref{1.36}). Put
${G}_s(\widehat{v}) = {G}(s\widehat{v})$. 

\begin{proposition} \label{2.20}
Suppose that $F = 0$.
Given $\Lambda > 1$, we have
\begin{equation} \label{2.21}
\lim_{s \rightarrow \infty} 
\frac{1}{\sqrt{\det G_s}} \widehat{L}_s^{-1} 
\Tr \left( \left( G_s^{-1} \partial_{\widehat{v}} G_{s} \right)^2 \right)
d\widehat{v} \: \dvol(S^1, \widehat{h}_s(\widehat{v}))  = 0
\end{equation}
in norm convergence of measures on  $[\Lambda^{-1}, \Lambda] \times S^1$.
\end{proposition}
\begin{proof}
The proof is similar to that of Proposition \ref{1.37}.  We omit the details.
\end{proof}

\begin{remark} \label{2.22}
From Subsection \ref{subsectA.3}, if $G_s^{-1} \partial_{\widehat{v}} G_{s} 
= 0$ then $G$ is locally constant in $\widehat{v}$ and $\theta$, and 
$\widehat{g}$ is flat.
Hence Proposition \ref{2.20} can be interpreted as saying that
in an integral sense, the original flow ${\mathcal E}$ is approaching
a flow of Taub-flat type. 
If $F \neq 0$ then there is a result analogous to Proposition
\ref{2.20}, except more
complicated to state, using (\ref{A.27}) and (\ref{A.28}).
\end{remark}

We now make a further assumption about $\widehat{u}$, saying that it is
comparable to $u$.

\begin{assumption} \label{2.23}
In addition to Assumption \ref{2.18}, suppose that
there is a constant $\Lambda < \infty$ so that
\begin{enumerate}
\item
$\Lambda^{-1} u \le \widehat{u} \le \Lambda u$,
\item For all $r > u_0$, $p < \infty$ and $k + l \le 3$,
\begin{equation} \label{2.24}
\parallel \nabla_x^k \partial_u^l \widehat{u} \parallel_{
L^p((r, 2r) \times X^\infty)
} \le \const r^{1-k-l+\frac{2}{p}} and
\end{equation} 
\item 
$\frac{{g}(\nabla u, \nabla \widehat{u})}
{|\nabla u|_{{g}}  | \nabla \widehat{u}|_{{g}}} 
\le - \Lambda^{-1}$.
\end{enumerate}
\end{assumption}

\begin{remark}
The exponent on the right-hand side of (\ref{2.24}) ensures scale invariance.
\end{remark}

\begin{proposition} \label{2.25}
If Assumption \ref{2.23} holds then there is a sequence
$\{t_j^\prime \}_{j=1}^\infty$ with $\lim_{j \rightarrow \infty} t^\prime_j =
\infty$, and an Einstein flow
${\mathcal E}^{\infty, \infty}$ of Taub-flat type,
so that the rescalings ${\mathcal E}_{t^\prime_j}$ of ${\mathcal E}$ satisfy
$\lim_{j \rightarrow \infty}
{\mathcal E}_{t^\prime_j} = {\mathcal E}^{\infty, \infty}$.
\end{proposition}
\begin{proof}
Let $\{s_j\}_{j=1}^\infty$ be a sequence with $\lim_{j \rightarrow \infty} 
s_j = \infty$. Since $\lim_{i \rightarrow \infty} {\mathcal E}_{t_i} =
{\mathcal E}^\infty$, for fixed $j$, we have
$\lim_{i \rightarrow \infty} {\mathcal E}_{s_j t_{i}} =
{\mathcal E}^\infty_{s_j}$.
If $\{ t_{i_j} \}_{j=1}^\infty$ is a subsequence of $\{ t_{i} \}_{i=1}^\infty$
then after passing to a subsequence of $j$'s, we can assume that
$\lim_{j \rightarrow \infty} {\mathcal E}_{s_j t_{i_j}} =
{\mathcal E}^{\infty, \infty}$ for an Einstein flow 
${\mathcal E}^{\infty, \infty}$ on an \'etale groupoid
${\mathcal X}^{\infty, \infty}$, defined on 
the time interval $(0, \infty)$.  
From our definition of convergence of flows,
we can choose $\{ t_{i_j} \}_{j=1}^\infty$ so that
$\lim_{j \rightarrow \infty} {\mathcal E}^\infty_{s_j} =
\lim_{j \rightarrow \infty} {\mathcal E}_{s_j t_{i_j}} =
{\mathcal E}^{\infty, \infty}$.
(The rescaling in
${\mathcal E}^\infty_{s_j}$ involves pullback with respect to 
$u \rightarrow s_j u$ and a $j$-dependent diffeomorphism $\phi_j$ of
$S^1$, along with a $j$-dependent automorphism of the flat 
$\R^2$-vector bundle on $S^1$.)
From Assumption \ref{2.15},
the orbit space of ${\mathcal X}^{\infty, \infty}$ is one dimensional. 
The Lorentzian metric corresponding to the Einstein flow
${\mathcal E}^{\infty, \infty}$ can be locally written as
\begin{equation} \label{2.26}
- (L^\infty)^2 (du^\infty)^2 + h^\infty (d\theta^\infty)^2 + 
\sum_{I,J=1}^2 G^\infty_{IJ} (db^I+A^{\infty,I}) (db^J + A^{\infty,J}).
\end{equation}
We will show this is an Einstein flow of Taub-flat type and take
$t^\prime_j = s_j t_{i_j}$.

Let $\widehat{X}^\infty$ denote the level sets of $\widehat{u}$.
Let $\widehat{u}_{s_j}$ be $\frac{1}{s_j}$ times
the pullback of $\widehat{u}$ with respect to
$u  \rightarrow s_j u$ and $\phi_j \in \Diff(X^\infty)$.
From Assumption \ref{2.23}(1,2), after passing
to a subsequence we can assume that $\lim_{j \rightarrow \infty}
\widehat{u}_{s_j} = \widehat{u}^\infty$ in the weak topology on
$W^{3,p}_{loc}$, for some
$\widehat{u}^\infty \in W^{3,p}_{loc}((0, \infty) \times 
X^{\infty, \infty})$.
From Assumption \ref{2.18} and Assumption \ref{2.23}(3), the gradient
$\nabla \widehat{u}^\infty$ is timelike, and $\det(G)$ is
a function of $\widehat{u}^\infty$.
We can write the Lorentian metric on $(0, \infty) \times X^{\infty, \infty}$ 
as
\begin{equation} \label{2.27}
- (\widehat{L}^\infty)^2 (d\widehat{u}^\infty)^2 + \widehat{h}^\infty 
(d\widehat{\theta}^\infty)^2.
\end{equation}

Suppose first that the curvature $F^\infty$ of $A^\infty$ vanishes.
Applying Subsection \ref{A.2} to the flow ${\mathcal E}^\infty$, we know that
$(\partial_{\widehat{u}} \ln \det G) 
\int_{\widehat{X}^\infty} \widehat{L}^{-1} \dvol_{\widehat{X}^\infty}$ is
monotonically nonincreasing in $\widehat{u}$. It is clearly nonnegative. Since 
$\dim(\widehat{X}^\infty) = 1$,
the expression is invariant under rescaling.
Note that in forming the limit
$\lim_{j \rightarrow \infty} {\mathcal E}^\infty_{s_j} =
{\mathcal E}^{\infty, \infty}$, we are allowed to perform 
$j$-dependent automorphisms of
the flat $2$-dimensional vector bundle on $X^\infty$.
These automorphisms can change
$\ln \det G$ by a $j$-dependent additive constant, which vanishes
upon taking the $\widehat{u}$-derivative.  

Given $a \in (0, \infty)$, the level set $(\widehat{u}^\infty)^{-1}(a)
\subset (0, \infty) \times X^{\infty, \infty}$ is
the limit of rescalings of level sets $\widehat{u}^{-1}(s_j a)
\subset (0, \infty) \times X^{\infty}$.
It follows that the monotonic quantity
$(\partial_{\widehat{u}^\infty} \ln \det G^\infty) 
\int_{\widehat{X}^{\infty, \infty}} 
(\widehat{L}^\infty)^{-1} \dvol_{\widehat{X}^{\infty, \infty}}$ is constant in 
$\widehat{u}^\infty$. By Subsubsection \ref{subsubsectA.3.1}, we conclude that
${\mathcal E}^{\infty, \infty}$ is a flat solution.

\begin{remark} \label{2.28}
We could have reached the same conclusion using the functional
$\widehat{\mathcal E}$ of (\ref{A.19}). 
\end{remark}

Now suppose that $F^\infty \neq 0$. 
After pulling back from a finite cover of $S^1$ if necessary, we can
assume that the holonomy $H \in \SL(2, \R)$ over $S^1$ of the flat 
vector bundle $E$ has real positive eigenvalues.
The functional $\widehat{\mathcal E}_K$ of (\ref{A.26}) is
scale invariant and monotonically nonincreasing. It follows that
the corresponding functional for ${\mathcal E}^{\infty, \infty}$ is
constant. In terms of Subsubsection \ref{subsubsectA.3.2}, 
the metric (\ref{2.26}) equals
(\ref{A.29}), after a change of variable from $u^\infty$ to $R$. 
From Corollary \ref{2.4}, the function $L^\infty$ is uniformly bounded
below and so the metric (\ref{2.26}) admits future-directed timelike curves
along which the proper time goes to infinity.
Hence under the change of variable
from $u^\infty$ to $R$, we must have $\lim_{u^\infty \rightarrow \infty}
R(u^\infty) = \infty$. 
From (\ref{A.29}), the length of the $S^1$-fiber is uniformly
bounded as $R$ goes to infinity.  This contradicts (\ref{2.16}), showing
that $F^\infty$ cannot be nonzero.

Hence
${\mathcal E}^{\infty, \infty}$ is a flat solution.
There is a foliation of $(0, \infty) \times X^{\infty, \infty}$ by circles
$\{C_v\}_{v \in (0, \infty)}$ of constant geodesic curvature $- \:
\frac{1}{v}$. The lift $\widetilde{C}_v$ of such a circle to the
universal cover $(0, \infty) \times \widetilde{X}^{\infty, \infty}$ is
an embedded curve with a neighborhood that is isometric to a
neighborhood of a hyperbola, of constant geodesic curvature $- \:
\frac{1}{v}$, in the flat Lorentzian plane $\R^{1,1}$.  As $(0,
\infty) \times \widetilde{X}^{\infty, \infty}$ is foliated by such
lifts, it must be isometric to the chronological future of the origin
in $\R^{1,1}$, with its foliation by hyperbolas.  Then $(0, \infty)
\times {X}^{\infty, \infty}$ is the Lorentzian cone over a circle, and
${\mathcal E}^{\infty,\infty}$ is an Einstein flow of Taub-flat type.
\end{proof}

Let $\widetilde{X}$ denote the universal cover of $X$. We give it the
pullback Einstein flow.

\begin{corollary} \label{2.29}
Under the hypotheses of Proposition \ref{2.25},
choose $\widetilde{x}^\prime_j \in \widetilde{X}$.
Then $\{
( \widetilde{\mathcal E}_{t^\prime_j}, 
\widetilde{x}^\prime_j) \}_{j=1}^\infty$ 
approaches the
set of Taub-flat Einstein flows
on $\R^3$, in the pointed weak $W^{2,p}$-topology
and the pointed $C^{1,\alpha}$-topology.
\end{corollary}
\begin{proof}
Given Proposition \ref{2.25}, the corollary follows as in
\cite[Section 6.2]{Lott (2010)}.
\end{proof}

\subsection{Two dimensional orbit space} \label{subsect2.4}

In this subsection we deal with the case when a limiting Einstein
flow ${\mathcal E}^\infty$
is on an \'etale groupoid whose orbit space is two dimensional.
In our case, it will necessarily be a two dimensional orbifold. 
The goal is to show that after
performing a further rescaling, there is a new limit 
${\mathcal E}^{\infty,\infty}$
which is a  Bianchi-III flat flow;
hence an appropriate 
rescaling limit of the original Einstein flow ${\mathcal E}$
is a Bianchi-III flat flow.

If ${\mathcal E}$ has a rescaling limit whose orbit space has 
dimension zero or one
then we can consider the flow to be treated by
Corollary \ref{2.12} and Proposition \ref{2.25}.
Hence we assume that there is no such rescaling limit.
This implies that any rescaling limit ${\mathcal E}^{\infty,\infty}$ of 
${\mathcal E}^\infty$ also has
a two dimensional orbit space.

Using the type-III assumption on the original flow ${\mathcal E}$, we
argue that a rescaling limit ${\mathcal E}^{\infty,\infty}$ of
${\mathcal E}^{\infty}$ exists.  Then the issue is to show that it
is a Bianchi-III flat flow.  In order to do this, we make the
conformal change of Subsection \ref{subsectA.4} and assume that
the ensuing Lorentzian $3$-manifold has an expanding CMC foliation,
Using the monotonic quantity of Subsection \ref{subsectA.4}, 
we obtain an integral convergence result along the lines of
Proposition \ref{1.37}. To go further, we make the additional
assumption that there is a time function $\widehat{u}$ for the 
new CMC foliation, which is
comparable to $u$.
Using this time function and the monotonic quantity
from Subsection \ref{subsectA.4}, we deduce that 
${\mathcal E}^{\infty,\infty}$ is a Bianchi-III flat flow. Hence there are
arbitrarily large future time intervals on which the
original flow ${\mathcal E}$ is modelled, in a scale invariant way,
by a Bianchi-III flat flow.

To begin,
the Bianchi-III flat vacuum solution is the isometric product of $\R$ with
the Lorentzian cone over $H^2$. Suppose that ${\mathcal X}$ is a
three dimensional cross-product groupoid $(\R \times \R^2) \rtimes \Gamma$,
where $\Gamma$ is a group (with the discrete topology) that contains
$\R \times \Gamma_0$ as a finite-index subgroup, with
$\Gamma_0$ being a discrete subgroup of $\Isom(H^2)$. 
We say that an Einstein flow on
${\mathcal X}$ is of Bianchi-III flat type if the corresponding
Lorentzian groupoid is the cross-product of $\Gamma$ with
the Bianchi-III flat vacuum solution.

Let ${\mathcal E}$ be an Einstein flow as in the hypotheses of Corollary 
\ref{2.4}, satisfying Assumption \ref{2.7}.  
If $X$ has Thurston type $H^2 \times \R$ or $\widetilde{\SL(2, \R)}$,
 and $\{t_i\}_{i=1}^\infty$ is a sequence
with $\lim_{i \rightarrow \infty} t_i = \infty$, then a 
Gromov-Hausdorff limit of $\{ (X, t_i^{-2} h(t_i)) \}_{i=1}^\infty$
cannot be three dimensional by Proposition \ref{1.57}.
It cannot be zero dimensional (or else $X$
would have Thurston type
$\R^3$ or $Nil$) and it cannot be one dimensional 
(or else a finite cover of $X$ 
would be the total space of a $T^2$-bundle over a circle).  Thus if 
$X$ has Thurston type $H^2 \times \R$ or $\widetilde{\SL(2, \R)}$ then a 
Gromov-Hausdorff limit of $\{ (X, t_i^{-2} h(t_i)) \}_{i=1}^\infty$
must be two dimensional.
More generally, if
there is a sequence $\{t_i\}_{i=1}^\infty$ with $\lim_{i \rightarrow
\infty} t_i = \infty$ 
so that $\{ 
(X, t_i^{-2} h(t_i)) \}_{i=1}^\infty$
has a two dimensional Gromov-Hausdorff limit, 
then $X$ must have Thurston type
$H^2 \times \R$, $\widetilde{\SL(2, \R)}$, $\R^3$ or $Nil$.

With reference to Corollary \ref{2.4}, 
suppose that the groupoid ${\mathcal X}^\infty$ has
$\dim({\frak n}) = 1$. Then the orbit space is two dimensional.

If
there is a sequence $\{t_i\}_{i=1}^\infty$ with $\lim_{i \rightarrow
\infty} t_i = \infty$ so that $\{ 
(X, t_i^{-2} h(t_i)) \}_{i=1}^\infty$
has a Gromov-Hausdorff limit of dimension less than
two, then
we can consider the Einstein flow to be covered by 
Corollary \ref{2.12} and Proposition \ref{2.25}
(modulo the verification of Assumption \ref{2.23}).
Hence we make the following assumption.

\begin{assumption} \label{2.31}
There is no sequence $\{t_i\}_{i=1}^\infty$ with $\lim_{i \rightarrow
\infty} t_i = \infty$ so that $\{ 
(X, t_i^{-2} h(t_i)) \}_{i=1}^\infty$ has a Gromov-Hausdorff limit
of dimension less than two.
\end{assumption}

Let ${\mathcal E}^\infty$ be a limit flow on an \'etale groupoid
${\mathcal X}^\infty$, as in Corollary \ref{2.4}.
The orbit space $X^\infty$ of ${\mathcal X}^\infty$ 
is a two dimensional orbifold
\cite[Pf. of Proposition 3.5]{Lott (2010)}.
(This uses our assumption that $X$ is aspherical.)
From Assumption \ref{2.31}, there is no subsequence $\{u_j\}_{j=1}^\infty$ with
$\lim_{j \rightarrow \infty} u_j = \infty$ so that
$\{ 
(X^\infty, u_j^{-2} h^\infty(u_j)) \}_{j=1}^\infty$
has a Gromov-Hausdorff limit of dimension less than two.

The coordinate function $u$ on $(0, \infty) \times 
{\mathcal X}^\infty$ pulls back
from a function on $(0, \infty) \times X^\infty$, which we again denote by $u$.
Similarly, the lapse function $L$ pulls back from a function on
$(0, \infty) \times X^\infty$, which we again denote by $L$. 
As in Subsection \ref{A.1},
the Lorentzian 
metric corresponding to the groupoid Einstein flow
${\mathcal E}^\infty$ can locally be written as
\begin{equation} \label{2.32}
-L^2 du^2 + \sum_{\alpha, \beta =1}^2 h_{\alpha \beta} 
db^\alpha db^\beta  + G (d\theta+A)^2.
\end{equation}
Put 
\begin{equation} \label{2.33}
g = -L^2 du^2 + \sum_{\alpha, \beta =1}^2 h_{\alpha \beta} 
db^\alpha db^\beta,
\end{equation}
$\widehat{g} = Gg$ and 
$\widehat{G} = G^2$.
Then
\begin{equation} \label{2.34}
g + G (d\theta + A)^2 = \widehat{G}^{- \: \frac12} 
\left( \widehat{g} + \widehat{G} (d\theta + A)^2 \right).
\end{equation}

\begin{assumption} \label{2.35}
There is an open set $U \subset (0, \infty) \times X^\infty$ containing
$[u_0, \infty) \times X^\infty$ for some $u_0 < \infty$, and a proper function 
$\widehat{u} \in W^{3,p}_{loc}(U)$ so that
\begin{enumerate}
\item $\nabla \widehat{u}$ is timelike, and
\item On $U$, the level sets of $\widehat{u}$ have constant mean
curvature with respect to $\widehat{g}$.
\end{enumerate}
\end{assumption}

Assumption \ref{2.35} implies the level sets of $\widehat{u}$ are compact.
Let us denote their diffeomorphism type by
$\widehat{X}^\infty$. Then for suitable $\widehat{u}_0 < \infty$, the space
$\widehat{u}^{-1}([\widehat{u}_0,
\infty))$ is $W^{3,p}_{loc}$-diffeomorphic to 
$[\widehat{u}_0, \infty) \times 
\widehat{X}^\infty$.

Letting $\widehat{H}$ denote the (constant) mean curvatures of the level
sets,
suppose that $\widehat{H}$ is an increasing function in $\widehat{u}$ that 
takes all values in an interval $(-\widehat{H}_0, 0)$. Define a new
time parameter by $v = - \: \frac{2}{\widehat{H}}$. From (\ref{A.50}),
$v^{-2} \dvol(\widehat{X}^\infty, \widehat{h}(v))$ is pointwise decreasing.
Put 
\begin{equation} \label{2.36}
\widehat{\dvol}_\infty = \lim_{v \rightarrow \infty}
v^{-2} \dvol(\widehat{X}^\infty, \widehat{h}),
\end{equation}
an absolutely continuous measure on $\widehat{X}^\infty$. From (\ref{A.51}),
\begin{align} \label{2.37}
& \frac{d}{dv} \left( v^{-2} \vol(\widehat{X}^\infty, \widehat{h}(v)) 
\right) = 
- \: v \int_{\widehat{X}^\infty} 
\left[
\widehat{L} |\widehat{K}^{0}|^2 + 
\frac14 \widehat{L}^{-1} \left| \widehat{S}_0 \right|^2 + 
\right. \\
& \left. \frac14 \widehat{L}^{-1} \left( 
\frac{\partial \ln \det \widehat{G}}{\partial v} \right)^2 +
\frac14 \widehat{L}
\widehat{h}^{ij} \widehat{h}^{kl} 
\widehat{G}_{IJ} \widehat{F}^{I}_{ik} 
\widehat{F}^{J}_{jl}
\right]\: \dvol(\widehat{X}^\infty, \widehat{h}(v)). \notag
\end{align}

Given $s > 0$, define $\widehat{L}_s$, $\widehat{h}_s$, 
$\widehat{K}_s$ and $\widehat{K}_s^0$ as in (\ref{1.36}). Put
$\widehat{G}_s(v) = \widehat{G}(sv)$ and
$\widehat{F}_{s,ij}(v) = s^{-1} \widehat{F}_{ij}(sv)$.

\begin{proposition} \label{2.38}
Given $\Lambda > 1$, we have
\begin{align} \label{2.39}
& \lim_{s \rightarrow \infty} (\widehat{L}_s-1) =
\lim_{s \rightarrow \infty} |\widehat{K}^0|^2_s \widehat{L}_s =
\lim_{s \rightarrow \infty} |\widehat{S}^0|^2_s \widehat{L}_s =
\lim_{s \rightarrow \infty} 
\left( \frac{\partial \ln \det \widehat{G}_s}{\partial v} \right)^2
\widehat{L}_s^{-1} = \\
& \lim_{s \rightarrow \infty} \widehat{h}_s^{ij} \widehat{h}_s^{kl} 
\widehat{G}_{s,IJ} \widehat{F}^{I}_{s,ik} 
\widehat{F}^{J}_{s,jl} \widehat{L}_s = 0 \notag
\end{align}
in $L^1 \left( [\Lambda^{-1}, \Lambda] \times \widehat{X}_\infty,
dv \: \widehat{\dvol}_\infty \right)$.
\end{proposition}
\begin{proof}
The proof is similar to that of Proposition \ref{1.37}.  We omit the details.
\end{proof}

\begin{remark} 
From Subsection \ref{subsectA.4}, if $\widehat{L}-1=\widehat{K}^0=
\widehat{S}^0=
\frac{\partial \ln \det \widehat{G}}{\partial v} = 
\widehat{h}^{ij} \widehat{h}^{kl} \widehat{G}_{IJ} 
\widehat{F}^I_{ik} \widehat{F}^J_{jl} = 0$
then $\widehat{G}$ is locally constant and 
$\widehat{g}$ is flat.
Hence Proposition \ref{2.38} can be interpreted as saying that
in an integral sense, the original flow ${\mathcal E}$ is approaching
a flow of Bianchi-III flat type.
\end{remark}

Define $\widehat{H}^1_{s,\Lambda}$ as in the paragraph before
Proposition \ref{1.42}, replacing $h_s$ by $\widehat{h}_s$.

\begin{proposition} \label{2.40}
We have
\begin{equation} \label{2.41}
\lim_{s \rightarrow \infty} d_{symm}(\widehat{H}^1_{s,\Lambda}, I_n) = 0
\end{equation}
in $L^2(\widehat{X}_\infty, \widehat{\dvol_\infty})$.
\end{proposition}
\begin{proof}
The proof is similar to that of Proposition \ref{1.42}.  We omit the details.
\end{proof}

We now make a further assumption about $\widehat{u}$, saying that it is
comparable to $u$.

\begin{assumption} \label{2.42}
In addition to Assumption \ref{2.35}, there is some
$\Lambda < \infty$ so that
\begin{enumerate}
\item
$\Lambda^{-1} u \le \widehat{u} \le \Lambda u$,
\item For all $r > u_0$, $p < \infty$ and $k + l \le 3$,
\begin{equation} \label{2.43}
\parallel \nabla_x^k \partial_u^l \widehat{u} \parallel_{
L^p((r, 2r) \times X^\infty)
} \le \const r^{1-k-l+\frac{3}{p}},
\end{equation} 
and
\item
$\frac{{g}(\nabla u, \nabla \widehat{u})}
{|\nabla u|_{{g}}  | \nabla \widehat{u}|_{{g}}} 
\le - \Lambda^{-1}$.
\end{enumerate}
\end{assumption}

\begin{remark}
The exponent on the right-hand side of (\ref{2.43}) ensures scale invariance.
\end{remark}

\begin{proposition} \label{2.44}
If Assumption \ref{2.42} holds then there is a sequence
$\{t_j^\prime \}_{j=1}^\infty$ with $\lim_{j \rightarrow \infty} t^\prime_j =
\infty$, and an Einstein flow ${\mathcal E}^{\infty, \infty}$ of
Bianchi-III flat type,
so that the rescalings ${\mathcal E}_{t^\prime_j}$ of ${\mathcal E}$ satisfy
$\lim_{j \rightarrow \infty}
{\mathcal E}_{t^\prime_j} = {\mathcal E}^{\infty, \infty}$.
\end{proposition}
\begin{proof}
Let $\{s_j\}_{j=1}^\infty$ be a sequence with $\lim_{j \rightarrow \infty} 
s_j = \infty$. Since $\lim_{i \rightarrow \infty} {\mathcal E}_{t_i} =
{\mathcal E}^\infty$, for fixed $j$, we have
$\lim_{i \rightarrow \infty} {\mathcal E}_{s_j t_{i}} =
{\mathcal E}^\infty_{s_j}$.
If $\{ t_{i_j} \}_{j=1}^\infty$ is a subsequence of $\{ t_{i} \}_{i=1}^\infty$
then after passing to a subsequence of $j$'s, we can assume that
$\lim_{j \rightarrow \infty} {\mathcal E}_{s_j t_{i_j}} =
{\mathcal E}^{\infty, \infty}$ for an Einstein flow 
${\mathcal E}^{\infty, \infty}$ on an \'etale groupoid
$X^{\infty, \infty}$, defined on 
the time interval $(0, \infty)$.  
From our definition of convergence of flows,
we can choose $\{ t_{i_j} \}_{j=1}^\infty$ so that
$\lim_{j \rightarrow \infty} {\mathcal E}^\infty_{s_j} =
\lim_{j \rightarrow \infty} {\mathcal E}_{s_j t_{i_j}} =
{\mathcal E}^{\infty, \infty}$. (The rescaling in
${\mathcal E}^\infty_{s_j}$ involves pullback with respect to 
$u \rightarrow s_j u$ and a $j$-dependent diffeomorphism $\phi_j$ of
$X^\infty$, along with a $j$-dependent automorphism of the flat 
$\R$-vector bundle on $X^\infty$.)
From Assumption \ref{2.31},
the orbit space of $X^{\infty, \infty}$ is two dimensional.
The Lorentzian metric corresponding to the groupoid Einstein flow
${\mathcal E}^{\infty, \infty}$ can be locally written as
\begin{equation} \label{2.45}
-(L^\infty)^2 (du^\infty)^2 + \sum_{\alpha, \beta = 1}^2
h^\infty_{\alpha \beta} db^\infty_\alpha db^\infty_\beta + 
G^\infty (d\theta +A^{\infty})^2.
\end{equation}
We will show this is an Einstein flow of Bianchi-III flat type and take
$t^\prime_j = s_j t_{i_j}$. 

Put 
\begin{equation} \label{2.46} 
g^\infty = 
-(L^\infty)^2 (du^\infty)^2 + \sum_{\alpha, \beta = 1}^2
h^\infty_{\alpha \beta} db^\infty_\alpha db^\infty_\beta,
\end{equation}
$\widehat{g}^\infty = G^\infty g^\infty$ and 
$\widehat{G}^\infty = (G^\infty)^2$.
Then
\begin{equation} \label{2.47}
g^\infty + G^\infty (d\theta + A)^2 = (\widehat{G}^\infty)^{- \: \frac12} 
\left( \widehat{g}^\infty + \widehat{G}^\infty (d\theta + 
A^\infty)^2 \right).
\end{equation}

Let $\widehat{u}_{s_j}$ be $\frac{1}{s_j}$ times
the pullback of $\widehat{u}$ with respect to
$u  \rightarrow s_j u$ and $\phi_j \in \Diff(X^\infty)$.
From Assumption \ref{2.42}(1,2), after passing
to a subsequence we can assume that $\lim_{j \rightarrow \infty}
\widehat{u}_{s_j} = \widehat{u}^\infty$ in the weak topology on
$W^{3,p}_{loc}$, for some
$\widehat{u}^\infty \in W^{3,p}_{loc}((0, \infty) \times 
X^{\infty, \infty})$.
From Assumption \ref{2.35} and Assumption \ref{2.42}(3), the gradient
$\nabla \widehat{u}^\infty$ is timelike, and the level sets of
$\widehat{u}^\infty$ have constant mean curvature with respect to
$\widehat{g}^\infty$.

We will apply the monotonicity result of Subsection \ref{A.4}, with
$n = 2$ and $N=1$, to ${\mathcal E}^\infty$,
replacing the $g$ and $G$ of Subsection \ref{A.4} by
$\widehat{g}$ and $\widehat{G}$.
Let $\widehat{h}$ denote the induced metric on the level sets
$\widehat{X}^\infty$ of $\widehat{u}$. Let $\widehat{H}$ denote
the (constant)
mean curvatures of the level sets.
From Subsection \ref{A.4}, we know that
$(- \widehat{H})^2 \vol(\widehat{X}^{\infty},
\widehat{h}(\widehat{u}))$ is nonincreasing in $\widehat{u}$.
It is clearly nonnegative. Since 
$\dim(\widehat{X}^\infty) = 2$,
the expression is invariant under rescaling.
Note that when forming the limit
$\lim_{j \rightarrow \infty} {\mathcal E}^\infty_{s_j} =
{\mathcal E}^{\infty, \infty}$, we are allowed to perform 
$j$-dependent automorphisms of
the flat $1$-dimensional vector bundle on $X^\infty$.
These automorphisms can change
$G$ by a $j$-dependent multiplicative constant, and hence change
$\widehat{g}$ by a multiplicative constant.  One sees that on
a given level set, this does not change 
$(- \widehat{H})^2 \vol(\widehat{X}^{\infty},
\widehat{h})$.

Given $a \in (0, \infty)$, the level set $(\widehat{u}^\infty)^{-1}(a)
\subset (0, \infty) \times X^{\infty, \infty}$ is
the limit of rescalings of level sets $\widehat{u}^{-1}(s_j a)
\subset (0, \infty) \times X^{\infty}$.
It follows that the monotonic quantity
$(- \widehat{H}^\infty)^2 \vol(\widehat{X}^{\infty,\infty},
\widehat{h}^\infty)$ is constant in 
$\widehat{u}^\infty$. By Subsection \ref{A.4}, we conclude that
${\mathcal E}^{\infty, \infty}$ is an Einstein flow of
Bianchi-III flat type.
\end{proof}

Let $\widetilde{X}$ denote the universal cover of $X$. We give it the
pullback Einstein flow.

\begin{corollary} \label{2.48}
Under the hypotheses of Proposition \ref{2.44},
choose $\widetilde{x}^\prime_j \in \widetilde{X}$.
Then $\{
( \widetilde{\mathcal E}_{t^\prime_j}, 
\widetilde{x}^\prime_j) \}_{j=1}^\infty$ 
approaches the
set of Bianchi-III flat Einstein flows
on $\R^3$, in the pointed weak $W^{2,p}$-topology
and the pointed $C^{1,\alpha}$-topology.
\end{corollary}
\begin{proof}
Given Proposition \ref{2.44}, the corollary follows as in
\cite[Section 6.2]{Lott (2010)}.
\end{proof}

\begin{corollary} \label{2.49}
Under the hypotheses of Proposition \ref{2.44},
$X$ has Thurston type $H^2 \times \R$ or
$\widetilde{\SL(2, \R)}$.
\end{corollary}
\begin{proof}
The sequence $\{ (t_j^\prime)^{-2} h(t_j^\prime) \}_{j=1}^\infty$
of Riemannian metrics on $X$ Gromov-Hausdorff converges with bounded
curvature to a
two dimensional compact hyperbolic orbifold,
from which the corollary follows.
\end{proof}

\section{Type-II blowdown} \label{sect3}

Let ${\mathcal E}$ be an expanding CMC Einstein flow
that is not
type-III in the sense of Definition \ref{1.54}.  Then we say that
${\mathcal E}$ is a type-IIb Einstein flow.  One can get information
about such a flow by a rescaling analysis.  The rescaling now involves
the size of the curvature tensor, unlike in the type-III
case where the rescaling involves the Hubble time. 

After rescaling and passing to a limit, one obtains an Einstein
flow ${\mathcal E}^\infty$ on an \'etale groupoid, 
defined for times $t \in \R$, with
vanishing mean curvature.  We show that if the
second fundamental form of the original flow
${\mathcal E}$ is controlled by the
mean curvature, then ${\mathcal E}^\infty$ is the static flow
on a Ricci-flat Riemannian groupoid.  In particular, if
${\mathcal E}$ is locally homogeneous or on a three-dimensional
manifold, then ${\mathcal E}^\infty$ is flat.  This may seem to
contradict the fact that the rescalings in the blowdown procedure 
normalize the
size of the curvature tensor, but the point is that the convergence
to ${\mathcal E}^\infty$ is in the {\em weak} $W^{2,p}$-topology.
Relevant example come from the homogeneous Einstein flows on
$\widetilde{\SL(2, \R)}$ considered in \cite{Ringstrom (2006)}.

More generally, we show that if ${\mathcal E}$ is a type-IIb Einstein
flow on a three dimensional manifold then the
second fundamental form
fails to be controlled by the mean curvature, or
the first covariant derivative of the curvature tensor
fails to be controlled by the curvature norm.

To begin, let ${\mathcal E}$ be a type-IIb Einstein flow on a compact
$n$-dimensional manifold.
Given $t \in [t_0, \infty)$,
let $x_t \in X$ be a point where the time-$t$ curvature norm
$|\Rm|_T$ is maximized.  

\begin{proposition} \label{3.1}
We can find a sequence $\{t_i\}_{i=1}^\infty$ with
$\lim_{i \rightarrow \infty} t_i = \infty$ such that the following
property holds. 
Put $Q_i = |\Rm|_T(x_i, t_i)$ and
${\mathcal E}^{(i)}(u) = {\mathcal E}_{Q_i^{- \: \frac12}}
(u+Q_i^{\frac12}t_i)$.
Then after passing to a subsequence, there is a limit
$\lim_{i \rightarrow \infty} ({\mathcal E}^{(i)}, x_i) =
({\mathcal E}^\infty, x_\infty)$ in the pointed weak $W^{2,p}$-topology and the
pointed $C^{1,\alpha}$-topology.  Here ${\mathcal E}^\infty$ is an Einstein
flow on an $n$-dimensional \'etale groupoid, defined for
$t \in \R$. If there is some $C < \infty$ such that
$|K|^2 \le C H^2$ then $L^\infty$ is uniformly bounded below by a 
positive constant.
\end{proposition}
\begin{proof}
As in \cite[Chapter 8.2.1.3]{Chow-Lu-Ni (2006)}, 
we can make an initial choice of the $t_i$'s so that 
$\lim_{i \rightarrow \infty} Q_i t_i^2 = \infty$ and
for any
compact time interval $S \subset \R$, there are bounds on
$|\Rm|_T$ on $S$ for the rescaled flows ${\{\mathcal E}^{(i)}\}_{i=1}^\infty$
that are uniform in $i$. This implies uniform bounds on $S$ for $|K^{(i)}|$ 
\cite[Proposition 2.2]{Anderson (2001)}.
The rescaled Einstein flow ${\mathcal E}^{(i)}$ has
\begin{align} \label{3.2}
|\Rm^{(i)}|_T(u) \: = \: & Q_i^{-1} |\Rm|_T(Q_i^{- \: \frac12}u+t_i), \\
|K^{(i)}|(u) \: = \:  & 
Q_i^{- \: \frac12} |K|(Q_i^{- \: \frac12}u+t_i), \notag \\
H^{(i)}(u) \: = \:  & 
- \: \frac{nQ_i^{- \: \frac12}}{Q_i^{- \: \frac12}u+t_i}. \notag
\end{align}
From Proposition \ref{2.2}, after passing to a subsequence there is a limit
$\lim_{i \rightarrow \infty} ({\mathcal E}^{(i)}, x_i) =
({\mathcal E}^\infty, x_\infty)$ as stated.
If $|K(t)|^2 \le C H(t)^2 = C \frac{n^2}{t^2}$ then
\begin{equation} \label{3.3}
|K^{(i)}|^2(u) = Q_i^{-1} |K|^2(Q_i^{- \: \frac12}u+t_i) \le C Q_i^{-1} 
\frac{n^2}{(Q_i^{- \: \frac12}u+t_i)^2} = 
\frac{C}{n} \frac{\partial H^{(i)}(u)}{\partial u}.
\end{equation}
From Proposition \ref{2.2}, the lapse function $L^\infty$ is positive.
As in the proof of Corollary \ref{1.55},
it is uniformly bounded below by a positive constant.
(It is bounded above by one.)
\end{proof}

As noted in
\cite[Section 5]{Anderson (2001)}, because of the renormalization,
the flow ${\mathcal E}^\infty$ has vanishing mean curvature $H^\infty$,
since $\lim_{i \rightarrow \infty} H^{(i)}(u) = - \lim_{i \rightarrow \infty}
\frac{n}{u + Q_i^{\frac12} t_i} = 0$.

\begin{proposition} \label{3.4}
If a type-IIb expanding CMC Einstein flow 
${\mathcal E}$ has a uniform upper bound on
$\frac{|K|^2}{H^2}$ then
the blowdown
limit ${\mathcal E}^\infty$ is a static Einstein flow on a Ricci-flat 
Riemannian
groupoid.
\end{proposition}
\begin{proof}
Because of the rescaling, the blowdown
limit has vanishing second fundamental form $K^\infty$.
Hence ${\mathcal E}^\infty$ is
a static Einstein flow on a Riemannian groupoid 
$({\mathcal X}^\infty, h^\infty)$.
The static Einstein flow equations
become $L^\infty R^\infty_{ij} = L^\infty_{;ij}$ and 
$\triangle_{h^\infty} L^\infty = 0$.
In the smooth structure on the unit space of ${\mathcal E}^\infty$
coming from local harmonic coordinates, by elliptic regularity 
the metric $h^\infty$ is smooth
and $L^\infty$ is smooth.
We use the trick from \cite[Appendix]{Anderson (1999)} of passing
to ${\mathcal Y}^\infty = {\mathcal X}^\infty \times S^1$ 
with the Riemannian metric
$h^\infty + (L^\infty)^2 d\theta^2$, which is Ricci-flat.
The function $\log L^\infty$ is a bounded harmonic function on
${\mathcal Y}^\infty$.  The proof of
\cite[Corollary 1]{Yau (1975)} extends to the groupoid setting to show that
$L^\infty$ is constant. Then $h^\infty$ is Ricci-flat.
This shows that
${\mathcal E}^\infty$ is a static Einstein flow on a Ricci-flat
Riemannian groupoid ${\mathcal X}^\infty$, thereby proving the proposition.
\end{proof}

\begin{corollary} \label{3.5}
Under the hypotheses of Proposition \ref{3.4}, if $X$ is locally
homogeneous then the type-IIb blowdown
limit ${\mathcal E}^\infty$ is a static Einstein flow on a flat 
Riemannian groupoid.
\end{corollary}
\begin{proof}
A Gromov-Hausdorff limit of homogeneous spaces is still 
homogeneous \cite[p. 66]{Gromov (1981)}.
Applying this to the balls in the tangent spaces,
it follows that $(X^\infty, h^\infty(t))$ is a locally homogeneous
Ricci-flat Riemannian groupoid, and hence is flat 
\cite{Spiro (1993)}.
\end{proof}

\begin{corollary} \label{3.6}
Under the hypotheses of Proposition \ref{3.4}, if $n=3$
then the type-IIb blowdown
limit ${\mathcal E}^\infty$ is a static Einstein flow on a flat 
Riemannian groupoid.
\end{corollary}
\begin{proof}
A Ricci-flat three dimensional Riemannian groupoid is flat.
\end{proof}

It may seem contradictory that although we rescale so that the
norm of the curvature tensor at $(x_i, t_i)$ is one, the
limit flow is flat.  The point is that the convergence to the
limit flow is in the
{\em weak} $W^{2,p}$-topology, which does not imply
pointwise convergence of the curvature norm. In effect, there are 
increasing fluctuations
of the curvature tensor, which average it out to zero.

Under the hypotheses of Corollary \ref{3.6},
one does have pointed $C^{1,\alpha}$-convergence
of the normalized Einstein flows to the flat limit flow. 
In particular, put
$\widehat{h}_i = Q_i \exp_{x_i}^* h(t_i)$, a metric defined at
least on $B_i =B( 0, \pi Q_i^{- \: \frac12} ) \subset T_{x_i}X$.
Then the pointed balls $(B_i, x_i, \widehat{h}_i)$ converge in the
sense of distance geometry, i.e. in the pointed Gromov-Hausdorff topology,
to the flat Euclidean metric on a three dimensional
ball of radius $\pi$. However, their curvature
tensors do not converge.

\begin{corollary} \label{3.7}
When $n=3$, if there is some $C<\infty$ so that
an expanding CMC Einstein flow has $|K|^2 \le CH^2$ and
$|\nabla \Rm|_T(x,t) \le C \sup_{y \in X} |Rm|_T^{\frac32}(y,t)$ 
for all $x \in X$ and $t \in [t_0, \infty)$, 
then the flow must be type-III.
\end{corollary}
\begin{proof}
If ${\mathcal E}$ is not type-III 
then Corollary \ref{3.6} applies. From the bound on the 
normalized covariant
derivative of the curvature tensor, we  
have convergence of the normalized Einstein flows to ${\mathcal E}^\infty$ 
in the pointed weak $W^{3,p}$-topology. This implies pointwise 
convergence of the curvature tensors. The
normalized curvature tensors of ${\mathcal E}$
have norm $1$ at $(x_i, t_i)$, but
converge to the vanishing curvature of ${\mathcal E}^\infty$ at
$(x_\infty, 0)$, which is a contradiction.
\end{proof}

\begin{remark} \label{3.8}
Corollary \ref{3.7} can be proven by just working on balls in tangent
spaces, instead of dealing with \'etale groupoids.
\end{remark}

\begin{example} \label{3.9}
An example of a type-IIb Einstein flow was given in
\cite{Ringstrom (2006)}. Consider $\widetilde{\SL(2, \R)}$ with a
left-invariant Riemannian metric $\widetilde{h}(0)$.  
Let $\Gamma$ be a cocompact
lattice in $\widetilde{\SL(2, \R)}$. Let $h(0)$ be the quotient metric
on $X = \Gamma \backslash \widetilde{\SL(2, \R)}$. 
Let $\widetilde{K}(0)$ be a left-invariant symmetric covariant $2$-tensor field
on  $\widetilde{\SL(2, \R)}$. Let $K(0)$ be the quotient $2$-tensor
field on $X$. Let ${\mathcal E}$ be the ensuing Einstein flow on $X$,
with initial conditions $(h(0), K(0))$. 

Let $\R \subset \widetilde{\SL(2, \R)}$ be the lift of
$\SO(2) \subset {\SL(2, \R)}$. If $(h(0), K(0))$ is right-$\R$ invariant
then ${\mathcal E}$ is type-III.  Otherwise, it is type-IIb
\cite{Ringstrom (2006)}.

In the latter case, we claim that Corollary \ref{3.6} applies.
This follows from  results in \cite[Pf. of Theorem 3]{Ringstrom (2006)}.
In the notation there, the normalized traceless part $\frac{K^0}{H}$ of the
second fundamental form is determined by $\Sigma_{\pm}$. It is shown that
$\Sigma_{\pm}$ are uniformly bounded in $t$. Hence the blowdown Einstein
flow is the static flow on a flat Riemannian groupoid.

We claim that this limit groupoid is $\R^2 \times (\R \rtimes \R_\delta)$,
where $\R_\delta$ denotes $\R$ with the discrete topology.
From \cite[Theorem 3]{Ringstrom (2003)}, we can write
$h(t) = \sum_{i=1}^3 a_i^2(t) \xi^i \otimes \xi^i$ with
$a_1(t) \sim \alpha_1 (\ln t)^{\frac12}$ and $a_i(t) \sim \alpha_i t$
for $i \in \{2,3\}$. Here $\alpha_1,
\alpha_2, \alpha_3 > 0$. From 
\cite[Theorem 3]{Ringstrom (2006)}, we have
$|\Rm|_T(t) \sim \frac{c_0}{t \ln t}$ for some $c_0 > 0$.
Hence the normalized lengths are
comparable to 
$(t \ln t)^{- \: \frac12} (\ln t)^{\frac12}$ in the $1$-direction, and
$(t \ln t)^{- \: \frac12} t$ in the $2$ and $3$ directions.
That is, there are two expanding directions and one shrinking
direction. Then the limit flat \'etale groupoid must be
$\R^2 \times (\R \rtimes \R_\delta)$.

The fact that there are increasing fluctuations of the curvature
tensor, which cause its averaging out to zero, is consistent
with the nonuniform behavior shown in 
\cite[Theorem 3 and Proposition 2]{Ringstrom (2006)}.
\end{example}

\appendix

\section{Monotonicity formulas} \label{sectA}

In this section we derive monotonicity formulas for
dimensionally reduced Einstein flows.  We consider
a coupled system on a connected compact manifold $B$, where the fields on $B$
are (locally) a Lorentzian metric $g$, an $\R^N$-valued connection $A$ and
a map $G$ to positive definite $(N \times N)$-matrices.
Such coupled systems arise, for example, when doing dimensional reduction of
the vacuum Einstein equation on a manifold $M$ with a free $T^N$-action,
and quotient space $B$. 
The vacuum Einstein equation on $M$
becomes a coupled system consisting of a nonvacuum Einstein
equation for $g$, a Yang-Mills-type equation for $A$ and a
wave-type equation for $G$.
In Subsection \ref{subsectA.1} we write the equations and begin
their analysis.

Subsections \ref{subsectA.2} and \ref{subsectA.3} are concerned with 
monotonicity formulas when there is a spacetime
decomposition for which $\det(G)$ is constant along spatial
hypersurfaces. In \cite{Berger-Chrusciel-Isenberg-Moncrief (1997)}
and related papers, 
such a spacetime decomposition is said to provide
``areal'' coordinates. It is especially relevant when $\dim(B) = 2$.

Subsection \ref{subsectA.4} deals with a monotonic quantity that
exists when $\dim(B) > 2$. In the case when $\dim(B) = 3$ and $N = 1$, it
reduces to the ``first energy'' of Choquet-Bruhat 
\cite{Choquet-Bruhat (2004)}
and Choquet-Bruhat-Moncrief \cite{Choquet-Bruhat-Moncrief (2001)}.

More detailed descriptions are given at the beginnings of the subsections.

The results of this appendix extend in a straightforward way to the
setting where $B$ is an orbifold.  In the appendix, we only consider
the case when $B$ is a manifold.

\subsection{Curvature formulas under an $\R^N$-symmetry} 
\label{subsectA.1}

We begin with the geometric setup of
\cite[Section 4.1]{Lott (2010)}, to which we refer for more details.
Let ${\mathcal G}$ be an $N$-dimensional abelian Lie group, with Lie algebra
${\frak g}$.
Let ${\frak E}$ 
be a local system on $B$ of Lie groups isomorphic to ${\mathcal G}$.
There is a corresponding flat ${\frak g}$-vector bundle $e$ on $B$;
see \cite[Section 4.1]{Lott (2010)}.

Let $M$ be the total space of an 
${\frak E}$-twisted principal ${\mathcal G}$-bundle
with base $B$, in the sense of \cite[Section 4.1]{Lott (2010)}.
(An example is when ${\frak E}$ is the constant local system and
$M$ is the total space of a $T^N$-bundle on $B$.)
We write $\dim(B) = n+1$ and $\dim(M) = m = N+n+1$.

Let
$\overline{g}$ be a Lorentzian metric on $M$ with a
free local isometric ${\frak E}$-action. We assume that the
induced metrics on the ${\frak E}$-orbits are Riemannian.
In adapted coordinates, we can write
\begin{equation} \label{A.1} 
\overline{g} \: = \: \sum_{I,J=1}^N G_{IJ} \: (dx^I + A^I) (dx^J + A^J) \: + \:
\sum_{\alpha, \beta = 1}^{n+1} g_{\alpha \beta} \: db^\alpha db^\beta.
\end{equation}
Here
$G_{IJ}$ is the local expression of a Euclidean inner product on
$e$,
$\sum_{\alpha, \beta = 1}^{n+1} g_{\alpha \beta} \: db^\alpha db^\beta$ is
the local expression of a Lorentzian metric $g_B$ on $B$ and
$A^I = \sum_{\alpha} A^I_\alpha db^\alpha$ are the components of
a local $e$-valued $1$-form describing an connection $A$ on the
twisted ${\frak G}$-bundle $M \rightarrow B$.

Put $F^I_{\alpha \beta} = \partial_\alpha A^I_\beta - 
\partial_\beta A^I_\alpha$.
At a given point $b \in B$, we can assume that $A^I(b) = 0$.
We write
\begin{equation} \label{A.2}
G_{IJ;\alpha \beta} \: = \: G_{IJ;\alpha \beta} \: - \:
\Gamma^{\sigma}_{\: \: \alpha \beta} \: G_{IJ, \sigma},
\end{equation}
where $\{\Gamma^{\sigma}_{\: \: \alpha \beta}\}$ are the
Christoffel symbols for the metric $g_{\alpha \beta}$ on $B$.

From \cite[Section 4.2]{Lott (2010)},
the Ricci tensor of ${\overline g}$ on $M$
is given in terms of the curvature tensor
$R_{\alpha \beta \gamma \delta}$ of $B$, the $2$-forms $F^I_{\alpha \beta}$
and the metrics $G_{IJ}$ by
\begin{align} \label{A.3} 
\overline{R}_{IJ}^{\overline{g}} 
\:  =  \: & - \: \frac12 \: g^{\alpha \beta} \: 
G_{IJ; \alpha \beta} \: - \: \frac14 \: g^{\alpha \beta} \:
G^{KL} \: G_{KL, \alpha} \: G_{IJ, \beta} \: + \:
\frac12 \: g^{\alpha \beta} \: G^{KL} \: G_{IK, \alpha} \:
G_{LJ, \beta} \: + \\
&  \frac14 \: g^{\alpha \gamma} \: g^{\beta \delta} \:
G_{IK} \: G_{JL} \: F^K_{\alpha \beta} \: F^L_{\gamma \delta} \notag \\
\overline{R}_{I \alpha}^{\overline{g}}
 \:  =  \: & \frac12 \: g^{\gamma \delta} \:
G_{IK} \: F^K_{\alpha \gamma; \delta} \: + \: \frac12 \: g^{\gamma \delta} \:
G_{IK, \gamma} \: F^K_{\alpha \delta} \: + \: \frac14 \:
g^{\gamma \delta} \: G_{Im} \: G^{KL} \: G_{KL, \gamma} \: F^m_{\alpha \delta} 
\notag \\
\overline{R}_{\alpha \beta}^{\overline{g}} 
\:  =  \: & R_{\alpha \beta}^g \: - \: 
\frac12 \: G^{IJ} \: G_{IJ; \alpha \beta} \: + \: \frac14 \:
G^{IJ} \: G_{JK,\alpha} \: G^{KL} \: G_{LI,\beta} \: - \:
\frac12 \: g^{\gamma \delta} \: \: G_{IJ} \: F^I_{\alpha \gamma} \:
F^J_{\beta \delta}. \notag
\end{align}
The scalar curvature is
\begin{align} \label{A.4}
\overline{R}^{\overline{g}} \: = \: & 
R^g \: - \: g^{\alpha \beta} G^{IJ}  \: G_{IJ; \alpha \beta} \: + \: 
\frac34 \: g^{\alpha \beta} \: G^{IJ} \: G_{JK, \alpha} \: G^{KL} \:
G_{LI, \beta} \\
& \: - \: \frac14 \: g^{\alpha \beta} \: G^{IJ} \: 
G_{IJ, \alpha} \: G^{KL} \: G_{KL, \beta} \: - 
 \: \frac14 \:
g^{\alpha \gamma} \: g^{\beta \delta} \: G_{IJ} \:
F^I_{\alpha \beta} \: F^J_{\gamma \delta}. \notag
\end{align}

In what follows we will assume that the flat vector bundle $e$ has
holonomy in $\SL(N, \R)$, so that $\ln \det G$ is globally defined
on $B$. We have
\begin{equation} \label{A.5}
\nabla_\alpha \ln \det G = G^{IJ} G_{IJ, \alpha}
\end{equation}
and
\begin{equation} \label{A.6}
\triangle_g \ln \det G = 
g^{\alpha \beta} G^{IJ} G_{IJ; \alpha \beta} - 
g^{\alpha \beta} G^{IJ} G_{JK, \alpha} G^{KL} G_{LK, \beta}.
\end{equation}
Writing
\begin{equation} \label{A.7}
|F|^2 = G_{IJ} g^{\alpha \beta} g^{\gamma \delta} F^I_{\alpha \gamma}
F^J_{\beta \delta},
\end{equation}
the first equation in (\ref{A.3}) gives
\begin{equation} \label{A.8}
G^{IJ} \overline{R}_{IJ} =
- \frac12 \triangle_g \ln \det G - \frac14 g^{\alpha \beta}
(\nabla_\alpha \ln \det G) (\nabla_\beta \ln \det G) + \frac14 |F|^2.
\end{equation}
Note that $|F|^2$ need not be
nonnegative.

Given a foliation of $B$ by
compact spacelike hypersurfaces $Y$, we can write the metric $g$ on $B$ as
\begin{equation} \label{A.9}
g = - L^2 dt^2 + \sum_{i,j=1}^n h_{ij} dy^i dy^j.
\end{equation}
Here $L = L(y,t)$ is the lapse function and we have performed
spatial diffeomorphisms to kill the shift vectors.

\subsection{Monotonicity formulas for equivolume foliations}
\label{subsectA.2}

In this subsection we introduce a first monotonicity formula for
equivolume foliations.
Suppose that $\det G$ is spatially constant, i.e. only depends on $t$.
Then 
\begin{equation} \label{A.10}
g^{\alpha \beta}
(\nabla_\alpha \ln \det G) (\nabla_\beta \ln \det G) =
- \: L^{-2} ( \partial_t \ln \det G )^2
\end{equation}
and
\begin{equation} \label{A.11}
\triangle_g \ln \det G = 
- \: \frac{1}{L \sqrt{\det h}} \partial_t
\left( L^{-1} \sqrt{\det h} (\partial_t \ln \det G) \right).
\end{equation}

If $\overline{R}^{\overline{g}}_{IJ} = 0$ then (\ref{A.8}) becomes
\begin{equation} \label{A.12}
0 = \frac12 \frac{1}{L \sqrt{\det h}} \partial_t
\left( L^{-1} \sqrt{\det h} (\partial_t \ln \det G) \right) +
\frac14 L^{-2} ( \partial_t \ln \det G )^2 + \frac14 |F|^2.
\end{equation}
Multiplying by $L \sqrt{\det h}$ and integrating over $Y$ gives
\begin{align} \label{A.13}
\frac{\partial}{\partial t} \left( (\partial_t \ln \det G) \int_Y
L^{-1} \dvol_Y  \right) = & - \: \frac12
(\partial_t \ln \det G)^2 \int_Y
L^{-1} \dvol_Y \\
& - \: \frac12 \: \int_Y
|F|^2 L \dvol_Y. \notag
\end{align}
If $F = 0$ then $(\partial_t \ln \det G) \int_Y
L^{-1} \dvol_Y$ is monotonically nonincreasing in $t$.

\subsection{Two dimensions} \label{subsectA.3}

In this subsection we specialize to the case when $\dim(B) = 2$.
We begin with some generalities.  In Subsubsection 
\ref{subsubsectA.3.1} we
consider monotonic quantities in the case $F=0$. Besides the
monotonic quantity of Subsection \ref{subsectA.2}, we analyze
an energy-like monotonic functional $\widehat{\mathcal E}$. 

In Subsubsection \ref{subsubsectA.3.2} we look at the case when $F$ is nonzero.
In order to apply results from the literature, in that subsubsection we
specialize to the case $N=2$. We introduce the monotonic
quantity $\widehat{\mathcal E}_K$ and show that it is well-defined
no matter what the global twisting $H \in \SL(2, \R)$ may be.
We characterize when $\widehat{\mathcal E}_K$ is constant in $t$.

Continuing with Subsection \ref{subsectA.2},
suppose that $\dim(B) = 2$, i.e. $\dim(Y) = 1$.
We write $g$ locally (in $Y$) as $- L^2 dt^2 + h dy^2$.
We have
$R^g_{\alpha \beta} = \frac12 R g_{\alpha \beta}$, so
$g^{tt} R^g_{tt} = g^{yy} R^g_{yy}$. Hence
$- L^{-2} R^g_{tt} = h^{-1} R^g_{yy}$.
If $\overline{R}^{\overline{g}}_{\alpha \beta} = 0$ then the
third equation of (\ref{A.3}) gives
\begin{align} \label{A.14}
& L^{-2} \Tr \left( G^{- \: \frac12} G_{,t} G^{- \: \frac12} \right)^2 +
h^{-1} \Tr \left( G^{- \: \frac12} G_{,y} G^{- \: \frac12} \right)^2 = \\
& L^{-2} \Tr \left( G^{- \: 1} G_{,t} \right)^2 +
h^{-1} \Tr \left( G^{- \: 1} G_{,y}  \right)^2 = \notag \\
& - 2 L^{-2} (\ln \det G)_{;tt} =  
- 2 L^{-2} (\ln \det G)_{tt} + 2 L^{-3} L_t (\ln \det G)_t. \notag
\end{align}
If in addition $(\ln \det G)_t = 0$ then 
from (\ref{A.14}), 
$G^{- \: \frac12} G_{,t} G^{- \: \frac12}$ and
$G^{- \: \frac12} G_{,y} G^{- \: \frac12}$ vanish, so $G$ is locally
constant in $y$ and $t$. Then the
third equation of (\ref{A.3}) gives $R^g_{\alpha \beta} = 0$, so $B$ is flat.
The holonomy around $Y$ of the flat vector bundle $e$ must be orthogonal.

\subsubsection{Gowdy spacetime} \label{subsubsectA.3.1}
 
In this subsubsection we assume that $F=0$.
From (\ref{A.13}),
$(\partial_t \ln \det G) \int_Y
L^{-1} \dvol_Y$ is monotonically nonincreasing in $t$.
If it is constant in $t$
then the right-hand side of (\ref{A.13}) vanishes, so
$\partial_t \ln \det G = 0$. Hence $G$ is locally constant in $y$ and $t$, and
$B$ is flat.

For another monotonic quantity, consider
\begin{align} \label{A.15}
{\mathcal E}(t) = & \int_{Y} \left[ h^{-1} \Tr \left( \left( G^{-1} 
\frac{\partial G}{\partial y} \right)^2 \right) + L^{-2}
\Tr \left( \left( G^{-1} 
\frac{\partial G}{\partial t} \right)^2 \right) \right] L \dvol \\
= & \int_{Y} \left[ L h^{- \: \frac12} \Tr \left( \left( G^{-1} 
\frac{\partial G}{\partial y} \right)^2 \right) + L^{-1} h^{\frac12}
\Tr \left( \left( G^{-1} 
\frac{\partial G}{\partial t} \right)^2 \right) \right] dy. \notag
\end{align}
Still assuming that $F=0$, equation (\ref{A.12}) gives
\begin{equation} \label{A.16}
(\ln \det G)_t \partial_t (L h^{- \: \frac12}) =
L h^{- \: \frac12} ((\ln \det G)_{tt} + \frac12 (\ln \det G)_t^2 ). 
\end{equation}
When $\overline{R}^{\overline{g}}_{IJ} = 0$, 
equation (\ref{A.3}) gives the matrix equation
\begin{align} \label{A.17}
& - L^{-2} (G^{-1} G_{tt} - G^{-1} G_t G^{-1} G_t) +
h^{-1} (G^{-1} G_{yy} - G^{-1} G_y G^{-1} G_y) + \\
& L^{-3} L_t G^{-1} G_t + L^{-1} h^{-1} L_y G^{-1} G_y
- \frac12 L^{-2} h^{-1} h_t G^{-1} G_t - \notag \\
& \frac12
h^{-2} h_y G^{-1} G_y - \frac12 L^{-2} (\ln \det G)_t G^{-1} G_t = 0. \notag
\end{align}
Using (\ref{A.16}) and (\ref{A.17}), one finds
\begin{align} \label{A.18}
\frac{d}{dt} \left( (\ln \det G)_t {\mathcal E} \right) =
& \left( 2 (\ln \det G)_{tt} + \frac12 (\ln \det G)_t^2 \right) 
{\mathcal E} - \\
& \frac12 (\ln \det G)_t^2 \int_{Y} L^{-1} \Tr \left( 
\left( G^{-1} G_t \right)^2 \right) \dvol. \notag
\end{align}

If $(\ln \det G)_t \neq 0$ then
a scale invariant quantity is given by
\begin{equation} \label{A.19}
\widehat{{\mathcal E}}(t) = \frac{2}{(\ln \det G)_t \sqrt{\det G}} 
{\mathcal E}(t).
\end{equation}
Using (\ref{A.18}), one finds
\begin{equation} \label{A.20}
\frac{d\widehat{{\mathcal E}}}{dt} \: = \: - \: \frac{1}{\sqrt{\det G}}
\int_{Y} L^{-1} \Tr \left( 
\left( G^{-1} G_t \right)^2 \right) \dvol
\end{equation}
If the right-hand side of (\ref{A.20}) vanishes then $G$ is constant in $t$.
As before, this implies that $G$ is constant in $y$ and $t$, and
$B$ is flat.

\begin{remark} \label{A.21}
If we use the areal time variable $t = \sqrt{\det G}$ then
$\widehat{\mathcal E}(t) = {\mathcal E}(t)$ and
\begin{equation}
\frac{d\widehat{{\mathcal E}}}{dt} \: = \: - \: \frac{1}{t}
\int_{Y} L^{-1} \Tr \left( 
\left( G^{-1} G_t \right)^2 \right) \dvol;
\end{equation}
compare with (\ref{A.26}) and (\ref{A.27}).
\end{remark}

\subsubsection{NonGowdy spacetime} \label{subsubsectA.3.2}

We now assume that $F \neq 0$. 
If $\overline{R}_{I \alpha}^{\overline{g}} = 0 $ then 
from the second equation in (\ref{A.3}), one finds that the
$\R^N$-valued vector 
\begin{equation} \label{A.23}
C_I = L^{-1} h^{- \: \frac12} \: \sqrt{\det G} \: G_{IK} F^K_{ty}
\end{equation}
is locally constant on the two dimensional spacetime.  More
precisely, it is a locally constant section of the flat vector bundle
$e^*$ (using our assumption that $e$ is unimodular).

We now restrict to the case when $N = 2$ and the flat $\R^2$-bundle $e$
has holonomy $H$, around the circle $Y$, lying in $\SL(2, \R)$.  
When $H = \Id$, the components of $C$ are 
called the ``twist quantities'' in
\cite{Berger-Chrusciel-Isenberg-Moncrief (1997)} and subsequent papers
such as \cite{LeFloch-Smulevici (2015)}. 
We mostly follow the notation of \cite[p. 1256-1283]{LeFloch-Smulevici (2015)},
with coordinates $(R, \theta)$ for the two dimensional base. We use
linear coordinates
$x^1, x^2$ for the $\R^2$-fiber.  In that paper, $R = \det G$ and $\theta$
is the coordinate for the spacelike hypersurface $Y$.  The coordinates
$x^1$ and $x^2$ are chosen so that $C_1 = 0$ and $C_2 = K$, where $K$ is
a constant. The Lorentzian metric on
$(0, \infty) \times Y$ can be written as
\begin{equation} \label{A.24}
g = e^{2(\eta - U)} (-dR^2 + a^{-2} d\theta^2) +
e^{2U} (dx^1 + Adx^2 + (G+AH) d\theta)^2 +
e^{-2U}R^2(dx^2+Hd\theta)^2.
\end{equation}
Put
\begin{equation} \label{A.25}
{\mathcal D} = 
a^{-1} U_R^2 + a U_\theta^2 + R^{-2} e^{4U} 
(a^{-1} A_R^2 + a A_\theta^2)
\end{equation}
and 
\begin{equation} \label{A.26}
\widehat{\mathcal E}_K(R) = \int_{Y} \left( {\mathcal D}
+ \frac14 K^2 R^{-4} e^{2\eta} a^{-1} \right) d\theta. 
\end{equation}
Then from \cite[p. 1283]{LeFloch-Smulevici (2015)}
\begin{equation} \label{A.27}
\frac{d\widehat{\mathcal E}_K}{dR} = 
- 2 R^{-1} \int_{Y} \left( a^{-1} U_R^2 + \frac14
R^{-2} e^{4U} a A_\theta^2 \right) d\theta - \frac12 K^2 R^{-3}
\int_{Y} {\mathcal D} e^{2 \eta} \: d\theta.
\end{equation}

If $t$ is a time variable, with $R$ a monotonically increasing function of $t$,
then
\begin{equation} \label{A.28}
\frac{d\widehat{\mathcal E}_K}{dt} = \frac{(\det G)_t}{2 \sqrt{\det G}}
\frac{d\widehat{\mathcal E}_K}{dR}.
\end{equation}
The quantity $\widehat{\mathcal E}_K$ is scale invariant.

To treat the more general case when $H \in \SL(2, \R)$, since
$C$ is a nonzero flat section of $e^*$, the matrix $H^{-T}$ must
be unipotent, i.e. conjugate to 
$\begin{pmatrix}
1 & c \\
0 & 1
\end{pmatrix}$.
The local coordinates $\{x^1, x^2\}$ are such that
$C_1 = 0$ and $C_2 \neq 0$. We claim that
the formula for $\widehat{\mathcal E}_K$ still makes sense.  To see this,
the result of parallel transport around $Y$ is
$x^1 \rightarrow x^1 +cx^2$ and $x^2 \rightarrow x^2$. In terms of the
metric (\ref{A.24}), this is the same as
$\eta \rightarrow \eta$,
$U \rightarrow U$, 
$a \rightarrow a$,
$A \rightarrow A + c$, 
$G \rightarrow G - cH$, 
$H \rightarrow H$ and
$K \rightarrow K$. 
One sees that the integrand of (\ref{A.26}) is preserved under these changes.
Hence the formula for $\widehat{\mathcal E}_K$ makes sense and
(\ref{A.27}) still holds.

Now suppose that $\widehat{\mathcal E}_K$ is constant in $t$.
From (\ref{A.27}) and (\ref{A.28}), if $\det G$ is not constant in $t$
(in which case $G$ is locally constant in $y$ and $t$ and $B$ is flat)
then $U$ and $A$ are constant in $y$ and $R$.  Using the
equations in 
\cite[Proposition 4.4]{LeFloch-Smulevici (2015)}, one finds that
the Lorentzian metric on $(0, \infty) \times Y$ is a constant times
\begin{equation} \label{A.29}
- \: \frac{R^2}{R^2 - CK^2} dR^2 +
\frac{1}{R^2} (R^2 - CK^2) e^{-2\sigma(\theta)} d\theta^2,
\end{equation}
where $C$ is a constant and $\sigma : Y \rightarrow \R$ is arbitrary.
If (\ref{A.29}) admits a future timelike curve along which the proper
time goes to infinity then
$R$ must range over an interval $[R_0, \infty)$. The length of the
$S^1$-fiber is bounded as $R \rightarrow \infty$.

\begin{remark} \label{A.30}
The second equation below 
\cite[(4.26)]{LeFloch-Smulevici (2015)} should read
$F := 2U_R U_\theta + 2 R^{-2} e^{4U} A_R A_\theta$.
\end{remark}

\subsection{Monotonicity of reduced volume} \label{subsectA.4}

In this subsection, we consider monotonic quantities when $\dim(B) > 2$.
As in 
Choquet-Bruhat 
\cite{Choquet-Bruhat (2004)}
and Choquet-Bruhat-Moncrief \cite{Choquet-Bruhat-Moncrief (2001)},
we make an appropriate conformal transformation of the Lorentzian metric
on $B$ and assume that the new metric has an expanding CMC foliation.
It turns out that
the normalized volume of the time slice is monotonically nonincreasing.

To simplify the calculations, we start with a Lorentzian metric 
$\overline{g}$ of the form (\ref{A.1}) and consider a conformally related
metric $\overline{h} = e^{2 \overline{\phi}} \overline{g}$, where 
$\overline{\phi}$ pulls back from $B$.  We impose the vacuum Einstein
equations on $\overline{h}$. With an appropriate choice of $\overline{\phi}$,
the monotonic quantity is derived from the geometry of $(B, g)$.

The papers \cite{Choquet-Bruhat (2004)}
and \cite{Choquet-Bruhat-Moncrief (2001)} deal with the case $N=2$.
The space of inner products $G$ on $\R^2$ is isomorphic to 
$\R^+ \times H^2$, which gives the link between the present paper 
and the formalism of
\cite{Choquet-Bruhat (2004)} and \cite{Choquet-Bruhat-Moncrief (2001)}.

The monotonic quantity in this section is only defined when 
$\dim(B) > 2$. If $\dim(B) = 2$ then the formula for $\overline{\phi}$
is such that $\overline{h}$ would necessarily have a constant volume
density on its $\R^N$-fibers, which need not be the case.

To begin, we consider the effect of a conformal change on an
arbitrary Lorentzian metric $\overline{g}$ on $M$.
Given 
$\overline{\phi} \in C^\infty(M)$,
put $\overline{h} = e^{2 \overline{\phi}} \overline{g}$. 
Then the Ricci curvature
of $\overline{h}$ is given by
\begin{equation} \label{A.31}
\overline{R}^{\overline{h}}_{ab} =
\overline{R}^{\overline{g}}_{ab} - (m-2) \overline{\phi}_{;ab} +
(m-2) \overline{\phi}_{,a} \overline{\phi}_{,b} - 
(\triangle_{\overline{g}} \overline{\phi} + (m-2) 
|\nabla \overline{\phi}|_{\overline{g}}^2
) \: \overline{g}_{ab}.
\end{equation}

We now assume that $\overline{g}$ is of the form (\ref{A.1}).
Given $\phi \in C^\infty(B)$, put $\overline{\phi} = \pi^* \phi
\in C^\infty(M)$. Then on a fiber $\pi^{-1}(b)$,
\begin{align} \label{A.32}
\overline{\phi}^{\overline{g}}_{;IJ} =  \: & \frac12 \langle \nabla G_{IJ},
\nabla \phi \rangle,  \\
\overline{\phi}^{\overline{g}}_{;I\alpha} = \: & 0 \notag \\
\overline{\phi}^{\overline{g}}_{;\alpha \beta} = \: & 
{\phi}^g_{;\alpha \beta}  \notag \\
\triangle_{\overline{g}} \overline{\phi} = & \triangle_g \phi +
\frac12 \langle \nabla \ln \det G, \nabla \phi \rangle. \notag
\end{align}
Combining (\ref{A.3}), (\ref{A.31}) and (\ref{A.32}) gives
\begin{align} \label{A.33} 
\overline{R}_{IJ}^{\overline{h}} 
\:  =  \: & - \: \frac12 \: g^{\alpha \beta} \: 
G_{IJ; \alpha \beta} \: - \: \frac14 \: g^{\alpha \beta} \:
G^{KL} \: G_{KL, \alpha} \: G_{IJ, \beta} \: + \:
\frac12 \: g^{\alpha \beta} \: G^{KL} \: G_{IK, \alpha} \:
G_{LJ, \beta} \: + \\
&  \frac14 \: g^{\alpha \gamma} \: g^{\beta \delta} \:
G_{IK} \: G_{JL} \: F^K_{\alpha \beta} \: F^L_{\gamma \delta} 
- \frac12 (n+N-1) \langle \nabla G_{IJ}, \nabla \phi \rangle_g - \notag \\
& \left( \triangle_g \phi + \frac12 
\langle \nabla \ln \det G, \nabla \phi \rangle_g
+ (n+N-1) |\nabla \phi|_g^2 \right) G_{IJ}
\notag \\
\overline{R}_{I \alpha}^{\overline{h}}
 \:  =  \: & \frac12 \: g^{\gamma \delta} \:
G_{IK} \: F^K_{\alpha \gamma; \delta} \: + \: \frac12 \: g^{\gamma \delta} \:
G_{IK; \gamma} \: F^K_{\alpha \delta} \: + \: \frac14 \:
g^{\gamma \delta} \: G_{IM} \: G^{KL} \: G_{KL; \gamma} \: F^M_{\alpha \delta} 
\notag \\
\overline{R}_{\alpha \beta}^{\overline{h}} 
\:  =  \: & R_{\alpha \beta}^g \: - \: 
\frac12 \: G^{IJ} \: G_{IJ; \alpha \beta} \: + \: \frac14 \:
G^{IJ} \: G_{JK,\alpha} \: G^{KL} \: G_{LI,\beta} \: - \:
\frac12 \: g^{\gamma \delta} \: \: G_{IJ} \: F^I_{\alpha \gamma} \:
F^J_{\beta \delta} - \notag \\
& (n+N-1) \phi_{;\alpha \beta} + (n+N-1)
\phi_{,\alpha} \phi_{,\beta} - \notag \\
& \left(\triangle_g \phi + \frac12 \langle \nabla \ln \det G, 
\nabla \phi \rangle_g 
+ (n+N-1) |\nabla \phi|_g^2 \right) g_{\alpha \beta}. \notag
\end{align}

We now set 
\begin{equation} \label{A.34}
\phi = - \frac{1}{2(n+N-1)} \ln \det G,
\end{equation}
so that
\begin{equation} \label{A.35}
\frac12 \langle \nabla \ln \det G, 
\nabla \phi \rangle_g 
+ (n+N-1) |\nabla \phi|_g^2 = 0.
\end{equation}
We set the left-hand side of (\ref{A.33}) to be zero.  Multiplying the
first equation of (\ref{A.33}) by $G^{IJ}$, summing over $I$ and $J$,
and using the equation
\begin{equation} \label{A.36}
\triangle_g \ln \det G = 
g^{\alpha \beta} G^{IJ} G_{IJ; \alpha \beta} - 
g^{\alpha \beta} G^{IJ} G_{JK, \alpha} G^{KL} G_{LK, \beta},
\end{equation}
gives
\begin{align} \label{A.37}
0 = & \: - \frac12 \triangle_g \ln \det G - \frac14 |\nabla \ln \det G|_g^2
+ \frac14 g^{\alpha \gamma} g^{\beta \delta} G_{IJ} F^I_{\alpha \beta}
F^J_{\gamma \delta} - \\
& \: \frac12 (n+N-1) \langle \nabla \ln \det G, \nabla \phi
\rangle_g - N \triangle_g \phi \notag \\
= & \: \frac{1-n}{2(n+N-1)} \triangle_g \ln \det G
+ \frac14 g^{\alpha \gamma} g^{\beta \delta} G_{IJ} F^I_{\alpha \beta}
F^J_{\gamma \delta}. \notag  
\end{align}

Using the equation
\begin{equation} \label{A.38}
(\ln \det G)_{;\alpha \beta} = 
G^{IJ}  G_{IJ; \alpha \beta} - 
G^{IJ}  G_{JK, \alpha} G^{KL} G_{LI, \beta},
\end{equation}
the last equation of (\ref{A.33}) becomes
\begin{align} \label{A.39}
0 \:  =  \: & R_{\alpha \beta}^g \: - \: 
\frac12 \: (\ln \det G)_{; \alpha \beta} \: - \: \frac14 \:
G^{IJ} \: G_{JK,\alpha} \: G^{KL} \: G_{LI,\beta} \: - \:
\frac12 \: g^{\gamma \delta} \: \: G_{IJ} \: F^I_{\alpha \gamma} \:
F^J_{\beta \delta} - \\
& (n+N-1) \phi_{;\alpha \beta} + (n+N-1)
\phi_{,\alpha} \phi_{,\beta} - 
(\triangle_g \phi) g_{\alpha \beta} \notag  \\
=  \: & R_{\alpha \beta}^g \: - \: \frac14 \:
G^{IJ} \: G_{JK;\alpha} \: G^{KL} \: G_{LI;\beta} \: - \:
\frac12 \: g^{\gamma \delta} \: \: G_{IJ} \: F^I_{\alpha \gamma} \:
F^J_{\beta \delta} + \notag \\
& \frac{1}{4(n+N-1)} (\ln \det G)_{,\alpha} (\ln \det G)_{,\beta}
+ \frac{1}{2(n+N-1)} (\triangle_g \ln \det G)
g_{\alpha \beta}. \notag 
\end{align}
Using (\ref{A.37}), if $n > 1$ then
\begin{align} \label{A.40}
R_{\alpha \beta}^g = & \: 
\frac14 \:
G^{IJ} \: G_{JK,\alpha} \: G^{KL} \: G_{LI,\beta} \: + \:
\frac12 \: g^{\gamma \delta} \: \: G_{IJ} \: F^I_{\alpha \gamma} \:
F^J_{\beta \delta} - \\
& \: \frac{1}{4(n+N-1)} (\ln \det G)_{,\alpha} (\ln \det G)_{,\beta}
- \frac{1}{4(n-1)} g^{\mu \nu} g^{\gamma \delta} G_{IJ} F^I_{\mu \gamma}
F^J_{\nu \delta}
g_{\alpha \beta}. \notag
\end{align}

In terms of the decomposition (\ref{A.9}),
let $K_{ij}$ be the second fundamental form of the spatial
hypersurfaces. 
By performing a gauge transformation, we can assume that
$A_0 = 0$. Put 
\begin{align} \label{A.41}
H = & h^{ij}K_{ij},\\
K^0_{ij} = & K_{ij} - \frac{1}{n} H h_{ij}, \notag \\
|K|^2 = & K^{ij}K_{ij}, \notag \\
\left| K^0 \right|^2 = & K^{0,ij} K^0_{ij} = |K|^2 - \frac{1}{n} H^2. 
\notag \\ 
\left| \frac{\partial G}{\partial t} \right|^2 = &
\Tr \left( \left( G^{-1} G_{,0} \right)^2 \right) =
G^{IJ} \: G_{JK,0} \: G^{KL} \: G_{LI,0}, \notag \\
\left|\nabla G \right|^2_{G,h} = &
h^{ij} \Tr \left( G^{-1} G_{,i} G^{-1} G_{,j} \right) =
h^{ij} G^{IJ} \: G_{JK,i} \: G^{KL} \: G_{LI,j}, \notag \\
S_\alpha = & G^{- \frac12} G_{,\alpha} G^{- \frac12} - \frac{1}{N}
(\ln \det G)_{,\alpha} I_N, \notag \\
|S_0|^2 = & \Tr(S_0^2) = \left| \frac{\partial G}{\partial t} \right|^2 -
\frac{1}{N} \left( \frac{\partial \ln \det G}{\partial t} \right)^2, \notag \\
|\vec{S}|^2 = & h^{ij} \Tr(S_i S_j) =  h^{ij} \Tr \left( G^{-1} G_{,i}
G^{-1} G_{,j} \right) - 
\frac{1}{N} |\nabla \ln \det G|^2. \notag
\end{align}
Then from the Gauss-Codazzi equation,
\begin{equation} \label{A.42}
R^g_{00} - \frac12 R^g g_{00} = \frac{L^2}{2} 
\left( R^h - |K|^2 + H^2 \right) =
\frac{L^2}{2} 
\left( R^h - |K^0|^2 + \left( 1 - \frac{1}{n} \right) H^2 \right).   
\end{equation}
From (\ref{A.40}),
\begin{align} \label{A.42.5}
R^g_{00} - \frac12 R^g g_{00} = & \:
\frac18 \:
\left| \frac{\partial G}{\partial t} \right|^2
\: + \:
\frac18 \: L^2 |\nabla G|^2_{G,h}
 \:  -  \\ 
& \: \frac{1}{8(n+N-1)} \left( \frac{\partial \ln \det G}{\partial t} \right)^2
 -
 \frac{1}{8(n+N-1)} L^2
|\nabla \ln \det G|^2_h + \notag \\
& \:
\frac14 \: h^{ij} \: \: G_{IJ} \: F^I_{0i} \:
F^J_{0j}
 \: + \:
\frac18 \: L^2 h^{ik} h^{jl} \: \: G_{IJ} \: F^I_{ij} \:
F^J_{kl}. \notag
\end{align}
Hence we obtain the constraint equation
\begin{align} \label{A.43}
& L^2 \left( R^h - |K^0|^2 + \left( 1 - \frac{1}{n} \right) H^2 \right)
= \\
& \frac14 \: 
\left| \frac{\partial G}{\partial t} \right|^2
\: + \:
\frac14 \: L^2 |\nabla G|^2_{G,h} - \notag \\
& \: \frac{1}{4(n+N-1)} \left( 
\frac{\partial \ln \det G}{\partial t} \right)^2 -
\frac{1}{4(n+N-1)} L^2 |\nabla \ln \det G|^2_h 
 \: + \notag \\
& \:
\frac12 \: h^{ij} \: G_{IJ} \: F^I_{0i} \:
F^J_{0j}
 \: + \:
\frac14 \: L^2 h^{ik} h^{jl} \: G_{IJ} \: F^I_{ij} \:
F^J_{kl} \notag
\end{align}
or, equivalently,
\begin{align} \label{A.44}
& L^2 \left( R^h - |K^0|^2 + \left( 1 - \frac{1}{n} \right) H^2 \right)
= \\
& \: \frac14 \:
\left| S_0 \right|^2
\: + \:
\frac14 \: L^2 |\vec{S}|^2 + \notag \\
& \frac{n-1}{4N(n+N-1)} 
\left( \frac{\partial \ln \det G}{\partial t} \right)^2 
\: + \:
\frac{n-1}{4N(n+N-1)}  \: L^2 |\nabla \ln \det G|^2_{h} + \notag \\
& \frac12 \: h^{ij} \: \: G_{IJ} \: F^I_{0i} \:
F^J_{0j}
 \: + \:
\frac14 \: L^2 h^{ik} h^{jl} \: \: G_{IJ} \: F^I_{ij} \:
F^J_{kl}. \notag
\end{align}

From the spacetime splitting,
\begin{equation} \label{A.45}
\frac{\partial h_{ij}}{\partial t} = - 2 L K_{ij}
\end{equation}
and
\begin{equation} \label{A.46}
\frac{\partial K_{ij}}{\partial t} =  L H K_{ij} - 2 L
h^{kl} K_{ik} K_{lj} - L_{;ij} + L R^h_{ij} - L
R^g_{ij},
\end{equation}
where the covariant derivatives are now with respect to $h$.
Then using (\ref{A.40}), (\ref{A.41}) and (\ref{A.44}),
\begin{align} \label{A.47}
\frac{\partial H}{\partial t} = & \: L H^2 - \triangle_h L + L R^h -
L h^{ij} R^g_{ij} \\
= & \: L H^2 - \triangle_h L + L R^h -
\frac14 L |\nabla G|_{G,h}^2 +
\frac{1}{4(n+N-1)} L |\nabla \ln \det G|^2_h - \notag \\
& \frac{1}{2(n-1)} L^{-1} 
h^{ij} G_{IJ} F^I_{0i} F^J_{0j} - \frac{n-2}{4(n-1)} L
h^{ij} h^{kl} G_{IJ} F^I_{ik} F^J_{jl} \notag \\
= & \: - \triangle_h L + L |K^0|^2 + \frac{1}{n} L H^2 + 
\notag \\
& \frac14 L^{-1} \left| S_0 \right|^2 + 
\frac{n-1}{4N(n+N-1)} L^{-1} \left( 
\frac{\partial \ln \det G}{\partial t} \right)^2 +
\notag \\
& \: \frac{n-2}{2(n-1)} L^{-1} 
h^{ij} G_{IJ} F^I_{0i} F^J_{0j} + \frac{1}{4(n-1)} L
h^{ij} h^{kl} G_{IJ} F^I_{ik} F^J_{jl}. \notag
\end{align}

Now suppose that $H$ is spatially constant but time-dependent.
The maximum principle, when applied to (\ref{A.47}), gives
\begin{equation} \label{A.48}
L \le \frac{n}{H^2} \frac{\partial H}{\partial t}.
\end{equation}
We have the pointwise identity
\begin{equation} \label{A.49}
\frac{\partial}{\partial t} \dvol(Y, h) =
\frac12 h^{ij} \frac{\partial h_{ij}}{\partial t} \: \dvol_h =
- L H \: \dvol_h,
\end{equation}
so
\begin{equation} \label{A.50}
\frac{\partial}{\partial t} \left( (-H)^n \dvol(Y, h) \right) =
(-H)^{n+1} \left( L -  \frac{n}{H^2} \frac{\partial H}{\partial t}
\right) \: \dvol(Y,h).
\end{equation}
Assuming that $H$ is negative,
it follows from (\ref{A.48}) and (\ref{A.50}) 
that $(-H)^n \dvol(Y,h(t))$ is pointwise 
monotonically nonincreasing
in $t$, and hence $(-H)^n \vol(Y,h(t))$ is monotonically nonincreasing in $t$. 
Applying (\ref{A.47}) to (\ref{A.50}) gives
\begin{align} \label{A.51}
& \frac{d}{dt} \left( (-H)^n \vol(Y, h) \right) = 
- \: n (-H)^{n-1} \int_Y 
\left[
L |K^0|^2 + 
\frac14 L^{-1} \left| S_0 \right|^2 + \right. \\
& \frac{n-1}{4N(n+N-1)} L^{-1} \left( 
\frac{\partial \ln \det G}{\partial t} \right)^2 +
\frac{n-2}{2(n-1)} L^{-1} 
h^{ij} G_{IJ} F^I_{0i} F^J_{0j} + \notag \\
& \left. \frac{1}{4(n-1)} L
h^{ij} h^{kl} G_{IJ} F^I_{ik} F^J_{jl}
\right]\: \dvol(Y,h). \notag
\end{align}

We note in passing that (\ref{A.44}) gives an energy-type interpretation 
for the normalized volume, as
\begin{align} \label{A.52}
& (-H)^n \vol(Y, h) = \\
& \frac{n}{n-1} (-H)^{n-2} \int_Y \left[ - R^h + |K^0|^2 + 
\: \frac14 L^{-2} \:
\left| S_0 \right|^2
\: + \:
\frac14 \: |\vec{S}|^2 + \right. \notag \\
& \frac{n-1}{4(n+N-1)} L^{-2} 
\left( \frac{\partial \ln \det G}{\partial t} \right)^2 
\: + \:
\frac{n-1}{4(n+N-1)}  \: |\nabla \ln \det G|^2_{h} + \notag \\
& \left. \frac12 L^{-2} \: h^{ij} \: \: G_{IJ} \: F^I_{0i} \:
F^J_{0j}
 \: + \:
\frac14 \: h^{ik} h^{jl} \: \: G_{IJ} \: F^I_{ij} \:
F^J_{kl} \right] \dvol_Y. \notag
\end{align}

If $(-H)^n \vol(Y, h)$ is constant in $t$ and $n > 2$ then
from (\ref{A.51}), we must have
\begin{equation} \label{A.53}
0 = K^0 = S_0 = \frac{\partial \ln \det G}{\partial t} =  F^I_{ij} = F^I_{0i}.
\end{equation}
Then $K_{ij} = \frac{1}{n} H h_{ij}$, the connection $A^I_i$ is
spatially flat and time-independent, and $G$ is time-independent.
Equation (\ref{A.47}) now has the unique solution
\begin{equation} \label{A.54}
L = n H^{-2} \frac{dH}{dt}.
\end{equation}
From (\ref{A.45}),
\begin{equation} \label{A.55}
\frac{\partial h_{ij}}{\partial t} = - 2 H^{-1} 
\frac{dH}{dt} h_{ij},
\end{equation}
so
\begin{equation} \label{A.56}
h_{ij}(t) = H^{-2}(t) H^2(1) h_{ij}(1).
\end{equation}
From (\ref{A.37}), we have $\triangle_h \ln \det G = 0$, so
$\ln \det G$ is constant.
Then from the first equation in (\ref{A.33}), $G$ satisfies
\begin{equation} \label{A.57}
0 \: = \: h^{ij} G_{IJ;ij} \: - \:  h^{ij} G^{KL}
G_{IK,i} G_{LJ,j},
\end{equation}
where the covariant derivatives are now with respect to $h$.
Equations  (\ref{A.40}) and (\ref{A.46}) now give
\begin{align} \label{A.58}
R^h_{ij} \: = & \: - \: \frac{n-1}{n^2} H^2 h_{ij} + R^g_{ij} \\
= & \: - \: \frac{n-1}{n^2} H^2 h_{ij} + \frac14 G^{IJ} G_{JK,i}
G^{KL} G_{LI,j}. \notag
\end{align}

Conversely, given a static solution $(h_{ij}, G_{IJ})$ to the pair
\begin{align} \label{A.59}
0 = & \: h^{ij} G_{IJ;ij} - h^{ij} G^{KL} G_{IK,i} G_{LJ,j}, \\
R^h_{ij} = & \: - (n-1) h_{ij} +
\frac14 G^{IJ} G_{JK,i} G^{KL} G_{LI,j}, \notag
\end{align}
and an increasing positive function $\sigma(t)$, we get a solution 
\begin{align} \label{A.60}
L(t) \: = & \: \frac{d\sigma}{dt}, \\
h_{ij}(t) \: = & \: \sigma^{2}(t) h_{ij}, \notag  \\
K_{ij}(t) \: = & \: - \: \sigma(t) h_{ij} \notag \\
G_{IJ}(t) \: = & \: G_{IJ} \notag
\end{align}
with $H(t) \: = \: - \: \frac{n}{\sigma(t)}$.
Solutions to (\ref{A.59}) are discussed in
\cite[Proposition 4.80]{Lott (2010)}. 

\subsubsection{The case $n=2$} \label{subsubsectA.3.3}

If $n=2$ and $(-H)^2 \vol(Y,h)$ is constant in $t$ then from (\ref{A.51}),
\begin{equation} \label{A.61}
0 = K^0 = S_0 = \frac{\partial \ln \det G}{\partial t} =  F^I_{ij}
\end{equation}
and so
\begin{equation} \label{A.62}
L = n H^{-2} \frac{dH}{dt}.
\end{equation}
Equation (\ref{A.37}) becomes
\begin{equation} \label{A.63}
\frac{1}{N+1} \triangle_h \ln \det G \: = \: - \: 
L^{-2} h^{ij} G_{IJ} F^I_{0i} F^J_{0j}.
\end{equation}
Integrating over $Y$ gives $F^I_{0i} = 0$.
The discussion in (\ref{A.55})-(\ref{A.60}) is now valid.

\end{document}